\documentclass [a4paper,10pt]{amsart}
\usepackage{longtable}
\usepackage{hyperref}
\usepackage[T1]{fontenc}
\usepackage[utf8]{inputenc}
\usepackage[english]{babel}
\usepackage{textcomp}
\usepackage{dsfont}
\usepackage{latexsym}
\usepackage{amssymb}
\usepackage{amsthm}
\usepackage{amsmath}
\DeclareMathAlphabet{\mathpzc}{OT1}{pzc}{m}{en}
\usepackage{yfonts}
\usepackage{xfrac}
\usepackage{newlfont}
\usepackage{graphicx}
\usepackage{mathtools}
\usepackage{comment}
\usepackage{indentfirst}
\usepackage{braket}
\usepackage{mathrsfs}
\usepackage{xcolor}

\usepackage{etoolbox}

\usepackage{scalerel}[2014/03/10]
\usepackage[usestackEOL]{stackengine}
\newcommand{\dashint}{\,\ThisStyle{\ensurestackMath{%
			\stackinset{c}{.2\LMpt}{c}{.5\LMpt}{\SavedStyle-}{\SavedStyle\phantom{\int}}}%
		\setbox0=\hbox{$\SavedStyle\int\,$}\kern-\wd0}\int}

\DeclareMathOperator{\card}{Card}

\DeclareMathOperator{\supp}{Supp}

\DeclareMathOperator{\tr}{Tr}

\DeclareMathOperator{\Hol}{Hol}

\renewcommand{\Re}{\mathrm{Re}\,}
\renewcommand{\Im}{\mathrm{Im}\,}
\newcommand{\Supp}[1]{\supp\left( #1\right) }

\newcommand{\ee}{\mathrm{e}}

\newcommand{\vect}[1]{\mathbf{{#1}}}
\newcommand{\dd}{\mathrm{d}}

\DeclarePairedDelimiter{\abs}{\lvert}{\rvert}

\DeclarePairedDelimiter{\norm}{\lVert}{\rVert}

\let\originalleft\left
\let\originalright\right
\renewcommand{\left}{\mathopen{}\mathclose\bgroup\originalleft}
\renewcommand{\right}{\aftergroup\egroup\originalright}

\newcommand{\N}{\mathds{N}}
\newcommand{\Z}{\mathds{Z}}

\newcommand{\C}{\mathds{C}}

\newcommand{\R}{\mathds{R}}

\newcommand{\Ff}{\mathfrak{F}}

\newcommand{\nf}{\mathfrak{n}}

\newcommand{\zf}{\mathfrak{z}}

\newcommand{\Bc}{\mathcal{B}}
\newcommand{\Cc}{\mathcal{C}}

\newcommand{\Fc}{\mathcal{F}}

\newcommand{\Hc}{\mathcal{H}}
\newcommand{\Ic}{\mathcal{I}}

\newcommand{\Kc}{\mathcal{K}}
\newcommand{\Lc}{\mathcal{L}}
\newcommand{\cM}{\mathcal{M}}
\newcommand{\Nc}{\mathcal{N}}
\newcommand{\Oc}{\mathcal{O}}
\newcommand{\Pc}{\mathcal{P}}

\newcommand{\Rc}{\mathcal{R}}
\newcommand{\Sc}{\mathcal{S}}

\newcommand{\Lin}{\mathscr{L}}

\theoremstyle{definition}
\newtheorem{deff}{Definition}[section]

\newtheorem{oss}[deff]{Remark}

\theoremstyle{plain}
\newtheorem{teo}[deff]{Theorem}

\newtheorem{lem}[deff]{Lemma}

\newtheorem{prop}[deff]{Proposition}

\newtheorem{cor}[deff]{Corollary}

\title[Basov Spaces of Analytic Type]{Besov Spaces of Analytic Type:  Interpolation, Convolution, Fourier Multipliers, Inclusions}

\author[M.\ Calzi]{Mattia Calzi}
\address{Dipartimento di Matematica, Universit\`a degli Studi di Milano, Via C.\ Saldini 50, 20133 Milano, Italy}
\email{mattia.calzi@unimi.it}


\thanks{{\em Math Subject Classification.} {Primary:} 43A15; Secondary:  42B15, 46B70. }

\thanks{The author is a member of the Gruppo Nazionale per l'Analisi Matematica, la Probabilit\`a e le 	loro Applicazioni (GNAMPA) of the Istituto Nazionale di Alta Matematica	(INdAM) and is  partially supported by the 2020	GNAMPA grant {\em Fractional Laplacians and subLaplacians on Lie groups and trees}.}

\begin{document}
	
	\begin{abstract}
		We consider a family of Besov spaces of analytic type on the \v Silov boundary $\Nc$ of a homogeneous Siegel domain $D$, and study their properties in relation to convolution, Fourier multipliers, and complex interpolation. In addition, we study how these Besov spaces of analytic type can be compared with the `classical' Besov spaces $\Nc$.
	\end{abstract}
	
	\maketitle

\section{Introduction}

Let $E$ be a finite-dimensional complex vector space, $F$ a finite-dimensional real vector space, $\Omega$ a non-empty open convex cone in $F$ not containing affine lines, and $\Phi\colon E\times E\to F_\C$, where $F_\C$ denotes the complexification of $F$, an $\overline\Omega$-positive non-degenerate hermitian map. Then, the open convex set
\[
D\coloneqq \Set{(\zeta,z)\in E\times F_\C\colon \Im z-\Phi(\zeta)\in \Omega},
\] 
where $\Phi(\zeta)\coloneqq\Phi(\zeta,\zeta)$ for every $\zeta\in E$, is called a Siegel domain (of type II). The \v Silov boundary of $D$, that is, the smallest closed subset of $\partial D$ on which every bounded continuous function on $\overline D$ which is holomorphic on $D$ has the same supremum as on $D$, has a canonical group structure which acts affinely on $D$, and may be identified with the group $\Nc\coloneqq E\times F$, endowed with the product
\[
(\zeta,x)(\zeta',x')\coloneqq (\zeta+\zeta', x+x'+2 \Im \Phi(\zeta,\zeta')),
\]
for every $(\zeta,x),(\zeta',x')\in \Nc$, by means of the mapping $(\zeta,x)\mapsto (\zeta,x+i\Phi(\zeta))$.

When $E=\Set{0}$ and $F=\R$, $D$ is simply the upper half-plane in $\C$, and $\Nc=\R$. In this case, it is proved in~\cite{RicciTaibleson} that the boundary values of several (mixed-norm) weighted Bergman spaces on $D$ may be identified with suitable homogeneous Besov spaces on $\Nc$, namely $\Set{u\in \dot B^s_{p,q}(\R)\colon  \Supp{\Fc u}\subseteq \R_+}$, for $p,q\in ]0,\infty]$ and $s<0$.
In~\cite{BekolleBonamiGarrigosRicci}, the preceding results are extended to the case in which $E=\Set{0}$, but $\Omega$ is a general irreducible symmetric cone.\footnote{In other words, $\Omega$ is self-dual with respect to some scalar product on $F$, the group of linear automorphisms of $F$ preserving $\Omega$ acts transitively on $\Omega$, and $\Omega$ cannot be written as a product of two non-trivial convex cones.} Therein, suitable homogeneous Besov spaces on $\Nc=F$ are constructed and proved to be the boundary value spaces of several (mixed-norm) weighted Bergman spaces. 
These results have been recently extended in~\cite{CalziPeloso} to the case in which $E$ is arbitrary and $D$ is (affinely) homogeneous, that is, the group of (affine) biholomorphisms of $D$ acts transitively on $D$.

The purpose of this paper is to develop some aspects of the theory of the Besov spaces $B^{\vect s}_{p,q}(\Nc,\Omega)$ introduced in~\cite{CalziPeloso}. 

\medskip

In Section~\ref{sec:2}, we collect some basic definitions and results concerning homogeneous Siegel domains and the Besov spaces $B^{\vect s}_{p,q}(\Nc,\Omega)$ that will be necessary in the sequel. We shall mainly refer to~\cite{CalziPeloso} for a much more thorough exposition. For the sake of simplicity, we shall sometimes slightly modify some of the notation and terminology adopted in~\cite{CalziPeloso} and (apparently) allow more flexibility to the constructions developed therein. We shall also describe the group $\Nc$ in a slightly more axiomatic way, without reference to the corresponding Siegel domain $D$.

In Section~\ref{sec:3},  we develop some basic tools to deal with with the spaces $B^{\vect s}_{p,q}(\Nc,\Omega)$ for $p<1$, following the classical theory exposed in~\cite{Triebel}. Except for the sampling Theorem~\ref{prop:10} (and Lemma~\ref{lem:5}), the results of this section are trivial consequences of Young's inequality when $p\geqslant 1$. Some particular cases of these results have already been proved in~\cite{CalziPeloso}, but are not sufficient to prove some results of Sections~\ref{sec:6} and~\ref{sec:7}.

In Section~\ref{sec:4}, we deal with continuity results for convolution between the spaces $B^{\vect s}_{p,q}(\Nc,\Omega)$. These results are (somewhat technical, but) quite simple and natural when $\Nc$ is abelian. 
When $\Nc$ is not abelian, the fact that convolution `preserves regularity' (apparently expressed in a non-symmetric form) may appear peculiar. Nonetheless,  as~\cite[Theorem 4.26]{CalziPeloso} shows, the `regularity parameter' $\vect s$ actually depends only on the convolution with a distribution supported on the centre of $\Nc$: this explains why convolution may act on the Besov spaces $B^{\vect s}_{p,q}(\Nc,\Omega)$ as in the classical situation.
Since, nonetheless, the definition of the spaces $B^{\vect s}_{p,q}(\Nc,\Omega)$ is non-symmetric, some auxiliary `symmetrized' versions $\Bc^{\vect s}_{p,q}(\Nc,\Omega)$  of the $B^{\vect s}_{p,q}(\Nc,\Omega)$ appear naturally in this context. The main result is therefore expressed in terms of inclusions of the form $B*\Bc\subseteq B$, which imply analogous inclusions of the form $\Bc*\Bc\subseteq \Bc$, $B*B\subseteq B$, $B*B^*\subseteq B$, etc., since the spaces $\Bc$ can be naturally identified with quotients, as well as closed subspaces, of the spaces $B$.

In Section~\ref{sec:5}, we prove Mihlin-H\"ormander multiplier theorems for \emph{right} Fourier multiplies on the Besov spaces $B^{\vect s}_{p,q}(\Nc,\Omega)$.  Our main result is completely analogous to the classical one when $\Nc$ is abelian, and may appear peculiar when $\Nc$ is not abelian, since the relevant dimension for the regularity threshold is not that of $\Nc$, as one may expect, but rather that of its centre $F$. In order to explain this fact, one may observe that, roughly speaking, the Fourier transforms of the elements of $B^{\vect s}_{p,q}(\Nc,\Omega)$ do not have a vectorial nature `on the right' (since their kernels contain a \emph{fixed} hyperplane). It is therefore to be expected that \emph{left} Fourier multipliers behave in a quite different way.

In Section~\ref{sec:6}, we deal with complex interpolation. We first show, using classical techniques, an almost complete picture of how the Besov spaces $B^{\vect s}_{p,q}(\Nc,\Omega)$ interact with the classical complex interpolation functor (Proposition~\ref{prop:5}). We then introduce a modified complex interpolation functor, following~\cite{Triebel}, in order to deal with the case $\min(p,q)<1$, in which the $B^{\vect s}_{p,q}(\Nc,\Omega)$  are \emph{not} Banach spaces (but rather \emph{quasi}-Banach spaces), and with the case $\max(p,q)=\infty$, in which the usual complex interpolation functor behaves in a somewhat irregular way (Theorem~\ref{prop:9}). 
In contrast to the classical case, this modified interpolation functor does not operate on Banach pairs, but rather on `admissible triples' $(A_0,A_1,X)$, where  $A_0,A_1$ are quasi-Banach spaces which embed in the Hausdorff semi-complete \emph{locally convex} space $X$. Holomorphy is then defined with reference to $X$ and the usual arguments apply. 
We mention explicitly that  the complex interpolation spaces of a Banach pair $(A_0,A_1)$ need \emph{not} equal the complex interpolation spaces of the admissible triple $(A_0,A_1,\Sigma(A_0,A_1))$ (with the notation of~\cite{BerghLofstrom}), except when $A_0$ or $A_1$ is reflexive: this is unavoidable if one desires a nice treatment of the case $\max(p,q)=\infty$. In particular, in general $A_0\cap A_1$ is \emph{not} dense in the complex interpolation spaces of an admissible tripe $(A_0,A_1,X)$. Again, this is unavoidable if one wishes the spaces $B^{\vect s}_{p,q}(\Nc,\Omega)$ to interpolate nicely also when $\max(p,q)=\infty$, in full generality.

In Section~\ref{sec:7}, we introduce some more classical Besov spaces $B^{s}_{p,q}(\Nc,\Lc)$ on $\Nc$. When $\Nc$ is abelian, they reduce to the classical homogeneous Besov spaces on $\Nc$. The non-commutative analogues of these spaces are defined by means of the spectral calculus of a suitable positive Rockland operator $\Lc$ on $\Nc$. 
Some comments are in order. First of all, when $\Nc$ is stratified, we could have replaced $\Lc$ with a sub-Laplacian and referred to the general construction of homogeneous Besov spaces on a metric measure space developed in~\cite{GKKP}. Nonetheless, $\Nc$ need \emph{not} be stratified, so that it is more convenient to choose $\Lc$ as an operator of order $4$.
Secondly, it is likely that the so-defined Besov spaces do not depend on the chosen positive Rockland operator $\Lc$, since the same holds for the Sobolev spaces on graded groups developed in~\cite{FischerRuzhansky}.
Finally, we shall not embed \emph{a priori} our Besov spaces into the quotient of the space of tempered distributions on $\Nc$ by the space of polynomials, as one may expect, but rather in the strong dual of a suitable (dense) subspace $\Sc_\Lc(\Nc)$ of the space of Schwartz functions with all vanishing moments on $\Nc$, endowed with a stronger topology. \emph{A posteriori}, the resulting spaces are the same,  but the use of $\Sc_\Lc(\Nc)$ has some technical advantages that greatly simplify some of the arguments employed in Section~\ref{sec:8}. In addition to that, this choice parallels the analogous one made for the Besov spaces $B^{\vect s}_{p,q}(\Nc,\Omega)$.
For the sake of symmetry, we shall therefore also prove that also the spaces $B^{\vect s}_{p,q}(\Nc,\Omega)$ naturally embed in a suitable quotient of the space of tempered distributions (Proposition~\ref{prop:6}).

In Section~\ref{sec:8}, we compare the Besov spaces $B^{\vect s}_{p,q}(\Nc,\Omega)$ and $B^{s}_{p,q}(\Nc,\Lc)$. First we show that $B^{s}_{p,q}(\Nc,\Lc)$ is the closure (in a suitable weak topology, if $\max(p,q)=\infty$) of the union of an increasing sequence of closed subspaces $V_k$ which embed canonically as closed subspaces of $B^{\sum_j s_j}_{p,q}(\Nc,\Lc)$. Therefore,   it is reasonable to consider only `inclusions' of the form $B^{\vect s}_{p,q}(\Nc,\Omega)\to B^{\sum_j s_j}_{p,q}(\Nc,\Lc)$. 
When $\Omega=\R_+^*$,  the situation is clear since $B^{s}_{p,q}(\Nc,\Omega)$ is canonically isomorphic to a closed subspace of  $ B^{s}_{p,q}(\Nc,\Lc)$. 
When $\Omega\neq \R_+^*$, the situation is more complicated and the existence of a canonical embedding $B^{\vect s}_{p,q}(\Nc,\Omega)\to B^{\sum_j s_j}_{p,q}(\Nc,\Lc)$ turns out to be equivalent to a certain $\ell^q$-decoupling property $(D')^{\vect s}_{p,q}$ (Theorem~\ref{teo:1}). In particular, $(D')^{\vect s}_{p,q}$  may hold only if $\vect s \leqslant \vect 0$ (Proposition~\ref{prop:8}).

Property $(D')^{\vect s}_{p,q}$ is not new, as it is essentially related to a similar property which plays an important role in the determination of the boundary value spaces of several mixed-norm weighted Bergman spaces on $D$ (cf.~Proposition~\ref{prop:15},~\cite[Theorem 5.10]{CalziPeloso} and~\cite[Theorem 4.11 and Proposition 4.16]{BekolleBonamiGarrigosRicci}). In addition, when $\Omega$ is a Lorentz (or light) cone, then $(D')^{\vect s}_{p,q}$ is essentially related to the so-called `cone multiplier problem' (cf.~\cite{Wolff,LabaWolff,BekolleBonamiGarrigosRicci,GSS}). 

\subsection{Acknowledgements}

The author would like to thank professor A.\ Martini for posing the problem of investigating the relationships between the Besov spaces $B^{\vect s}_{p,q}(\Nc,\Omega)$ and the `classical' ones.  The author would also like to thank professor M.\ M.\ Peloso for helpful suggestions to improve the structure of the manuscript. 

\section{Preliminaries}\label{sec:2}

In this section we shall summarize several definitions and results on homogeneous cones, homogeneous Siegel domains, and the Besov spaces of analytic type $B^{\vect s}_{p,q}(\Nc,\Omega)$. We refer the reader to~\cite{CalziPeloso} for a more thorough exposition.

Throughout the paper, we shall fix (with the exception of Proposition~\ref{prop:16}):
\begin{itemize}
	\item a real hilbertian space $F$ of dimension $m>0$  and a complex hilbertian space $E$ of dimension $n$;
	
	\item a homogeneous cone $\Omega$ in $F$, that is, a non-empty open convex cone not containing any affine line on which the group $G(\Omega)\coloneqq \Set{t\in GL(F)\colon t \Omega=\Omega}$ acts transitively;
	
	\item a triangular\footnote{In other words, there is a basis of $F$ such that for every $t\in T_+$  the matrix associated with the mapping $x\mapsto t\cdot x$ is upper triangular. Equivalently, all the eigenvalues of the mapping $x\mapsto t\cdot x$ are real for every $t\in T_+$ (cf.~\cite{Vinberg2}).} Lie group $T_+$ which acts linearly and simply transitively on $\Omega$, and a homomorphism $\Delta\colon T_+\to (\R_+^*)^r$ which induces an isomorphism of $T_+/[T_+,T_+]$ onto $(\R_+^*)^r$, for some $r\in \N^*$;\footnote{A $T_+$ with the required properties always exists (cf.~\cite{Vinberg}). In addition, $r$ does not depend on the choice of $T_+$ and is called the rank of $\Omega$.}
	
	\item a non-degenerate $\overline\Omega$-positive hermitian map $\Phi\colon E\times E\to F_\C$, where $F_\C$ denotes the complexification of $F$, such that for every $t\in T_+$ there is $g\in GL(E)$ such that $t\cdot \Phi=\Phi\circ (g\times g)$.
\end{itemize}

Notice that $T_+$ may be equivalently characterized as a maximal connected triangular subgroup of $G(\Omega)$, and that any two such subgroups of $G(\Omega)$ are conjugate by an inner automorphism of $G(\Omega)$ (cf.~\cite{Vinberg,Vinberg2}). In particular, we may always assume that $T_+$ has the form considered in~\cite{CalziPeloso}.
In  particular, $T_+$ contains a subgroup acting by homotheties on $F$, and $T_+$ acts, by transposition, on the dual cone $\Omega'\coloneqq \Set{\lambda\in F'\colon \forall h\in \overline \Omega \setminus \Set{0}\:\: \left\langle \lambda,h\right\rangle>0}$ (cf.~also~\cite[Theorem 1]{Vinberg}). We denote the corresponding \emph{right} action by $\lambda\cdot t$, for $\lambda\in \Omega'$ and $t\in T_+$.

Then, $D\coloneqq \Set{(\zeta,z)\in E\times F_\C\colon \Im z-\Phi(\zeta)\in \Omega}$ is the homogeneous Siegel domain (of type II) associated to $E,F,\Omega$, and $\Phi$, where $\Phi(\zeta)\coloneqq \Phi(\zeta,\zeta)$ for every $\zeta\in E$. The \v Silov boundary $b D$ of $D$ can be canonically identified with $\Nc\coloneqq E\times F$ through the mapping $(\zeta,x)\mapsto (\zeta,x+i\Phi(\zeta))$. If we  endow $\Nc$ with the product
\[
(\zeta,x)(\zeta',x')\coloneqq (\zeta+\zeta',x+x'+2 \Im \Phi(\zeta,\zeta'))
\]
for every $(\zeta,x),(\zeta',x')\in \Nc$, then $\Nc$ acts holomorphically on $E\times F_\C$, $b D$, and  $D$, setting
\[
(\zeta,x)\cdot (\zeta',z')\coloneqq(\zeta+\zeta',x+i\Phi(\zeta)+z'+2 i\Phi(\zeta',\zeta))
\]
for every $(\zeta,x)\in \Nc$ and for every $(\zeta',z')\in E\times F_\C$. Then, $b D$ is the orbit of $(0,0)$, and $\Nc$ acts simply transitively on $b D$.
Note that $\Nc$ is a $2$-step nilpotent Lie group with centre $F$ and commutator subgroup equal to the vector subspace of $F$ generated by $\Phi(E)$. We endow $\Nc$ with the dilations given by $\rho \cdot (\zeta,x)\coloneqq (\rho^{1/2}\zeta,\rho x)$ for every $\rho>0$ and $(\zeta,x)\in \Nc$. With this choice, $\Nc$ need not be a stratified group, but has the usual dilations when it is abelian, that is, when $E=0$.

\subsection{An Intrinsic Perspective}

We shall now indicate a slightly more intrinsic description of the group $\Nc$ described above. For technical reasons, it is more convenient to describe its Lie algebra, instead.

\begin{prop}\label{prop:16}
	Let $\nf$ be a (finite-dimensional, real) $2$-step nilpotent Lie algebra with centre $\zf$. Then, $\nf$ is the Lie algebra of a group $\Nc$ satisfying the above conditions if and only if there are a complex structure $J$\footnote{That is, $J$ is an endomorphism of $\nf/\zf$ and $J^2=-I$.} on $\nf/\zf$ and a non-empty open convex cone $\Omega$ in $\zf$ not containing affine lines such that the following hold:
	\begin{itemize}
		\item for every $\lambda\in \Omega'$, the bilinear form $(X,Y)\mapsto \left\langle \lambda, [J X, Y]\right\rangle$, on $\nf/\zf$, is symmetric and positive on $\nf$;
		
		\item the group $G$  of the automorphisms of $\nf$ which preserve $\Omega$ and induce automorphisms of $\nf/\zf$ commuting with $J$ acts transitively on $\Omega$.\footnote{Note that every automorphism of (the Lie algebra) $\nf$ necessarily preserves $\zf$, so that it induces an automorphism of (the vector space) $\nf/\zf$.}
	\end{itemize}
\end{prop}

\begin{proof}
	It is clear that the Lie algebra of a group $\Nc$ as above satisfies the conditions of the statement. Then, take $\nf$ as above. Set $E\coloneqq\nf/\zf$, endowed with the structure of a complex vector space induced by $J$,  $F\coloneqq\zf$, 
	\[
	\Phi\colon E\times E\ni (\zeta,\zeta') \mapsto \frac 1 4 [i \zeta,\zeta']+\frac i 4 [\zeta,\zeta']\in F_\C,
	\]	
	and $\Nc\coloneqq E\times F$, endowed with the product
	\[
	(\zeta,x)(\zeta',x')\coloneqq (\zeta+\zeta', x+x'+2 \Im \Phi(\zeta,\zeta'))
	\]
	for every $(\zeta,x),(\zeta',x')\in \Nc$. Notice that $\Phi$ is well defined, since $[\nf,\nf]\subseteq \zf$ and $[\nf,\zf]=\Set{0}$. In addition, it is clear that $\left\langle \lambda_\C, \Phi\right\rangle$ is hermitian  for every $\lambda\in \Omega'$, where $\lambda_\C$ denotes the complexification of $\lambda$, that is, $\lambda\otimes I_\C$. Therefore, $\Phi$ is hermitian.	
	Furthermore, $\Im \Phi$ is clearly non-degenerate, so that also $\Phi$ is. In addition, $\Phi(E)\subseteq \overline \Omega$ since $\overline\Omega$ is the polar of $\Omega'$ and $\left\langle \lambda_\C, \Phi\right\rangle$ is positive for every $\lambda\in \Omega'$.
	
	Next, define $G'$ as the set of automorphisms of $\Nc$ of the form
	\[
	(\zeta,x)\mapsto (A \zeta, B \zeta+Cx)
	\]
	with $A\in GL_\C(E)$, $B\in \Lc_\R(E; F)$, $C\in GL_\R(F)$, $C\Omega=\Omega$,  and $C[\zeta,\zeta']=[A \zeta,A \zeta']$ for every $\zeta,\zeta'\in E$. In other words, $G'$ corresponds to $G$ under the canonical identification of $\nf$ with $\Nc$. Then, by assumption, $G'$ acts transitively on $\Omega$. 
	In addition, with the above notation, $A\times C$ is an automorphism of $\Nc$  preserving $\Omega$, and the set $G''$ of the $A\times C$ is still a group acting transitively on $\Omega$. Notice that   $G''$ is closed in $GL(E\times F)$, and that its identity component $G''_0$ acts transitively on $\Omega$ (cf.~\cite[\S 3]{Murakami}).
	
	Now, by~\cite[Proposition 14]{Vinberg}, the identity component $G_0(\Omega)$ of the group of linear automorphisms of $\Omega$ is the identity component of an algebraic subgroup $N$ of $GL(F)$ (namely, the normalizer of $G_0(\Omega)$ in $GL(F)$). Then,
	\[
	N'\coloneqq \Set{A\times C\colon A\in GL_\C(E), C\in N, \forall \zeta,\zeta'\in E\:\: C[\zeta,\zeta']=[A \zeta,A \zeta']}
	\]
	is an algebraic subgroup of $GL_\R(E\times F)$, and its identity component $N'_0$ is contained in $G''$, hence in $G''_0$. Since clearly $G''_0\subseteq N'$, this proves that $N'_0=G''_0$. Therefore, $G''_0$ is the semidirect product of a maximal connected triangular subgroup $T$ and of a maximal connected compact subgroup $K$ (cf.~\cite[Theorem 1]{Vinberg}). Since $K$ is then the stabilizer of a point in the convex region $E\times \Omega$, $T_+\coloneqq\Set{C\colon A\times C\in T}$ is a connected triangular subgroup of $GL(F)$ which necessarily acts transitively on $\Omega$. Since every maximal connected triangular subgroup of $\Set{C\in GL(F)\colon C\Omega=\Omega}$ acts simply transitively on $\Omega$ (cf.~\cite{Vinberg2,Vinberg}), this implies that $T_+$ acts simply transitively on $\Omega$.\footnote{Note that this, im particular, implies that the mapping $A\times C\to C$ is an isomorphism of $T$ onto $T_+$, so that $T_+$ naturally acts on $E\times F$. One may define a group homomorphism $T_+\ni t\mapsto g_t\in GL(E)$ so that $t\cdot \Phi=\Phi\circ (g_t\times g_t)$ for every $t\in T_+$.}
\end{proof}

Notice that, if $\Phi(E)$ generates $F$ as a vector space (which is equivalent to saying that $\Nc$ is a stratified group), then the homogeneity condition implies that $\Omega$ is the interior of the (closed) convex cone generated by $\Phi(E)$.
If, otherwise, $\Phi(E)$ is contained in a proper subspace of $F$, then $\Omega$ is not uniquely determined by the above conditions. 
For example, if $\Nc$ is the product of a Heisenberg group $H$ and $\R^{m-1}$, $m>1$, then the only requirement on $(\Omega,T_+)$ is that $(1,0,\dots,0)$ is an eigenvector of every element of $T_+$. For example, every Lorentz cone\footnote{That is, a cone isomorphic to the quadric cone $\Set{(x_1,x_2)\in \R\times \R^{m-1}\colon x_1>\abs{x_2}}$.} in $F$ having $\R_+(1,0,\dots,0)$ as a generator would do, with an appropriate choice of $T_+$.

\subsection{The Homogeneous Cone $\Omega$}

Here we collect some definitions and results concerning $\Omega$, $T_+$, and $\Delta$. Cf.~\cite[Chapter 2]{CalziPeloso} for a more thorough exposition.

Given $\vect s\in \C^r$, we define $\Delta^{\vect s}(t)\coloneqq \prod_{j=1}^r \Delta_j(t)^{s_j}$ for every $t\in T_+$, so that $\Delta^{\vect s}$ is a character of $T_+$ and each character of $T_+$ is of the form $\Delta^{\vect s}$ for some $\vect s\in \C^r$.  
We fix two base points $e_\Omega$ and $e_{\Omega'}$ in $\Omega$ and in its dual $\Omega'$. 
Then, we may transfer $\Delta^{\vect s}$ to functions $\Delta_{\Omega}^{\vect s}$ and $\Delta_{\Omega'}^{\vect s}$ on $\Omega$ and $\Omega'$, respectively, such that 
\[
\Delta_{\Omega}^{\vect s}(t\cdot e_\Omega)\coloneqq \Delta^{\vect s}(t) \qquad \text{and} \qquad\Delta_{\Omega'}^{\vect s}(e_{\Omega'}\cdot t)\coloneqq \Delta^{\vect s}(t)
\]
for every $t\in T_+$.

Further, there is $\vect d\in \R^r$ such that the measures 
\[
\nu_\Omega\coloneqq \Delta_\Omega^{\vect d}\cdot \Hc^m \qquad \text{and} \qquad \nu_{\Omega'}\coloneqq \Delta_{\Omega'}^{\vect d}\cdot \Hc^m ,
\]
where $\Hc^m$ denotes the $m$-dimensional Hausdorff measure, are $G(\Omega)$-invariant on $\Omega$ and $\Omega'$, respectively. 
In addition, for a suitable choice of $\Delta$, there are $\vect m,\vect {m'}\in \N^r$ such that $\vect d=-(\vect 1_r-\frac 1 2 \vect m-\frac 1 2 \vect{m'})$ and such that the measures $\Delta_\Omega^{\vect s}\cdot \nu_\Omega$ and $\Delta_{\Omega'}^{\vect s}\cdot \nu_{\Omega'}$ extend to Radon measures on $\overline \Omega$ and $\overline{\Omega'}$, respectively, if and only if  $\Re\vect s\in \frac 1 2 \vect m+(\R_+^*)^r$ and $\Re\vect s\in \frac 1 2 \vect{m'}+(\R_+^*)^r$, respectively.
If this is the case, then, denoting with $\Lc$ the Laplace transform,
\[
\Lc(\Delta_\Omega^{\vect s}\cdot \nu_\Omega)=\Gamma_\Omega(\vect s)\Delta_{\Omega'}^{-\vect s}\qquad \text{and} \qquad \Lc(\Delta_{\Omega'}^{\vect s}\cdot \nu_{\Omega'})=\Gamma_{\Omega'}(\vect s)\Delta_{\Omega}^{-\vect s}
\]
respectively, where
\[
\Gamma_\Omega(\vect s)=c \prod_{j=1}^r\Gamma(s_j-m_j/2) \qquad \text{and}\qquad \Gamma_{\Omega'}(\vect s)=c \prod_{j=1}^r\Gamma(s_j-m'_j/2)
\]
for a suitable constant $c>0$. 

In addition, there are unique holomorphic families of tempered distributions $(I^{\vect s}_\Omega)_{\vect s\in \C^r}$ and $(I^{\vect s}_{\Omega'})_{\vect s\in \C^r}$ such that
\[
I^{\vect s}_\Omega= \frac{1}{\Gamma_\Omega(\vect s)}\Delta_\Omega^{\vect s}\cdot \nu_\Omega \qquad \text{and } \qquad I^{\vect s}_{\Omega'}= \frac{1}{\Gamma_{\Omega'}(\vect s)}\Delta_{\Omega'}^{\vect s}\cdot \nu_{\Omega'}
\]
when  $\Re\vect s\in \frac 1 2 \vect m+(\R_+^*)^r$ and $\Re\vect s\in \frac 1 2 \vect{m'}+(\R_+^*)^r$, respectively. We call these distributions `Riemann--Liouville potentials', since they reduce to the classical Riemann--Liouville potentials when $\Nc=\R$. 

We shall also fix $\vect b\in \R^r$ in such a way that $\Delta^{-\vect b}(t)=\det_\R g=\abs{\det_\C g}^2$ for every $t\in T_+$ and for every $g\in GL(E)$ such that $t\cdot \Phi=\Phi\circ (g\times g)$. Then, $\Delta_{\Omega'}^{-\vect b}$ is polynomial on $\Omega'$ and we shall extend it on $F'_\C$ accordingly. In addition, $\Phi_*(\Hc^{2 n})= c I^{-\vect b}_\Omega$ for some constant $c>0$.

Finally, we endow $\Omega$ and $\Omega'$ with complete $T_+$-invariant Riemannian metrics (for example, the canonical $G(\Omega)$-invariant metrics described in~\cite[I.4]{FarautKoranyi}).  We denote by $d_\Omega$ and $d_{\Omega'}$ the corresponding $T_+$-invariant distances.\footnote{Even though these distances may differ from the ones employed in~\cite{CalziPeloso}, they will still be locally bi-Lipschitz equivalent in a uniform way, by invariance, so that the difference will be immaterial for our purposes.}

Given a metric space $(X,d)$, $\delta>0$ and $R>1$, we say that a family $(x_k)_{k\in K}$ of elements of $X$ is a $(\delta,R)$-lattice if the balls $B(x_j,\delta)$ are pairwise disjoint and the balls $B(x_j,R\delta)$ cover $X$. 
Notice that $(\delta,2)$-lattices always exist, since one may take any maximal family $(x_k)_{k\in K}$ of elements of $X$ such that $d(x_k,x_h)\geqslant 2\delta$ for every $h,k\in K$, $h\neq k$. In addition, every $(\delta,R)$-lattice is also a $(\delta', R')$-lattice for every $\delta'\in ]0,\delta]$ and for every $R'\in [R'\delta/\delta',+\infty[$.
If, in addition, there is a  positive Radon measure $\mu$ on $X$ such that $\mu(B(x,R))=\mu(B(x',R))\in \R_+^*$ for every $x,x'\in X$ and for every $R>0$ (which will always be the case in what follows), then for every $(\delta,R)$-lattice $(x_k)_{k\in K}$ on $X$ there is  $N\in \N$, depending only on $\delta$ and $R$, such that every ball $B(x_k,R\delta)$ meets at most $N$ of the balls $B(x_{k'},R\delta)$.

The preceding considerations apply to lattices on $\Omega$ and $\Omega'$, but also to the lattices on $\Nc$ associated to any left-invariant distance. The measures $\nu_\Omega$, $\nu_{\Omega'}$, and $\Hc^{2n+m}$, where $\Hc^{2n+m}$ denotes the $(2n+m)$-dimensional Hausdorff measure on $\Nc$ (which is a left an right Haar measure), clearly satisfy the conditions imposed on the measure $\mu$ above, by invariance.

\subsection{Fourier Analysis on $\Nc$}

We refer the reader to~\cite[Chapter 1]{CalziPeloso} for a more thorough exposition.

Define $W\coloneqq \Set{\lambda\in F'\colon \left\langle \lambda_\C, \Phi\right\rangle \text{ is degenerate}}$, and observe that $W$ is an algebraic variety which does not meet $\Omega'$. Then, for every $\lambda\in F'\setminus W$, the quotient $\Nc/\ker \lambda$ is isomorphic to a Heisenberg group (or to $\R$), so that the Stone--Von Neumann theorem shows that there is, up to unitary equivalence, a unique irreducible unitary representation $\pi_\lambda$ of $\Nc$  in some hilbertian space $H_\lambda$ such that $\pi_\lambda(0, i x)=\ee^{- i \left\langle\lambda,x\right\rangle }$ for every $x\in F$. It turns out that these representations are sufficient to write down the Plancherel formula. More precisely, there is a constant $c>0$   such that
\[
\norm{f}_{L^2(\Nc)}^2= c \int_{F'\setminus W} \norm{\pi_\lambda(f)}_{\Lin^2(H_\lambda)}^2\abs{\Delta_{\Omega'}^{-\vect b}(\lambda)}\,\dd \lambda
\]
for every $f\in L^2(\Nc)$, where $\Lin^2(H_\lambda)$ denotes the space of Hilbert--Schmidt operators on $H_\lambda$.\footnote{Recall that $\Delta_{\Omega'}^{-\vect b}$ is polynomial on $\Omega'$, so that it may be extended to the whole of $F'$ by analyticity.}

If $\lambda\in \Omega'$, we may describe $\pi_\lambda$ as follows. Define $H_\lambda\coloneqq \Hol(E)\cap L^2(\ee^{-2 \left\langle \lambda,\Phi(\,\cdot\,)\right\rangle}\cdot\Hc^{2 n})$, where $\Hol(E)$ denotes the space of holomorphic functions on $E$, and define
\[
[\pi_\lambda(\zeta,x) \varphi](\omega)\coloneqq \ee^{\left\langle \lambda_\C, -i x + 2 \Phi(\omega,\zeta)-\Phi(\zeta)\right\rangle} \varphi(\omega-\zeta)
\]
for every $(\zeta,x)\in \Nc$, for every $\varphi\in H_\lambda$, and for every $\omega\in E$. Then, one may prove that $\pi_\lambda$ is a continuous irreducible unitary representation of $\Nc$ in $H_\lambda$.
We denote by $P_{\lambda,0}$ the self-adjoint projector of $H_\lambda$ onto the space of constant functions. Then, 
\[
\tr(\pi_\lambda(\zeta,x)P_{\lambda})= \frac{\left\langle \pi_\lambda(\zeta,x) 1\vert 1\right\rangle}{\norm{1}_{H_\lambda}^2}= \ee^{-\left\langle \lambda_\C, \Phi(\zeta)+i x\right\rangle}
\]
for every $(\zeta,x)\in \Nc$ and for every $\lambda\in \Omega'$.

\subsection{Besov Spaces of Analytic Type}

We report below some definitions and basic results concerning the spaces $B^{\vect s}_{p,q}(\Nc,\Omega)$. We refer the reader to~\cite[Chapter 4]{CalziPeloso} for a more thorough exposition.

\begin{deff}
	For every compact subset $K$ of $\Omega'$, we define 
	\[
	\Sc_\Omega(\Nc,K)\coloneqq \Set{\varphi\in \Sc(\Nc)\colon \forall \lambda\in F'\setminus W\quad \pi_\lambda(\varphi)=\chi_K(\lambda) P_{\lambda,0}\pi_\lambda(\varphi)P_{\lambda,0} }
	\]
	and
	\[
	\Sc_{\Omega,L}(\Nc,K)\coloneqq \Sc(\Nc)*\Sc_\Omega(\Nc,K),
	\]
	endowed with the topology induced by $\Sc(\Nc)$. We define $\Sc_\Omega(\Nc)$ and $\Sc_{\Omega,L}(\Nc)$ as the inductive limits of the $\Sc_\Omega(\Nc,K)$ and the $\Sc_{\Omega,L}(\Nc,K)$, as $K$ runs through the set of compact subsets of $\Omega'$, ordered by inclusion.\footnote{The definition of $\Sc_\Omega(\Nc)$ and $\Sc_{\Omega,L}(\Nc)$ is slightly different, but still equivalent, to that employed in~\cite{CalziPeloso}.}
	
	We define $\Fc_\Nc\colon \Sc_\Omega(\Nc)\ni \varphi\mapsto (\lambda \mapsto \tr \pi_\lambda(\varphi))\in \C^{\Omega'}$.
\end{deff}

It can be proved that $\Fc_\Nc$ induces an isomorphism of $\Sc_\Omega(\Nc)$ onto $C^\infty_c(\Omega')$.
In particular, convolution is commutative and associative on $\Sc_\Omega(\Nc)$, and $\Fc_\Nc(\varphi*\psi)=\Fc_\Nc(\varphi)\Fc_\Nc(\psi)$ for every $\varphi,\psi\in \Sc_\Omega(\Nc)$. Further, if $\varphi\in \Sc_\Omega(\Nc)$, then $\varphi^*\coloneqq \overline{\varphi(\,\cdot\,^{-1})}$  belongs to $\Sc_\Omega(\Nc)$ and $\Fc_\Nc(\varphi^*)=\overline{\Fc_\Nc(\varphi)}$.

Notice that both $\Sc_\Omega(\Nc)$ and $\Sc_{\Omega,L}(\Nc)$ are complete Hausdorff bornological Montel spaces.

Further, it can be proved that convolution by $I^{\vect s}_\Omega$ induces automorphisms of $\Sc_\Omega(\Nc)$ and $\Sc_{\Omega,L}(\Nc)$ for every $\vect s\in \C^r$. In addition, 
\[
\pi_\lambda(\varphi* I^{\vect s}_\Omega)= \ee^{-\frac{\pi}{2} \sum_j s_j i}\Delta^{-\vect s}_{\Omega}(\lambda)\pi_\lambda(\varphi)
\]
for every $\varphi\in \Sc_\Omega(\Nc)$, for every $\lambda\in \Omega'$, and for every $\vect s\in \C^r$.

\begin{deff}
	We denote by $\Sc_\Omega'(\Nc)$ and $\Sc'_{\Omega,L}(\Nc)$ the strong duals of the conjugates of $\Sc_\Omega(\Nc)$ and $\Sc_{\Omega,L}(\Nc)$, respectively, and define
	\[
	\left\langle u\vert \varphi\right\rangle\coloneqq \overline {\left\langle u,\overline \varphi\right\rangle}
	\]
	for every $(u,\varphi)\in \Sc_\Omega'(\Nc)\times \Sc_\Omega(\Nc)$ and for every $(u,\varphi)\in \Sc'_{\Omega,L}(\Nc)\times \Sc_{\Omega,L}(\Nc)$, respectively.
	
	In addition, for every $u\in \Sc_\Omega'(\Nc)$, for every $u'\in \Sc'_{\Omega,L}(\Nc)$,  for every $\varphi,\varphi'\in \Sc_\Omega(\Nc)$, for every $\psi\in \Sc_{\Omega,L}(\Nc)$, and for every $\tau\in \Sc(\Nc)$, we define
	\[
	\left\langle \varphi*u*\varphi'\vert \tau\right\rangle\coloneqq \left\langle u\vert \varphi^**\tau*\varphi'^*\right\rangle \qquad \text{and} \qquad \left\langle u'*\psi^*\vert \tau\right\rangle\coloneqq \left\langle u'\vert \tau*\psi\right\rangle.
	\]
\end{deff}

Observe that, if $u,u'$, $\varphi,\varphi'$, and $\psi$ are as above, then $\varphi*u*\varphi', u'*\psi\in \Sc'(\Nc)\cap C^\infty(\Nc)$. In addition, if $u$ and $u'$ are induced by elements of $\Sc'(\Nc)$, then $\varphi*u*\varphi', u'*\psi$ agree with their usual definition.

Furthermore, both $\Sc'_\Omega(\Nc)$ and $\Sc'_{\Omega,L}(\Nc)$ are complete Hausdorff Montel spaces and can be identified with the projective limits of the quotients of $\Sc'(\Nc)$ by the polars of the conjugates of $\Sc_\Omega(\Nc,K)$ and $\Sc_{\Omega,L}(\Nc,K)$, respectively, as $K$ runs through the set of compact subsets of $\Omega'$, ordered by inclusion.

\begin{oss}\label{oss:1}
	Take $\psi_1,\psi_2,\psi_3,\psi_4\in \Sc_\Omega(\Nc)$ such that $\psi_1*\psi_2=\psi_3*\psi_4$. Then, $\psi_1*u*\psi_2=\psi_3*u*\psi_4$ for every $u\in \Sc'_\Omega(\Nc)$. In particular, $\psi_1*u*\psi_2=\psi_2*u*\psi_1$.
	
	It suffices to prove that $\psi_1^**\tau*\psi_2^*=\psi_3^**\tau*\psi_4^*$ for every $\tau\in \Sc(\Nc)$. Then, observe that
	\[
	\pi_\lambda(\psi_1^**\tau*\psi_2^*)=\overline{\Fc_\Nc(\psi_1*\psi_2)} P_{\lambda,0}\pi_\lambda(\tau)P_{\lambda,0}=\overline{\Fc_\Nc(\psi_3*\psi_4)} P_{\lambda,0}\pi_\lambda(\tau)P_{\lambda,0}=\pi_\lambda(\psi_3^**\tau*\psi_4^*)
	\]
	for every $\lambda\in F'\setminus W$, so that the assertion follows.
\end{oss}

\begin{deff}\label{def:2}
	Take $\vect s\in \R^r$ and $p,q\in]0,\infty]$. Take a $(\delta,R)$-lattice $(\lambda_k)_{k\in K}$ on $\Omega'$, with $\delta>0$ and $R>1$, and a bounded family $(\varphi_k)$ of positive elements of $C^\infty_c(\Omega')$ such that
	\[
	\sum_{k\in K} \varphi_k(\,\cdot\, t_k^{-1})\geqslant 1
	\]
	on $\Omega'$, where $t_k\in T_+$ is defined so that $\lambda_k=e_{\Omega'}\cdot t_k$. Define $\psi_k\coloneqq \Fc_\Nc^{-1}(\varphi_k(\,\cdot\, t_k^{-1}))$ for every $k\in K$. Then, we define $B^{\vect s}_{p,q}(\Nc,\Omega)$ (resp.\ $\mathring B^{\vect s}_{p,q}(\Nc,\Omega)$) as the space of $u\in \Sc_{\Omega,L}'(\Nc)$ such that
	\[
	(\Delta_{\Omega'}^{\vect s}(\lambda_k) (u*\psi_k))\in \ell^q(K;L^p(\Nc)) \qquad \text{(resp.\ $(\Delta_{\Omega'}^{\vect s}(\lambda_k) (u*\psi_k))\in \ell^q_0(K;L^p_0(\Nc))$),}
	\]
	endowed with the corresponding topology.\footnote{Here, $L^p_0(\Nc)$ equals $L^p(\Nc)$ if $p<\infty$, and $C_0(\Nc)$ otherwise. The space $\ell^q_0(K;L^p_0(\Nc))$ is defined analogously.}
\end{deff}

Notice that the definitions of $ B^{\vect s}_{p,q}(\Nc,\Omega)$ and $\mathring B^{\vect s}_{p,q}(\Nc,\Omega)$ do not depend on the choice of $(\lambda_k)$ and $(\varphi_k)$. In addition, $\mathring B^{\vect s}_{p,q}(\Nc,\Omega)$ is the closure of (the canonical image of) $\Sc_{\Omega,L}(\Nc)$ in $B^{\vect s}_{p,q}(\Nc,\Omega)$.

Notice that    $ B^{\vect s}_{p,q}(\Nc,\Omega)$ and $\mathring B^{\vect s}_{p,q}(\Nc,\Omega)$ are quasi-Banach spaces and embed continuously into $\Sc'_{\Omega,L}(\Nc)$. 
In addition, there are continuous inclusions $B^{\vect{s_1}}_{p_1,q_1}(\Nc,\Omega)\subseteq B^{\vect{s_2}}_{p_2,q_2}(\Nc,\Omega)$ whenever $p_1\leqslant p_2$, $q_1\leqslant q_2$, and $\vect{s_2}=\vect{s_1}+\big(\frac{1}{p_1}-\frac{1}{p_2}\big)(\vect b+\vect d)$.

\begin{deff}
	Take $\vect s\in \R^r$ and $p,q\in ]0,\infty]$, and take $(\lambda_k)$, $(\varphi_k)$, and $(\psi_k)$ as in Definition~\ref{def:2}. Assume that
	\[
	\sum_{k\in K} \varphi_k(\,\cdot\, t_k^{-1})^2=1
	\]
	on $\Omega'$. Then, we define a continuous sesquilinear form on $B^{\vect s}_{p,q}(\Nc,\Omega)\times B^{-\vect s-(1/p-1)_+(\vect b+\vect d)}_{p',q'}(\Nc,\Omega)$, where $p'\coloneqq \max(1,p)'$ and $q'\coloneqq \max(1,q)'$, by
	\[
	(u,u')\mapsto \sum_{k\in K}\left\langle u*\psi_k\vert u'*\psi_k\right\rangle.
	\]
	
	We define $\sigma^{\vect s}_{p,q}$ as the weak topology $\sigma(B^{\vect s}_{p,q}(\Nc,\Omega), \mathring B^{-\vect s-(1/p-1)_+(\vect b+\vect d)}_{p',q'}(\Nc,\Omega))$.
\end{deff}

Note that the definition of the above sesquilinear form  does not depend on the choice of $(\lambda_k)$, $(\varphi_k)$, and $(\psi_k)$ as above. In addition, it induces an antilinear isomorphism of $B^{-\vect s-(1/p-1)_+(\vect b+\vect d)}_{p',q'}(\Nc,\Omega)$ onto $\mathring B^{\vect s}_{p,q}(\Nc,\Omega)'$. Hence, $\sigma^{\vect s}_{p,q}$ is the usual weak dual topology when $p,q\geqslant 1$, and has analogous properties also when $\min(p,q)<1$. For example, the bounded subsets of $ B^{\vect s}_{p,q}(\Nc,\Omega)$ are relatively compact in the topology $\sigma^{\vect s}_{p,q}$, while $\Sc_{\Omega,L}(\Nc,\Omega)$ is dense in $B^{\vect s}_{p,q}(\Nc,\Omega)$ in the topology $\sigma^{\vect s}_{p,q}$.

If $(\psi_k)$ as in Definition~\ref{def:2} is chosen in such a way that $\sum_k \Fc_\Nc(\psi_k)=1$ on $\Omega'$, then 
\[
u=\sum_{k\in K} u* \psi_k
\]
for every $u\in \Sc_{\Omega,L}(\Nc)$ (finite sum),  for every $u\in \Sc'_{\Omega,L}(\Nc)$ (with convergence in $\Sc'_{\Omega,L}(\Nc)$),  for every $u\in \mathring B^{\vect s}_{p,q}(\Nc,\Omega)$ (with convergence in $ B^{\vect s}_{p,q}(\Nc,\Omega)$), and for every $u\in  B^{\vect s}_{p,q}(\Nc,\Omega)$ (with convergence in $ \sigma^{\vect s}_{p,q}$).

As we shall see later (cf.~Proposition~\ref{prop:6}), we may have defined $B^{\vect s}_{p,q}(\Nc,\Omega)$ requiring $u$ to belong to the dual of the conjugate of the closure of $\Sc_{\Omega,L}(\Nc)$ in $\Sc(\Nc)$ (which is a quotient of $\Sc'(\Nc)$). Nonetheless, our choice has some technical advantages, since, for example, the sum $\sum_{k} u_k*\psi_k$, with $(\psi_k)$ as above, converges in $\Sc'_{\Omega,L}(\Nc)$ for every $(u_k)\in \Sc'_{\Omega,L}(\Nc)^K$, which would not be the case in the other case. We shall make use of these remarks in the proof of Proposition~\ref{prop:14}.

\section{Sampling}\label{sec:3}

The results of this section are based on~\cite[Chapter 1]{Triebel} and can be extended with minor modifications replacing $\Nc$ with a general homogeneous group (with the exception of Corollary~\ref{cor:13}). The extension to more general Lie groups (or even metric measure spaces) requires more careful considerations (cf., for instance,~\cite{KP}).

Recall that we endow $\Nc$ with the dilations $t\cdot (\zeta,x)\mapsto (t^{1/2}\zeta,t x)$, for $t>0$ and $(\zeta,x)\in \Nc$. We denote by $Q$ the corresponding homogeneous dimension, that is, $n+m$.
We define a $(1/2)$-homogeneous left-invariant control distance $d$ as follows. Denote by $\nabla_E$ and $\nabla_F$ the left-invariant differential operators on $\Nc$ which induce the gradients on $E$ and $F$, respectively, at $(0,0)$.
Given $(\zeta,x),(\zeta',x')\in \Nc$ we define $d((\zeta,x),(\zeta',x'))$ as the greatest lower bound of the set of $\varepsilon>0$ such that there is an absolutely continuous curve $\gamma\colon [0,1]\to \Nc$ joining $(\zeta,x)$ and $(\zeta',x')$ such that, if $a$ and $b$ are the component functions of $\gamma'$ with respect to $\nabla_E$ and $\nabla_F$, respectively, then
\[
\abs{a}, \abs{b}^{1/2}\leqslant \varepsilon
\]
almost everywhere.

Then, $d((\zeta,x)(\zeta',x'),(\zeta,x)(\zeta'',x''))=d((\zeta',x'),(\zeta'',x''))$ and $d(t\cdot (\zeta,x),t \cdot (\zeta',x'))=t^{1/2} d((\zeta,x),(\zeta',x'))$ for every $(\zeta,x),(\zeta',x'),(\zeta'',x'')\in \Nc$ and for every $t>0$. 
In addition, if $f\in C^1(\Nc)$, then
\[
\abs{f(\zeta,x)-f(\zeta',x')}\leqslant \max(\delta,\delta^2) \max_{\overline B((\zeta,x),\delta)}(\abs{\nabla_E f}+\abs{\nabla_F f}) ,
\]
where $\delta=d((\zeta,x),(\zeta',x'))$.

We denote by $\norm{\,\cdot\,}$ the homogeneous norm\footnote{By a homogeneous norm we mean a positive proper function on $\Nc$ which is symmetric and  homogeneous of degree $1$.} $d((0,0),\,\cdot\,)^2$.

For every $p\in ]0,\infty[$, for every measurable function $f$ on $\Nc$, and for every $(\zeta,x)\in \Nc$, we define 
\[
(\cM_p f)(\zeta,x)\coloneqq \sup\limits_B \left( \dashint_B \abs{f(\zeta',x')}^p\,\dd (\zeta',x')\right)^{1/p},
\]
where $B$ runs through the set of balls containing $(\zeta,x)$. We define $\cM_\infty f$ as the function constantly equal to $\norm{f}_{L^\infty(\Nc)}$.

Then, the usual theory of maximal functions show that for every $p,q\in ]0,\infty]$ such that $p<q$ there is a constant $C_{p,q}>0$ such that
\[
\norm{\cM_p f}_{L^q(\Nc)}\leqslant C_{p,q}\norm{f}_{L^q(\Nc)}
\]
for every $f\in L^q(\Nc)$.

\begin{deff}
	Take $\varphi\in \Sc(\Nc)$. We denote by $\Oc(\varphi)$ the space of $u\in \Sc'(\Nc)$ such that $u=u*\varphi$. We shall identify the elements of $\Oc(\varphi)$ with elements of $C^\infty(\Nc)$.
	
	For every $p\in ]0,\infty]$, we define $\Oc^p(\varphi)\coloneqq\Oc(\varphi)\cap L^p(\Nc)$ and $\Oc^p_0(\varphi)\coloneqq\Oc(\varphi)\cap L^p_0(\Nc)$, endowed with the topology induced by $L^p(\Nc)$.
\end{deff}

Notice that, if $\Nc=F$, then $\Oc(\varphi)=\Set{u\in \Sc'(\Nc)\colon \Supp{\Fc_F u}\subseteq (\Fc_F\varphi)^{-1}(1)}$.

\begin{lem}\label{lem:5}
	Take $\varphi\in \Sc(\Nc)$, $p\in ]0,\infty]$, and a left-invariant differential operator $X$ on $\Nc$. Then, there is a constant $C>0$ such that
	\[
	\abs{(X u)(\zeta,x)}\leqslant C (1+ d((\zeta,x),(\zeta',x')))^{2 Q/p} (\cM_p u)(\zeta',x')
	\]
	for every $u\in \Oc(\varphi)$, and for every $(\zeta,x),(\zeta',x')\in \Nc$.
\end{lem}

The proof is based on that of~\cite[Theorem 1.3.1]{Triebel}.

\begin{proof}
	\textsc{Step I.} Let us first show that we may reduce to proving the assertion for $X=I$ (the identity). Indeed, for every $N>0$ there is $C_{1,N}>0$ such that
	\[
	\abs{(X\varphi)(\zeta,x)}\leqslant \frac{C_{1,N}}{(1+\norm{(\zeta,x)}^{1/2})^N}
	\]
	for every $(\zeta,x)\in \Nc$. Therefore, choosing $N>\frac{2 Q}{p}+2 Q$,
	\[
	\begin{split}
		\abs{(X u)(\zeta,x)}&\leqslant C_{1,N} \int_\Nc \frac{ \abs{u(\zeta',x')}}{(1+d((\zeta,x),(\zeta',x'))^N}\,\dd (\zeta',x')\\
			&\leqslant C_{1,N} C'_{1,N} (1+d((\zeta,x),(\zeta'',x'')))^{2 Q/p}\sup\limits_\Nc \frac{\abs{u}}{(1+d((\zeta'',x''),\,\cdot\,))^{2 Q/p} } 
	\end{split}
	\]
	for every $u\in \Oc(\varphi)$ and for every $(\zeta,x),(\zeta'',x'')\in \Nc$, since
	\[
	1+d((\zeta',x'),(\zeta'',x''))\leqslant (1+d((\zeta,x),(\zeta',x')))(1+d((\zeta,x),(\zeta'',x'')))
	\]
	for every $(\zeta',x')\in \Nc$, and where
	\[
	C'_{1,N} \coloneqq\int_\Nc \frac{1}{(1+\norm{(\zeta',x')}^{1/2})^{N-2Q/p}}\,\dd (\zeta',x')<\infty .
	\]

	\textsc{Step II.} We now prove the assertion for $X=I$. Observe that the assertion is trivial for $p=\infty$, so that we may assume that $p<\infty$.
	Observe that, for every $\varepsilon\in ]0,1[$,
	\[
	\begin{split}
	\abs{u(\zeta,x)}&\leqslant \inf\limits_{ B((\zeta,x),\varepsilon)} \abs{u}+\sup\limits_{B((\zeta,x),\varepsilon)} \abs{u- u(\zeta,x)}\\
		&\leqslant \left( \dashint_{B((\zeta,x),\varepsilon)} \abs{u}^p\,\dd \Hc^{2 n+m}\right)^{1/p}+\varepsilon \sup\limits_{B((\zeta,x),  \varepsilon)} (\abs{\nabla_E u}+\abs{\nabla_F u})
	\end{split}
	\]
	for every $u\in \Oc(\varphi)$ and for every $(\zeta,x)\in \Nc$.
	Now, observe that 
	\[
	\begin{split}
	\left( \dashint_{B((\zeta,x),\varepsilon)} \abs{u}^p\,\dd \Hc^{2 n+m}\right)^{1/p}&\leqslant \left( \frac{1+d((\zeta,x),(\zeta',x'))}{\varepsilon}  \right)^{2Q/p}   \left( \dashint_{B((\zeta',x'),1+d((\zeta,x),(\zeta',x')))} \abs{u}^p\,\dd \Hc^{2n+m}\right)^{1/p}\\
	&\leqslant  \left( \frac{1+d((\zeta,x),(\zeta',x'))}{\varepsilon}  \right)^{2Q/p}   (\cM_p u)(\zeta,x)
	\end{split} 
	\]
	for every $u\in \Oc(\varphi)$ and for every $(\zeta,x),(\zeta',x')\in \Nc$. In addition, \textsc{step I} shows that there is a constant $C_2>0$ such that
	\[
	\sup\limits_{(\zeta,x)\in\Nc} \frac{ \sup\limits_{B((\zeta,x),\varepsilon)}(\abs{\nabla_E u}+\abs{\nabla_F u})}{(1+d((\zeta,x),(\zeta',x')))^{2 Q/p}}\leqslant\sup\limits_{\Nc} \frac{\abs{\nabla_E u}+\abs{\nabla_F u}}{(1/2+d(\,\cdot\,,(\zeta',x')))^{2 Q/p}} \leqslant C_2 \sup\limits_{\Nc} \frac{ \abs{u}}{(1+d(\,\cdot\,,(\zeta',x')))^{2 Q/p}}
	\]
	for every $u\in \Oc(\varphi)$ and for every $(\zeta',x')\in \Nc$, provided that $\varepsilon\in ]0,1/2]$. Hence,
	\[
	\sup\limits_{\Nc} \frac{ \abs{u}}{(1+d(\,\cdot\,,(\zeta,x)))^{2 Q/p}}\leqslant \varepsilon^{-2 Q/p} (\cM_p u)(\zeta,x)+\varepsilon C_2 \sup\limits_{\Nc} \frac{ \abs{u}}{(1+d(\,\cdot\,,(\zeta,x)))^{2 Q/p}}
	\]
	for every $u\in \Oc(\varphi)$ and for every $(\zeta,x)\in \Nc$, provided that $\varepsilon\in ]0,1/2]$. 
	If we choose $\varepsilon\leqslant\min(1/2, 1/(2C_2))$, we then find
	\[
	\sup\limits_{\Nc} \frac{ \abs{u}}{(1+d(\,\cdot\,,(\zeta,x)))^{2 Q/p}}\leqslant 2\varepsilon^{-2 Q/p} (\cM_p u)(\zeta,x)
	\]
	for every $u\in \Oc(\varphi)$ and for every $(\zeta,x)\in \Nc$.
\end{proof}

\begin{teo}\label{prop:10}
	Take $\varphi\in \Sc(\Nc)$, $p\in ]0,\infty]$, $\delta_+>0$, and $R_0>1$. Then, there are $C,\delta_->0$ such that for every $(\delta,R)$-lattice $(\zeta_j,x_j)_{j\in J}$, with $\delta>0$ and $R\in ]1,R_0]$, and for every $u\in \Oc(\varphi)$,
	\[
	\frac{1}{C}\norm{u}_{L^p(\Nc)}\leqslant \delta^{2 Q/p}\norm*{ \max_{\overline B((\zeta_j,x_j),R\delta)} \abs{u}}_{\ell^p(J)}\leqslant C\norm{u}_{L^p(\Nc)}
	\]
	if $\delta\leqslant \delta_+$, and
	\[
	\frac{1}{C}\norm{u}_{L^p(\Nc)}\leqslant \delta^{2 Q/p}\norm*{ \min_{\overline B((\zeta_j,x_j),R\delta)} \abs{u}}_{\ell^p(J)}\leqslant C\norm{u}_{L^p(\Nc)}
	\]
	if $\delta\leqslant \delta_-$ and $u\in \Oc^p(\varphi)$.
\end{teo}

This result is based on~\cite[Proposition 1.3.3]{Triebel}.

In some situations, it is possible to remove the assumption that $u\in \Oc^p(\varphi)$ in the third inequality. For example, if $\Nc=f$, then we may find a sequence of Schwartz functions $(\psi_j)$ of the form $\psi_j=\psi_0((j+1)\,\cdot\,)$, $j\in\N$, such that $\psi_0(0)=1$ and $u \psi_j\in \Oc(\tau)$ for some $\tau\in \Sc(\Nc)$ (for example, so that $\Fc\tau=1$ on the convex envelope of $(\Fc_F\varphi)^{-1}(1)\cup \Supp{\Fc \psi_0}\cup \Set{0}$). Then, arguing by approximation, the assertion  follows in this case.

This type of arguments can be extended also to the non-abelian case when $\varphi\in \widetilde \Sc_\Omega(\Nc)$ (cf.~the arguments of~\cite[Section 3.1]{CalziPeloso}). We do not know if the assumption that $u\in \Oc^p(\varphi)$ in the third inequality can be removed in full generality.

\begin{proof}
	\textsc{Step I.} Observe that, clearly,
	\[
	\norm{u}_{L^p(\Nc)}\leqslant \norm*{ \Hc^{2 n+m}(B((\zeta_j,x_j),R\delta))^{1/p}\max_{\overline B((\zeta_j,x_j),R\delta)} \abs{u}}_{\ell^p(J)},
	\]
	so that the first inequality  follows from the homogeneity of $\Hc^{2 n+m}$ and $d$.
	
	In order to prove the second inequality, we may limit ourselves to considering $u\in \Oc^p(\varphi)$.
	Then, choose, for every $j\in J$, $(\zeta'_j,x'_j)\in \overline B((\zeta_j,x_j),R\delta)$, so that $\abs{u(\zeta'_j,x'_j)}=\max_{\overline B((\zeta_j,x_j),R\delta)} \abs{u}$.
	Notice that 
	\[
	\abs{u(\zeta'_j,x'_j)}\leqslant \abs{u(\zeta,x)}+\max(R\delta,(R\delta)^2) \max_{B((\zeta_j,x_j),R\delta)}(\abs{\nabla_E u}+\abs{\nabla_F u})
	\]
	for every $(\zeta,x)\in B((\zeta_j,x_j),R\delta)$ and for every $j\in J$. Now, Lemma~\ref{lem:5} implies that there is a constant $C_1>0$ such that
	\[
	\max_{B((\zeta_j,x_j),R\delta)}(\abs{\nabla_E u}+\abs{\nabla_F u})\leqslant C_1 \cM_{p/2} u
	\]
	on $B((\zeta_j,x_j),R\delta)$, for every $j\in J$ and for every $u\in \Oc^p(\varphi)$. Hence, assuming that $R_0\delta_+\geqslant 1$ to simplify the notation,
	\[
	\Hc^{2 n+m}(B((0,0),1))^{1/p} \delta^{2 Q/p}\norm*{ \max_{\overline B((\zeta_j,x_j),R\delta)} \abs{u}}_{\ell^p(J)}^{\min(1,p)}\leqslant \norm{u}_{L^p(\Nc)}^{\min(1,p)}+C_1^{\min(1,p)}(R_0\delta_+)^{2\min(1,p)}  \norm*{\cM_{p/2} u}_{L^p(\Nc)}^{\min(1,p)}
	\] 
	for every $u\in \Oc^p(\varphi)$. Since there is a constant $C_2>0$ such that $\norm{\cM_{p/2} u}_{L^p(\Nc)}\leqslant C_2 \norm{u}_{L^p(\Nc)}$ for every $u\in L^p(\Nc)$, the second inequality follows.
	
	\textsc{Step II.} Observe that, clearly,
	\[
	\norm*{ \Hc^{2 n+m}(B((\zeta_j,x_j),\delta))^{1/p}\min_{\overline B((\zeta_j,x_j),R\delta)} \abs{u}}_{\ell^p(J)}\leqslant \norm{u}_{L^p(\Nc)}
	\]
	so that the fourth inequality  follows from the homogeneity of $\Hc^{2 n+m}$ and $d$.
	
	Then, choose, for every $j\in J$, $(\zeta'_j,x'_j)\in \overline B((\zeta_j,x_j),R\delta)$, so that $\abs{u(\zeta'_j,x'_j)}=\min_{\overline B((\zeta_j,x_j),R\delta)} \abs{u}$.
	Notice that 
	\[
	\abs{u(\zeta,x)}\leqslant \abs{u(\zeta'_j,x'_j)} +\max(R\delta,(R\delta)^2) \max_{B((\zeta_j,x_j), R\delta)}(\abs{\nabla_E u}+\abs{\nabla_F u})
	\]
	for every $(\zeta,x)\in B((\zeta_j,x_j),R\delta)$ and for every $j\in J$.  Hence,  choosing $\delta_-\leqslant 1/R_0$, and defining $C_1$ as in \textsc{step I},
	\[
	\norm{u}_{L^p(\Nc)}^{\min(1,p)}\leqslant  \Hc^{2 n+m}(B((0,0),R_0))^{1/p} \delta^{2 Q/p}\norm*{ \min_{\overline B((\zeta_j,x_j),R\delta)} \abs{u}}_{\ell^p(J)}^{\min(1,p)} + (C_1 R_0\delta_-)^{\min(1,p)} \norm*{\cM_{p/2} u}_{L^p(\Nc)}^{\min(1,p)}
	\] 
	for every $u\in \Oc^p(\Nc)$.
	Hence, there is a constant $C_3>0$ such that
	\[
	\norm{u}_{L^p(\Nc)}\leqslant C_3\delta^{2 Q/p}\norm*{ \min_{\overline B((\zeta_j,x_j),R\delta)} \abs{u}}_{\ell^p(J)}+ C_3 \delta_-\norm*{u}_{L^p(\Nc)}
	\]
	for every $u\in \Oc^p(\varphi)$. If we choose $\delta_-\leqslant \min(R_0^{-1},(2C_3)^{-1})$, then the third inequality follows.
\end{proof}

\begin{cor}\label{cor:10}
	Take $\varphi\in \Sc(\Nc)$ and $p,q\in ]0,\infty]$ such that $p\leqslant q$. Then, $\Oc^p(\varphi)$ embeds continuously into $\Oc^q(\varphi)$, and $\Oc^p_0(\varphi)$ embeds continuously into $\Oc^q_0(\varphi)$.
\end{cor}

\begin{proof}
	Fix a $(\delta,2)$-lattice $(\zeta_j,x_j)_{j\in J}$ on $\Nc$, and observe that Theorem~\ref{prop:10} implies that there is a constant $C>0$ such that
	\[
	\frac{1}{C}\norm{u}_{L^p(\Nc)}\leqslant\norm*{ \max_{\overline B((\zeta_j,x_j),2\delta)} \abs{u}}_{\ell^p(J)}\leqslant C\norm{u}_{L^p(\Nc)}
	\]
	and
	\[
	\frac{1}{C}\norm{u}_{L^q(\Nc)}\leqslant \norm*{ \max_{\overline B((\zeta_j,x_j),2\delta)} \abs{u}}_{\ell^q(J)}\leqslant C\norm{u}_{L^q(\Nc)}
	\]
	for every $u\in \Oc(\varphi)$. Therefore,
	\[
	\norm{u}_{L^q(\Nc)}\leqslant C\norm*{ \max_{\overline B((\zeta_j,x_j),2\delta)} \abs{u}}_{\ell^q(J)}\leqslant C\norm*{ \max_{\overline B((\zeta_j,x_j),2\delta)} \abs{u}}_{\ell^p(J)}\leqslant C^2\norm{u}_{L^p(\Nc)}
	\]
	for every $u\in \Oc(\varphi)$, whence the first assertion. For what concerns the second assertion, it will suffice to prove that $\Oc^{\max(1,p)}(\varphi)\subseteq \Oc^\infty_0(\varphi)$ if $p<\infty$. However, this follows from Young's inequality, since $u=u*\varphi\in C_0(\Nc)$ for every $u\in \Oc^p(\varphi)$, as $\varphi\in L^{p'}_0(\Nc)$.
\end{proof}

\begin{cor}\label{cor:11}
	Take $\varphi_1,\varphi_2\in \Sc(\Nc)$ and $p_1,p_2,p_3\in ]0,\infty]$ such that $p_1,p_2\leqslant p_3$ and $\frac{1}{p_1'}+\frac{1}{p_2'}\leqslant \frac{1}{p_3'}$. Then, there is a constant $C>0$ such that
	\[
	\norm{u_1*u_2^*}_{L^{p_3}(\Nc)}\leqslant C\norm{u_1}_{L^{p_1}(\Nc)}\norm{u_2}_{L^{p_2}(\Nc)}
	\]	
	for every $u_1\in \Oc^{p_1}(\varphi_1)$ and for every $u_2\in \Oc^{p_2}(\varphi_2)$.
\end{cor}

This result is an immediate consequence of the following more general one, which will be very useful when dealing with complex interpolation.

\begin{cor}\label{cor:12}
	Take $\varphi_1,\varphi_2\in \Sc(\Nc)$, $\alpha_1,\alpha_2>0$, and $p_1,p_2,p_3\in ]0,\infty]$ such that $p_1,p_2\leqslant p_3$ and $\frac{1}{p_1'}+\frac{1}{p_2'}\leqslant \frac{1}{p_3'}$. Then, there is a constant $C>0$ such that
	\[
	\norm*{\abs{u_1}^{\alpha_1}*\abs{ u_2^*}^{\alpha_2}}_{L^{p_3}(\Nc)}\leqslant C\norm{\abs{u_1}^{\alpha_1}}_{L^{p_1}(\Nc)}\norm{\abs{u_2}^{\alpha_2}}_{L^{p_2}(\Nc)}
	\]	
	for every $u_1\in \Oc(\varphi_1)$ and for every $u_2\in \Oc(\varphi_2)$.
\end{cor}

\begin{proof}
	The assertion follows from Young's inequality and Corollary~\ref{cor:10} if $p_3\geqslant 1$, so that we may reduce to the case $p_3<1$. By Corollary~\ref{cor:10}, we may then reduce to the case $p_1=p_2=p_3\eqqcolon p$.
	Then, fix a $(\delta,2)$-lattice $(\zeta_j,x_j)_{j\in J}$ for some $\delta>0$ such that $\Hc^{2 n+m}(B((0,0),2\delta))\leqslant 1$, and observe that
	\[
	\begin{split}
	\int_\Nc \abs{(\abs{u_1}^{\alpha_1}*\abs{ u_2^*}^{\alpha_2})(\zeta,x)}^p\,\dd (\zeta,x) &\leqslant \int_\Nc  \left( \int_\Nc \abs{u_1(\zeta',x')}^{\alpha_1} \abs{ u_2((\zeta,x)^{-1}(\zeta',x'))}^{\alpha_2}\,\dd (\zeta',x') \right)^p \,\dd (\zeta,x)\\
		&\leqslant \int_\Nc  \sum_{j\in J} \max_{\overline B((\zeta_j,x_j),2 \delta)} \abs{u_1}^{\alpha_1 p} \max_{\overline B((\zeta,x)^{-1}(\zeta_j,x_j),2 \delta)} \abs{u_2}^{\alpha_2 p}\,\dd (\zeta,x)\\
		&= \sum_{j\in J} \max_{\overline B((\zeta_j,x_j),2 \delta)} \abs{u_1}^{\alpha_1 p}\int_\Nc  \max_{\overline B((\zeta,x),2 \delta)} \abs{u_2}^{\alpha_2 p}\,\dd (\zeta,x)\\
		&\leqslant \sum_{j\in J} \max_{\overline B((\zeta_j,x_j),2 \delta)} \abs{u_1}^{\alpha_1 p} \sum_{j'\in J} \max_{\overline B((\zeta_{j'},x_{j'}),4 \delta)} \abs{u_2}^{\alpha_2 p}.
	\end{split}
	\]
	Since $\norm{\abs{u}^\alpha}_{L^p(\Nc)}=\norm{u}_{L^{\alpha p}(\Nc)}^{\alpha}$ for every $\alpha>0$ and for every measurable function $u$, Theorem~\ref{prop:10} leads to the conclusion.
\end{proof}

\begin{cor}\label{cor:13}
	Take $\varphi_1,\varphi_2\in \Sc_\Omega(\Nc)$ and $p_1,p_2,p_3\in ]0,\infty]$ such that $p_1,p_2\leqslant p_3$ and $\frac{1}{p_1'}+\frac{1}{p_2'}\leqslant \frac{1}{p_3'}$. Then, there is a constant $C>0$ such that
	\[
	\norm{u_1*u_2}_{L^{p_3}(\Nc)}\leqslant C\norm{u_1}_{L^{p_1}(\Nc)}\norm{u_2}_{L^{p_2}(\Nc)}
	\]	
	for every $u_1\in \Oc^{p_1}(\varphi_1)$ and for every $u_2\in \Oc^{p_2}(\varphi_2)$.
\end{cor}

\begin{proof}
	\textsc{Step I.} We first prove the assertion in the case $p_1=p_2=p_3\eqqcolon p<1$ and $u_1\in \Sc_\Omega(\Nc)\cap \Oc(\varphi_1)$.
 	Observe that Corollary~\ref{cor:10} implies that $u_2\in L^1(\Nc)\cap L^2(\Nc)$, so that the results of Section~\ref{sec:2} imply that there are constants $c,c'>0$ such that
	\[
	u_2(0,x)= c\int_{\Omega'} \tr(\pi_\lambda(u_2) \pi_\lambda(0,x)^*)\Delta_{\Omega'}^{-\vect b}(\lambda)\,\dd \lambda=c'\Fc_F^{-1}(\lambda \mapsto\tr(\pi_\lambda(u)) \Delta_{\Omega'}^{-\vect b}(\lambda) )(x)
	\]
	and
	\[
	\begin{split}
	(u_1*u_2)(\zeta,x)&=c\int_{\Omega'} \tr(\pi_\lambda(u_2)) \ee^{-\left\langle \lambda_\C , \Phi(\zeta)-i x\right\rangle} (\Fc_\Nc u_1)(\lambda) \Delta_{\Omega'}^{-\vect b}(\lambda)\,\dd \lambda\\
		&=c'\Fc_F^{-1}(\lambda \mapsto\tr(\pi_\lambda(u_2)) \Delta_{\Omega'}^{-\vect b}(\lambda)  \ee^{-\left\langle \lambda , \Phi(\zeta)\right\rangle} (\Fc_\Nc u_1)(\lambda))(x)
	\end{split}
	\]
	for every $(\zeta,x)\in \Nc$. Therefore, Corollary~\ref{cor:12} implies that there is a constant $C_1>0$ such that
	\[
	\norm{(u_1*u_2)(\zeta,\,\cdot\,)}_{L^p(F)}\leqslant C_1 \norm{\Fc_F^{-1}(  \ee^{-\left\langle \,\cdot\, , \Phi(\zeta)\right\rangle} \Fc_\Nc u_1)}_{L^p(F)}\norm{u_2(0,\,\cdot\,)}_{L^p(F)} 
	\]
	and such that
	\[
	\norm{\Fc_F^{-1}(  \ee^{-\left\langle \,\cdot\, , \Phi(\zeta)\right\rangle} \Fc_\Nc u_1)}_{L^p(F)}\leqslant C_1\norm{ \Fc_F^{-1}(  \ee^{-\left\langle \,\cdot\, , \Phi(\zeta)\right\rangle} (\Fc_\Nc u_1) \Delta_{\Omega'}^{-\vect b})}_{L^p(F) }=\frac{C_1}{c'} \norm{u_1(\zeta,\,\cdot\,)}_{L^p(\Nc)}
	\]
	for every $\zeta\in E$,	so that
	\[
	\norm{u_1*u_2}_{L^p(F)}\leqslant \frac{C_1^2}{c'}  \norm{u_1}_{L^p(\Nc)}\norm{u_2(0,\,\cdot\,)}_{L^p(F)}
	\]
	for every $u_1\in \Sc_\Omega(\Nc)\cap \Oc(\varphi_1)$ and $u_2\in \Oc^p(\varphi_2)$.
	
	Now, observe that we may construct a family $(\zeta_k,x_k)_{k\in K}$ of elements of $\Nc$ which is maximal for the property that $d((\zeta_k,x_k),(\zeta_{k'},x_{k'}))\geqslant 2$ for every $k,k'\in K$, $k\neq k'$ and such that, defining $K'\coloneqq \Set{k\in K\colon \zeta_k=0}$, the family $(x_k)_{k\in K'}$ is a  family of elements of $F$ which is maximal for the property that $d((0,x_k),(0, x_{k'}))\geqslant 2$ for every $k,k'\in K'$, $k\neq k'$. Then, $(\zeta_k,x_k)_{k\in K}$ is a $(1,2)$-lattice on $\Nc$ and $(x_k)_{k\in K'}$ is a $(1,2)$-lattice on $F$, for the (translation-invariant) distance induced by the identification of $F$ with the subspace $\Set{0}\times F$ of $\Nc$. Then, Theorem~\ref{prop:10} implies that there is a constant $C_2>0$ such that
	\[
	\norm*{\max_{\overline B((\zeta_k,x_k),2)}\abs{u_2} }_{\ell^p(K)}\leqslant C_2 \norm{u_2}_{L^p(\Nc)}
	\]
	for every $u_2\in \Oc(\varphi_2)$, so that
	\[
	\norm{u_2(0,\,\cdot\,)}_{L^p(F)}\leqslant \Hc^m(B_F(0,2))^{1/p}\norm*{\max_{\overline B_F(x_k,2)}\abs{u_2(0,\,\cdot\,)} }_{\ell^p(K')}\leqslant C_2 \Hc^m(B_F(0,2))^{1/p }\norm{u_2}_{L^p(\Nc)}
	\]
	for every $u_2\in \Oc(\varphi_2)$.
	Therefore,
	\[
	\norm{u_1*u_2}_{L^p(F)}\leqslant \frac{C_1^2C_2  \Hc^m(B_F(0,2))^{1/p }}{c'}  \norm{u_1}_{L^p(\Nc)}\norm{u_2}_{L^p(\Nc)}
	\]
	for every $u_1\in \Sc_\Omega(\Nc)\cap \Oc(\varphi_1)$ and for every $u_2\in \Oc^p(\varphi_2)$.
	
	\textsc{Step II.} We now consider the general case. The assertion follows from Young's inequality and Corollary~\ref{cor:10} if $p_3\geqslant 1$, so that we may reduce to the case $p_3<1$. By Corollary~\ref{cor:10}, we may then reduce to the case $p_1=p_2=p_3\eqqcolon p$.
	Then, combining \textsc{step I} with Corollary~\ref{cor:11}, we see that there is a constant $C_3>0$ such that
	\[
	\norm{u_1*u_2}_{L^p(\Nc)}=\norm{(u_1*\varphi_1)*u_2}_{L^p(\Nc)}=\norm{u_1*(u_2^* * \varphi_1^*)^*}_{L^p(\Nc)}\leqslant C_3 \norm{u_1}_{L^p(\Nc)}\norm{u_2^* * \varphi_1^*}_{L^p(\Nc)}
	\]
	and such that
	\[
	\norm{u_2^* * \varphi_1^*}_{L^p(\Nc)}=\norm{ \varphi_1* u_2}_{L^p(\Nc)}\leqslant C_3\norm{\varphi_1}_{L^p(\Nc)}\norm{u_2}_{L^p(\Nc)}
	\]
	for every $u_1\in \Oc^p(\varphi_1)$ and for every $u_2\in \Oc^p(\varphi_2)$. The assertion follows.
\end{proof}

\section{Convolution}\label{sec:4}

In this section we deal with convolution between the spaces $B^{\vect s}_{p,q}(\Nc,\Omega)$. Instead of proving that convolution induces continuous bilinear mappings on the product  $B^{\vect {s_1}}_{p_1,q_1}(\Nc,\Omega)\times B^{\vect{s_2}}_{p_2,q_2}(\Nc,\Omega)$ for suitable values of $p_1,q_1,p_2,q_2$ and $\vect{s_1},\vect{s_2}$, we shall introduce some `symmetrized' spaces $\Bc^{\vect s}_{p,q}(\Nc,\Omega)$ and consider convolution on the product$B^{\vect {s_1}}_{p_1,q_1}(\Nc,\Omega)\times \Bc^{\vect{s_2}}_{p_2,q_2}(\Nc,\Omega)$. On the one hand, the spaces $\Bc^{\vect s}_{p,q}(\Nc,\Omega)$ appear naturally in this kind of considerations. On the other hand, dealing with $B^{\vect {s_1}}_{p_1,q_1}(\Nc,\Omega)\times \Bc^{\vect{s_2}}_{p_2,q_2}(\Nc,\Omega)$ will allow us to get the best results we could have obtained arguing directly for $B^{\vect {s_1}}_{p_1,q_1}(\Nc,\Omega)\times B^{\vect{s_2}}_{p_2,q_2}(\Nc,\Omega)$.

\begin{deff}
	For every $\vect s\in \R^r$ and for every $p,q\in ]0,\infty]$, we define $\Bc^{\vect s}_{p,q}(\Nc,\Omega)$ and $\mathring \Bc^{\vect s}_{p,q}(\Nc,\Omega)$ as the spaces of $u\in \Sc'_\Omega(\Nc)$ such that
	\[
	(\Delta_{\Omega'}^{\vect s}(\lambda_k) \psi_k*u*\psi_k)\in \ell^q(K; L^p(\Nc)) \qquad \text{and}\qquad(\Delta_{\Omega'}^{\vect s}(\lambda_k) \psi_k*u*\psi_k)\in \ell^q_0(K; L^p_0(\Nc)),
	\]
	respectively, where $(\lambda_k)$ and $(\psi_k)$ are as in Definition~\ref{def:2}.
\end{deff}

Arguing as in the proof of~\cite[Lemma 4.14]{CalziPeloso}, one may prove that the definitions of $\Bc^{\vect s}_{p,q}(\Nc,\Omega)$ and $\mathring \Bc^{\vect s}_{p,q}(\Nc,\Omega)$ do not depend on the choice of $(\lambda_k)$ and $(\psi_k)$.
Because of Remark~\ref{oss:1} (and the preceding observations), one may replace $\psi_k*u*\psi_k$ with $\psi'_k*u*\psi_k$ (or $\psi_k*u*\psi'_k)$ in the above definition, where $\psi'_k\in \Sc_\Omega(\Nc)$ and $\psi'_k*\psi_k=\psi_k$, without altering the spaces $\Bc^{\vect s}_{p,q}(\Nc,\Omega)$ and $\mathring \Bc^{\vect s}_{p,q}(\Nc,\Omega)$.

The following result shows some of the relationships between $\Bc^{\vect s}_{p,q}(\Nc,\Omega)$ and $B^{\vect s}_{p,q}(\Nc,\Omega)$.

\begin{prop}\label{prop:14}
	Take $\vect s\in \R^r$ and $p,q\in]0,\infty]$. Let $V_0$ (resp.\ $V$) be the closure of $\Sc_\Omega(\Nc)$ in $B^{\vect s}_{p,q}(\Nc,\Omega)$ (resp.\ in the topology $\sigma^{\vect s}_{p,q}$). Then, the canonical mapping $\Sc'_{\Omega,L}(\Nc)\to \Sc_\Omega'(\Nc)$ induces an isomorphism of $V_0$ (resp.\ $V$) onto $\mathring \Bc^{\vect s}_{p,q}(\Nc,\Omega)$ (resp.\ $\Bc^{\vect s}_{p,q}(\Nc,\Omega)$).
	
	In addition, the canonical mapping $\Sc'_{\Omega,L}(\Nc)\to \Sc_\Omega'(\Nc)$ induces a strict morphism of $\mathring B^{\vect s}_{p,q}(\Nc,\Omega)$ (resp.\ $B^{\vect s}_{p,q}(\Nc,\Omega)$) onto $\mathring \Bc^{\vect s}_{p,q}(\Nc,\Omega)$ (resp.\ $\Bc^{\vect s}_{p,q}(\Nc,\Omega)$).
\end{prop}

\begin{proof}
	\textsc{Step I.} Take $u\in V_0$ (resp.\ $u\in V$) and $(\lambda_k)$, $(\psi_k)$ as in Definition~\ref{def:2}. Let us prove that $(u*\psi_k)*\psi_k=\psi_k*(u*\psi_k)$ for every $k\in K$. Let $\Ff$ be a filter on $\Sc_\Omega(\Nc)$ which converges to $u$ in  $B^{\vect s}_{p,q}(\Nc,\Omega)$ (resp.\ in the topology $\sigma^{\vect s}_{p,q}$). Then, $\Ff*\psi_k$ converges to $u*\psi_k$ in $\Sc'(\Nc)$ (resp.\ in the weak topology of $\Sc'(\Nc)$).
	Consequently, $\psi_k*(\Ff*\psi_k)$  and $(\Ff*\psi_k)*\psi_k$ converge to $\psi_k*(u*\psi_k)$ and $(u*\psi_k)*\psi_k$ in $\Sc'(\Nc)$ (resp.\ in the weak topology of $\Sc'(\Nc)$).
	Therefore, $(u*\psi_k)*\psi_k=\psi_k*(u*\psi_k)$ for every $k\in K$, since the equality is clear if $u\in \Sc_\Omega(\Nc)$. In particular, $\psi_k*(u*\psi_k)\in L^p_0(\Nc)$ (resp.\ $\psi_k*(u*\psi_k)\in L^p(\Nc)$) for every $k\in K$. It is then clear that the canonical image of $u$ in $\Sc'_\Omega(\Nc)$ belongs to $\mathring \Bc^{\vect s}_{p,q}(\Nc,\Omega)$ (resp.\ $\Bc^{\vect s}_{p,q}(\Nc,\Omega)$), and that the so-induced mapping is an isomorphism onto its image.
	
	Conversely, take $u\in \mathring \Bc^{\vect s}_{p,q}(\Nc,\Omega)$ (resp.\ $u\in\Bc^{\vect s}_{p,q}(\Nc,\Omega)$). Take $(\lambda_k)$, $(\psi_k)$ as above, and assume further that $\sum_{k\in K} \Fc_\Nc(\psi_k)^2=1$ on $\Omega'$. Then, define
	\[
	u'\coloneqq \sum_{k\in K} \psi_k*u*\psi_k,
	\]
	and observe that the sum clearly converges in $\Sc'_{\Omega,L}(\Nc)$. More precisely, the sum converges in $B^{\vect s}_{p,q}(\Nc,\Omega)$ (resp.\ in the topology $\sigma^{\vect s}_{p,q}$), so that $u'\in V_0$ (resp.\ $u'\in V$). Then, $u$ is the canonical image of $u'$ in $\Sc'_\Omega(\Nc)$, so that the first assertion follows.
	
	\textsc{Step II.} By \textsc{step I}, it will suffice to observe that Corollary~\ref{cor:13} and a homogeneity argument imply that there is a constant $C>0$ such that
	\[
	\norm{\psi_k*(u*\psi_k)}_{L^p(\Nc)}\leqslant C\norm{u*\psi_k}_{L^p(\Nc)}
	\]
	for every $k\in K$ and for every $u\in B^{\vect s}_{p,q}(\Nc,\Omega)$, where $(\psi_k)$ is as above. 
\end{proof}

\begin{teo}
	Take $\vect{s_1},\vect{s_2}\in \R^r$ and $p_1,p_2,p_3,q_1,q_2,q_2\in ]0,\infty]$ such that 
	\[
	p_1,p_2\leqslant p_3, \qquad \frac{1}{p_1'}+\frac{1}{p_2'}\leqslant\frac{1}{p_3'}, \qquad \frac{1}{q_1}+\frac{1}{q_2}\geqslant \frac{1}{q_3}, \qquad \vect{s_3}=\vect{s_1}+\vect{s_2}+\left(\frac{1}{p_1}+\frac{1}{p_2}-1-\frac{1}{p_3}\right)(\vect b+\vect d).
	\]
	Then, the mapping
	\[
	\Sc_{\Omega,L}(\Nc)\times \Sc_\Omega(\Nc)\ni (\varphi_1,\varphi_2)\mapsto \varphi_1*\varphi_2\in \Sc_{\Omega,L}(\Nc)
	\]
	induces a uniquely determined  continuous bilinear  mapping
	\begin{align*}
		B^{\vect{s_1}}_{p_1,q_1}(\Nc,\Omega)\times  \Bc^{\vect{s_2}}_{p_2,q_2}(\Nc,\Omega)&\to  B^{\vect{s_3}}_{p_3,q_3}(\Nc,\Omega)
	\end{align*}
	such that
	\[
	(u_1*u_2)*\psi=(u_1*\psi')*(\psi'*u_2*\psi)
	\]
	for every $\psi,\psi'\in \Sc_\Omega(\Nc)$ such that $\psi=\psi*\psi'$, for every $u_1\in B^{\vect{s_1}}_{p_1,q_1}(\Nc,\Omega)$, and for every $u_2\in \Bc^{\vect{s_2}}_{p_2,q_2}(\Nc,\Omega)$.
\end{teo}

Before we pass to the proof, let us draw some consequences of this result.
First of all, by means of Proposition~\ref{prop:14}, we see that convolution induces a  continuous bilinear mapping 
\[
B^{\vect{s_1}}_{p_1,q_1}(\Nc,\Omega)\times B^{\vect{s_2}}_{p_2,q_2}(\Nc,\Omega)\to B^{\vect{s_3}}_{p_3,q_3}(\Nc,\Omega).
\]

Analogously, again by means of Proposition~\ref{prop:14}, we see that convolution induces a  continuous \emph{commutative} bilinear mapping 
\[
\Bc^{\vect{s_1}}_{p_1,q_1}(\Nc,\Omega)\times \Bc^{\vect{s_2}}_{p_2,q_2}(\Nc,\Omega)\to \Bc^{\vect{s_3}}_{p_3,q_3}(\Nc,\Omega).
\]

Furthermore, since the spaces $\Bc$ are invariant under the mapping $u\mapsto u^*$,  we also find a continuous sesquilinear mapping
\[
B^{\vect{s_1}}_{p_1,q_1}(\Nc,\Omega)\times B^{\vect{s_2}}_{p_2,q_2}(\Nc,\Omega)\ni (u_1,u_2)\mapsto u_1* u_2^*\in B^{\vect{s_3}}_{p_3,q_3}(\Nc,\Omega).
\]

\begin{proof}
	Take a $(\delta,R)$-lattice $(\lambda_k)_{k\in K}$ on $\Omega'$ for some $\delta>0$ and some $R>1$, and define $t_k\in T_+$ so that $\lambda_k=e_{\Omega'}\cdot t_k$ for every $k\in K$. Fix two bounded families $(\varphi_k)$ and $(\varphi'_k)$ of positive elements of $C^\infty_c(\Omega')$ such that
	\[
	\sum_k \varphi_k(\,\cdot\, t_k^{-1})=\sum_{k }\varphi_k'^2(\,\cdot\, t_k^{-1})=1
	\]
	on $\Omega'$, and define $\psi_k\coloneqq \Fc_\Nc^{-1}(\varphi_k(\,\cdot\, t_k^{-1}))$ and $\psi'_k\coloneqq \Fc_\Nc^{-1}(\varphi'_k(\,\cdot\, t_k^{-1}))$ for every $k\in K$. 
	Take $u_1\in B^{\vect{s_1}}_{p_1,q_1}(\Nc,\Omega)$ and $u_2\in \Bc^{\vect{s_2}}_{p_2,q_2}(\Nc,\Omega)$, and define, for every two finite subsets $H,H'$ of $K$,
	\[
	v_{H,H'}\coloneqq \sum_{k\in H}\sum_{k'\in H'} (u_1*\psi_k)*(\psi'_{k'}*u_2*\psi'_{k'}).
	\]
	Notice that $u_1*\psi_k\in L^{\max(1,p_1)}(\Nc)$ and $\psi'_{k'}*u_2*\psi'_{k'}\in L^{\max(1,p_2)}(\Nc)$ (cf.~Corollary~\ref{cor:10}), so that $v_{H,H'}$ is well defined by Young's inequality.
	Observe that, if $\eta\in \Sc_{\Omega,L}(\Nc)$, then
	\[
	\left\langle v_{H,H'}\vert \eta\right\rangle=\sum_{k\in H}\sum_{k'\in H'} \left\langle u_1*\psi_k\vert \eta*(\psi'_{k'}*u_2^** \psi'_{k'})\right\rangle=\sum_{k\in H}\sum_{k'\in H'} \left\langle u_1\vert \eta*(\psi'_{k'}*u_2^** \psi'_{k'})*\psi_k\right\rangle
	\]
	converges to the (finite) sum
	\[
	\sum_{k,k'\in K} \left\langle u_1\vert \eta*(\psi'_{k'}*u_2^** \psi'_{k'})*\psi_k\right\rangle
	\]
	as $H,H'$ run through the set of finite subsets of $K$, endowed with the section filter associated with $\subseteq $.
	Furthermore, if $\eta$ stays in a bounded subset of $\Sc_{\Omega,L}(\Nc)$, the set of $(k,k')\in K^2$ such that  $\eta*(\psi'_{k'}*u_2^** \psi'_{k'})*\psi_k\neq 0$ stays in a fixed finite subset of $K^2$. Thus, $(v_{H,H'})$ converges in the strong topology of $\Sc'_{\Omega,L}(\Nc)$. Denote by $u_1*u_2$ its limit.
	
	Let us now prove that $(u_1*u_2)*\psi=(u_1*\psi')*(\psi'*u_2*\psi)$ for every $\psi,\psi'\in \Sc_\Omega(\Nc)$ such that $\psi=\psi*\psi'$.  Observe first that, by Remark~\ref{oss:1}, 
	\[
	(u_1*\psi')*(\psi'*u_2*\psi)=(u_1*\psi'')*(\psi'''*u_2*\psi)
	\]
	for every $\psi'',\psi'''\in \Sc_\Omega(\Nc)$ such that $\psi=\psi*\psi''=\psi*\psi'''$. Then, observe that, by definition,
	\[
	(u_1*u_2)*\psi=\sum_{k,k'} (u_1*\psi_k)*(\psi'_{k'}*u_2*\psi'_{k'})*\psi,
	\]
	since the sum on the right-hand side is actually finite. If we choose $\psi'''\in \Sc_\Omega(\Nc)$ so that $\psi'''*\psi'_{k'}=\psi'_{k'}$ for every $k'\in K$ such that $\psi*\psi'_{k'}\neq 0$, then Remark~\ref{oss:1} again implies that
	\[
	\sum_{k,k'} (u_1*\psi_k)*(\psi'_{k'}*u_2*\psi'_{k'})*\psi=\sum_{k,k'} (u_1*\psi_k)*(\psi'''*u_2*\psi_{k'}*\psi_{k'}*\psi)=\sum_{k} (u_1*\psi_k)*(\psi'''*u_2*\psi).
	\]
	To conclude, it suffices to choose $\psi''\in \Sc_\Omega(\Nc)$ as the sum of the $\psi_k$ such that $\psi_k*\psi'''\neq 0$. Note that this implies that $u_1*u_2$ does not depend on the choice of the $\psi_k$ and the $\psi'_{k'}$.
	
	Finally, choose $\varphi''\in C^\infty_c(\Omega')$ such that $\varphi''=1$ on the supports of the $\varphi_k$, and define $\psi''_k\coloneqq \Fc_\Nc^{-1}(\varphi''(\,\cdot\, t_k^{-1}))$ for every $k\in K$.
	Then,	observe that Corollary~\ref{cor:13}  and a homogeneity argument imply that there is a constant $C>0$ such that
	\[
	\norm{(u_1*u_2)*\psi_k}_{L^{p_3}(\Nc)}\leqslant C\Delta_{\Omega'}^{(1+1/p_3-1/p_1-1/p_2)(\vect b+\vect d)}(\lambda_k) \norm{u_1*\psi''_k}_{L^{p_1}(\Nc)} \norm{\psi''_k*u_2*\psi_k}_{L^{p_2}(\Nc)}
	\] 
	for every $u_1\in B^{\vect{s_1}}_{p_1,q_1}(\Nc,\Omega)$, and for every $u_2\in \Bc^{\vect{s_2}}_{p_2,q_2}(\Nc,\Omega)$. The conclusion follows.
\end{proof}

\section{Fourier Multipliers}\label{sec:5}

In this section we consider how \emph{right} Fourier multipliers interact with the spaces $B^{\vect s}_{p,q}(\Nc,\Omega)$. This leads to some simplifications which will allow us to find fairly complete multipliers theorems. We shall \emph{not} consider \emph{left} Fourier multipliers.

\medskip

Take a bounded measurable field of operators $(M_\lambda)_{\lambda\in F'\setminus W}$. Then, Fourier analysis shows that there is a uniquely determined continuous endomorphism $\Psi$ of $L^2(\Nc)$ such that
\[
\pi_\lambda(\Psi (f))= \pi_\lambda(f) M_\lambda
\]
for almost every $\lambda\in F'\setminus W$. Therefore, $\Psi$ induces a continuous linear mapping $\Psi'\colon \Sc_{\Omega,L}(\Nc)\to \Sc_{\Omega,L}'(\Nc)$. 
Observe, though, that for every family $(\psi_k)_{k\in K}$ as in Definition~\ref{def:2}, subject to the further requirement that $\sum_k \Fc_\Nc\psi_k=1$ on $\Omega'$,
\[
\Psi'(\varphi)=\sum_{k,k'} \Psi'(\varphi*\psi_k)*\psi_{k'},
\]
for every $\varphi\in \Sc_{\Omega,L}(\Nc)$, so that $\Psi'$ is uniquely determined by $\Psi'(\,\cdot\,*\psi_k)*\psi_{k'}$. In addition,
\[
\pi_\lambda(\Psi'(\varphi*\psi_k)*\psi_{k'}) =\pi_\lambda(\varphi)\Fc_\Nc (\psi_k*\psi_{k'})(\lambda)  P_{\lambda,0}M_\lambda P_{\lambda,0}
\]
for almost every $\lambda\in F'\setminus W$. Consequently, if we wish to study the properties of $\Psi'$, then there is no need to consider bounded measurable families of operators, but simply bounded measurable \emph{functions} on $\Omega'$.

\begin{deff}
	Take $M\in L^\infty(\Omega')$. Then, we define $\Psi(M)$ as the endomorphism of $L^2(\Nc)$ such that
	\[
	\pi_\lambda(\Psi(M) f)=M(\lambda)\pi_\lambda(f) 
	\]
	for almost every $\lambda\in \Omega'$ and for every $f\in L^2(\Nc)$. We denote $\Psi'(M)$ the  continuous linear mapping $\Sc_{\Omega,L}(\Nc)\to \Sc'_{\Omega,L}(\Nc)$ induced by $\Psi(M)$.
\end{deff}

\begin{teo}
	Take $p_0\in ]0,2]$, and take $q_0\in [2,\infty]$ so that $\frac{1}{q_0}=\frac{1}{\max(1,p_0)}-\frac{1}{2}$. Fix a non-zero positive $\varphi\in C^\infty_c(\Omega')$,  and define $\cM_{p_0}$ as the space of bounded measurable functions $M\colon \Omega \to \C$ such that
	\[
	\sup\limits_{t\in T_+} \norm{ \varphi M(\,\cdot\, t)}_{B^{m(1/p_0-1/2)}_{q_0,\min(p_0,1)}(F')}<\infty,
	\]
	endowed with the corresponding topology.
	Then,  $\Psi'$ induces uniquely determined continuous linear mappings $\cM_{p_0}\to \Lc(B^{\vect s}_{p,q}(\Nc,\Omega))$ and $\cM_{p_0}\to \Lc(\mathring B^{\vect s}_{p,q}(\Nc,\Omega))$ for every $p\in[p_0,p_0']$, for every $q\in ]0,\infty]$, and for every $\vect s\in \R^r$, such that
	\[
	\left\langle \Psi'(M) u_1\vert u_2\right\rangle= \left\langle u_1\vert \Psi'(\overline M) u_2\right\rangle
	\]
	for every $M\in \cM_{p_0}$, for every $u_1\in B^{\vect s}_{p,q}(\Nc,\Omega)$, and for every $u_2\in B^{-\vect s-(1/p-1)_+(\vect b+\vect d)}_{p',q'}(\Nc,\Omega)$.
\end{teo}

The fact that the relevant dimension in this case is that of $F$, even when $\Nc$ is not abelian, may seem peculiar. Nonetheless, this is related to the fact that the considered multiplier completely dismiss the operator-valued structure of the Fourier transform on $\Nc$, in a certain sense.

\begin{proof}
	Notice that there is a unique $\Kc(M)\in \Sc'(\Nc)$ such that $\Psi(M) f=f*\Kc(M)$ for every $f\in \Sc(\Nc)$.
	
	\textsc{Step I.} Assume that $p_0\in ]0,1]$, define $\cM'_{p_0}\coloneqq \varphi' B^{m(1/p_0-1/2)}_{2,p_0}(F')$ for some $\varphi'\in C^\infty_c(\Omega')$ such that $\varphi=\varphi \varphi'$, endowed with the topology induced by $B^{m(1/p_0-1/2)}_{2,p_0}(F')$, and let us prove that $\Kc$ induces a continuous linear mapping $\cM'_{p_0}\to L^{p_0}(\Nc)$.
	Observe first that, arguing as in the proof of~\cite[Proposition 4.2]{CalziPeloso},  there is a constant $c>0$ such that
	\[
	\Kc(M)(\zeta,x)= c \int_{\Omega'} M(\lambda) \Delta_{\Omega'}^{-\vect b}(\lambda) \ee^{-\left\langle \lambda, \Phi(\zeta)\right\rangle+i \left\langle \lambda,  x\right\rangle}\,\dd \lambda
	\]
	for every $M\in \cM'_{p_0}$ and for every $(\zeta,x)\in \Nc$.
	In particular,~\cite[1.5.4]{Triebel} implies that there is a continuous linear mapping
	\[
	\Kc'\colon \cM'_{p_0}\ni M\mapsto \Kc(M)(0,\,\cdot\,)\in L^{p_0}(F).
	\]
	Now, observe that there is a constant $C_1>0$ such that
	\[
	\left\langle e_{\Omega'}, \Phi(\zeta)\right\rangle\leqslant \frac{1}{2C_1} \left\langle \lambda,\Phi(\zeta)\right\rangle
	\]
	for every $\lambda$ is a neighbourhood of $ \Supp{\varphi'}$ and for every $\zeta\in E$.\footnote{It suffices to that $C_1$ such that a neighbourhood of $\Supp{\varphi'}$ is contained in $ 2C_1 e_{\Omega'}+\Omega'$.} 
	Therefore,  the mappings
	\[
	\cM'_{p_0}\ni M\mapsto M \ee^{C_1\left\langle e_{\Omega'}, \Phi(\zeta)\right\rangle}\ee^{-\left\langle\,\cdot\, ,\Phi(\zeta)\right\rangle}\in \cM'_{p_0}
	\]
	are equicontinuous (cf.~\cite[Theorem 2.82]{Triebel}). Since $\Kc(M)(\zeta,x)=\Kc'(\ee^{-\left\langle \,\cdot\,, \Phi(\zeta)\right\rangle}M)(x)$ for every $(\zeta,x)\in\Nc$ by the preceding arguments, this proves that $\Kc$ induces a continuous linear mapping $\cM'_{p_0}\to L^{p_0}(\Nc)$.
	
	\textsc{Step II.} Assume that $p_0\in ]1,2]$, and define $\cM''_{p_0}\coloneqq \varphi B^{m(1/p_0-1/2)}_{q_0,1}(F')$, endowed with the topology induced by $B^{m(1/p_0-1/2)}_{q_0,1}(F')$. Let us prove that $\Kc$ induces a continuous linear mapping $\cM''_{p_0}\to \Cc (L^{p}(\Nc))$ for every $p\in [p_0,p_0']$, where $ \Cc (L^{p}(\Nc))$ denotes the space of $u\in \Sc'(\Nc)$ such that the mapping $\Sc(\Nc)\ni f\mapsto f*u\in \Sc'(\Nc)$ induces an endomorphism of $L^{p}(\Nc)$.
	
	We  proceed by interpolation. 
	Let us first prove that $(\Cc(L^{\ell_0}(\Nc)),\Cc(L^{\ell_1}(\Nc)))_{[\theta]}$ embeds continuously into $\Cc(L^{\ell_\theta}(\Nc))$ for every $\theta\in ]0,1[$, where $\frac{1}{\ell_\theta}=\frac{1-\theta}{\ell_0}+\frac{\theta}{\ell_1}$.
	Take two integrable step functions $f,g\colon \Nc\to \C$ such that $\norm{f}_{L^{\ell_\theta}(\Nc)}=\norm{g}_{L^{\ell'_\theta}(\Nc)}=1$, and define two holomorphic functions which are bounded on every vertical strip
	\[
	F\colon \C\ni z \mapsto \abs{f}^{z\ell_\theta/\ell_1+(1-z)\ell_\theta/\ell_0-1} f\in L^1(\Nc)\cap L^\infty(\Nc)
	\]
	and
	\[
	G\colon \C\ni z \mapsto \abs{g}^{z\ell'_\theta/\ell'_1+(1-z)\ell'_\theta/\ell'_0-1} g\in L^1(\Nc)\cap L^\infty(\Nc).
	\]
	Observe that $F(\theta)=f$ and $G(\theta)=g$, and that
	\[
	\norm{F(j+i t)}_{L^{\ell_j}(\Nc)}= 	\norm{G(j+i t)}_{L^{\ell'_j}(\Nc)}=1
	\]
	for every $t\in \R$ and for $j=0,1$. Now, fix $u\in (\Cc(L^{\ell_0}(\Nc)),\Cc(L^{\ell_1}(\Nc)))_{[\theta]}$, and take a  bounded continuous function $U\colon \overline S\to \Cc(L^{\ell_0}(\Nc))+\Cc(L^{\ell_1}(\Nc))$ which is  holomorphic on $S$, where $S=\Set{z\in \C\colon 0<\Re z<1}$, such that $U(\theta)=u$ and such that the mappings
	\[
	\R\ni t \mapsto U(j+i t)\in \Cc(L^{\ell_j}(\Nc)) 
	\]
	are continuous and bounded for $j=0,1$. Then, the mapping
	\[
	\overline S\ni z \mapsto \left\langle F(z)*U(z), G(z)\right\rangle\in \C
	\]
	is bounded and continuous, and holomorphic in $S$. In addition,
	\[
	\abs{\left\langle F(j+it)*U(j+it), G(j+it)\right\rangle}\leqslant \norm{U}
	\]
	for every $t\in \R$ and for $j=0,1$, where $\norm{U}\coloneqq \max_{j=0,1}\sup\limits_{t\in \R} \norm{U(j+it)}_{\Cc(L^{\ell_j}(\Nc))}$. Hence, 
	\[
	\abs{\left\langle f*u, g\right\rangle} \leqslant \norm{U}.
	\]
	By the arbitrariness of $f$, $g$, and $U$, this implies that $u\in \Cc(L^{\ell_\theta}(\Nc))$, and that
	\[
	\norm{u}_{\Cc(L^{\ell_\theta}(\Nc))}\leqslant \norm{u}_{(\Cc(L^{\ell_0}(\Nc)),\Cc(L^{\ell_1}(\Nc)))_{[\theta]}},
	\]
	whence our assertion. 
	
	Now, take $\theta\in [0,1]$ so that $\frac{1}{p_0}=1-\frac{\theta}{2}$. Observe that $(B^{m/2}_{2,1}(F'), B^{0}_{\infty,1}(F'))_{[\theta]}= B^{m(1/p_0-1/2)}_{q_0,1}(F')$ (cf.~\cite[Theorem 6.4.5]{BerghLofstrom}). In addition, $B^{0}_{\infty,1}(F')$ embeds continuously into $L^\infty(F')$ (cf.~\cite[Remark  2.7.1/2]{Triebel}).
	Then, \textsc{step I} and standard Fourier analysis imply that $\Kc$ induces continuous linear mappings $\cM_{1}'\to L^1(\Nc)\subseteq \Cc(L^1(\Nc)),\Cc(L^\infty(\Nc))$ and $\cM_{2}'\to \Cc(L^2(\Nc))$.
	In addition, $\cM_{p_0}''$ embeds continuously into $(\cM_{1}', \cM_{2}')_{[\theta]}$ (cf.~\cite[Theorem 2.8.2]{Triebel}), so that our assertion follows by interpolation.
	
	\textsc{Step III.} Take $\vect s\in \R^r$,  $p\in [p_0,p_0']$, and   $q\in ]0,\infty]$. Take $(\lambda_k)_{k\in K}$, $(t_k)$, $(\varphi_k)$, and $(\psi_k)$ as in Definition~\ref{def:2}, and subject to the further condition that $\sum_k \Fc_\Nc \psi_k=1$ on $\Omega'$.
	Notice that, if we replace $\varphi$ with another non-zero positive element of $C^\infty_c(\Omega')$, the space $\cM_{p_0}$  is not altered, both set-theoretically and topologically. Hence, we may further assume that $\varphi_k=\varphi \varphi_k$ for every $k\in K$.
	
	Define $\psi'_k\coloneqq \Fc_\Nc^{-1}(\varphi(\,\cdot\, t_k^{-1}))$ for every $k\in K$.
	Then, for every $k\in K$, \textsc{steps I} and \textsc{II}, together with Corollary~\ref{cor:13} and a homogeneity argument, imply that $(u*\psi_k)*(\Kc(M)*\psi'_k)$ is well defined and belongs to $L^p(\Nc)$ (to $L^p_0(\Nc)$ if $u\in \mathring B^{\vect s}_{p,q}(\Nc,\Omega)$), and that there is a constant $C_{2,p}>0$ such that
	\[
	\norm{(u*\psi_k)*(\Kc(M)*\psi'_k)}_{L^{p}(\Nc)}\leqslant C_{2,p}\norm{M }_{\cM_{p_0}}\norm{u*\psi_k}_{L^{p}(\Nc)} 
	\]
	for every $k\in K$ and for every $u\in B^{\vect s}_{p,q}(\Nc,\Omega)$.
	Now, observe that, if $u\in \Sc_{\Omega,L}(\Nc)$, then
	\[
	\Psi'(M) u= \sum_k (u*\psi_k)*\Kc(M)=\sum_k (u*\psi_k)*(\Kc(M)*\psi'_k)
	\]
	and
	\[
	[\Psi'(M) u]*\psi_k= (u*\psi_k)*\Kc(M)=(u*\psi_k)*(\Kc(M)*\psi'_k)
	\]
	since $\psi_k=\psi_k*\psi'_k$ and $\psi*\Kc(M)=\Kc(M)*\psi$ for every $k\in K$ and for every $\psi\in \Sc_\Omega(\Nc)$ (use the Fourier transform). Then, the preceding estimates imply that $\Psi'(M)$ induces an endomorphism of $\mathring B^{\vect s}_{p,q}(\Nc)$, and that the linear mapping $\Psi'\colon \cM_{p_0}\mapsto \Lc(\mathring B^{\vect s}_{p,q}(\Nc))$ is continuous.
	
	Now, observe that
	\[
	\left\langle (u*\psi_k)*(\Kc(M)*\psi'_k)\vert u'\right\rangle= \left\langle u*\psi_k\vert u'*(\Kc(\overline M)*\psi'_k)\right\rangle=\left\langle u\vert (u'*\psi_k)*(\Kc(\overline M)*\psi'_k)\right\rangle
	\]
	for every $u\in B^{\vect s}_{p,q}(\Nc)$ and for every $u'\in \mathring B^{-\vect s-(1/p-1)_+}_{p',q'}(\Nc)$, since $\Kc(M)^*=\Kc(\overline M)$. It then follows that the sum $\sum_{k} (u*\psi_k)*(\Kc(M)*\psi'_k)$ converges in the topology $\sigma^{\vect s}_{p,q}$, and that
	\[
	\left\langle \sum_k (u*\psi_k)*(\Kc(M)*\psi'_k)\Bigg\vert u'\right\rangle= \left\langle u\vert \Psi'(\overline M) u'\right\rangle.
	\]
	We then define $\Psi'(M)u\coloneqq \sum_k (u*\psi_k)*(\Kc(M)*\psi'_k)$, so that the preceding estimates imply that $\Psi'(M)$ induces an endomorphism of $B^{\vect s}_{p,q}(\Nc)$ and that the linear mapping $\Psi'\colon \cM_{p_0}\mapsto \Lc( B^{\vect s}_{p,q}(\Nc))$ is continuous. The last assertion follows from the fact that $(\Psi'(M) u)*\psi=(u*\psi)*(\Kc(M)*\psi')$ for every $\psi,\psi'\in \Sc_\Omega(\Nc)$ such that $\psi=\psi*\psi'$, and from the definition of $\left\langle\,\cdot\,\vert \,\cdot\,\right\rangle$.
\end{proof}

\section{Complex Interpolation}\label{sec:6}

In this section, we study how the spaces $B^{\vect s}_{p,q}(\Nc,\Omega)$ interact with complex interpolation. In complete analogy to the classical case, complex interpolation behaves fairly well for the spaces $B^{\vect s}_{p,q}(\Nc,\Omega)$, except possibly when $p=\infty$. In order to overcome this problem, and to extend the complex interpolation techniques to the case $\min(p,q)<1$, we develop a method described in~\cite{Triebel} and prove that the spaces $B^{\vect s}_{p,q}(\Nc,\Omega)$ interact nicely with this modified interpolation functor.

\medskip

As a particular case of~\cite[Theorem 1.81.1 and Remarks 1.18.1/1,2,3]{Triebel2}, we get the following result.\footnote{Actually, the last assertion is not contained in~\cite[Remark 1.18.1/3]{Triebel2}, but can be proved with the same methods.}

\begin{prop}\label{prop:4}
	Take a Banach pair $(A_0,A_1)$, a countable discrete space $K$, $p_0,p_1\in [1,\infty]$,  and two functions $w_0,w_1\colon K\to \R_+^*$. Then,
	\[
	(w_0\ell^{p_0}_0(K;A_0), w_1\ell^{p_1}_0(K;A_1))_{[\theta]}= w_\theta \ell^{p_\theta}_0(K;(A_0,A_1)_{[\theta]})
	\]
	with equality of norms for every $\theta\in ]0,1[$, where 
	\[
	\frac{1}{p_\theta}= \frac{1-\theta}{p_0}+\frac{\theta}{p_1} \qquad \text{and} \qquad w_\theta= w_0^{1-\theta} w_1 ^\theta.
	\]
	If, in addition, $\min(p_0,p_1)<\infty$, then
	\[
	(w_0\ell^{p_0}(K;A_0), w_1\ell^{p_1}(K;A_1))_{[\theta]}= w_\theta \ell^{p_\theta}(K;(A_0,A_1)_{[\theta]}).
	\]
	Finally, if $A_0=A_1\eqqcolon A$ and $\frac{w_\theta}{\max(w_0,w_1)}\in \ell^q(K)$ for some $q\in]0,\infty[$, then 
	\[
	(w_0\ell^{\infty}(K;A), w_1\ell^{\infty}(K;A))_{[\theta]}= w_\theta \ell^{\infty}_0(K;A).
	\]
\end{prop}

\begin{prop}\label{prop:5}
	Take $p_0,p_1,q_0,q_1\in [1,\infty]$, $\vect{s_0},\vect{s_1}\in \R^r$ and $\theta\in ]0,1[$. Then,
	\[
	(\mathring B^{\vect{s_0}}_{p_0,q_0}(\Nc,\Omega),\mathring B^{\vect{s_1}}_{p_1,q_1}(\Nc,\Omega))_{[\theta]}=\mathring B^{\vect{s}_\theta}_{p_\theta,q_\theta}(\Nc,\Omega)
	\]
	where 
	\[
	\frac{1}{p_\theta}= \frac{1-\theta}{p_0}+\frac{\theta}{p_1} ,\qquad \frac{1}{q_\theta}= \frac{1-\theta}{q_0}+\frac{\theta}{q_1},\qquad \text{and} \qquad \vect{s}_\theta= (1-\theta)\vect{s_0} +\theta \vect{s_1}.
	\]
	In addition,
	\[
	( B^{\vect{s_0}}_{p_0,q_0}(\Nc,\Omega), B^{\vect{s_1}}_{p_1,q_1}(\Nc,\Omega))_{[\theta]}= B^{\vect{s}_\theta}_{p_\theta,q_\theta}(\Nc,\Omega)
	\]
	if $\min(q_0,q_1)<\infty$. 
	
	Finally, if $r=1$ and  $\vect{s_0}\neq \vect{s_1}$, then
	\[
	( B^{\vect{s_0}}_{p,\infty}(\Nc,\Omega), B^{\vect{s_1}}_{p,\infty}(\Nc,\Omega))_{[\theta]}=\mathring B^{\vect{s}_\theta}_{p,\infty}(\Nc,\Omega)
	\]
	for every $p\in [1,\infty]$.
\end{prop}

\begin{proof}
	Take a $(\delta,R)$-lattice $(\lambda_k)_{k\in K}$ on $\Omega'$ for some $\delta>0$ and some $R>1$. Choose $t_k\in T_+$ so that $\lambda_k=e_{\Omega'}\cdot t_k$ for every $k\in K$. In addition, fix a bounded family $(\varphi_k)$ of  elements of $C^\infty_c(\Omega')$ and $\varphi'\in C^\infty_c(\Omega')$ such that
	\[
	\chi_{B_{\Omega'}(e_{\Omega'}, \delta)}\leqslant \varphi_k\leqslant\chi_{B_{\Omega'}(e_{\Omega'},  R\delta)} \leqslant \varphi'\leqslant \chi_{B_{\Omega'}(e_{\Omega'},  2R\delta)},
	\]
	and such that
	\[
	\sum_k \varphi_k(\,\cdot\,t_k^{-1})=1
	\]
	on $\Omega'$. Define $\psi_k\coloneqq \Fc_\Nc^{-1}(\varphi(\,\cdot\, t_k^{-1}))$ and $\psi'_k\coloneqq \Fc_\Nc^{-1}(\varphi'(\,\cdot\, t_k^{-1}))$ for every $k\in K$. Then, $\psi_k*\psi'_k=\psi_k$ for every $k\in K$, and there is $N\in \N$ such that for every $k\in K$ the set $K_k\coloneqq \Set{k'\in K\colon \psi_k*\psi'_{k'}\neq 0}$ has at most $N$ elements, while every element of $K$ is contained in at most $N$ of the sets $K_k$ (cf.~Section~\ref{sec:2}).
	
	Observe that, by definition, the mapping
	\[
	\Ic\colon \Sc'_{\Omega,L}(\Nc)\ni u\mapsto (u*\psi_k)\in \Sc'(\Nc)^K
	\]
	induces isomorphisms of $\mathring B^{\vect s}_{p,q}(\Nc,\Omega)$ and $B^{\vect s}_{p,q}(\Nc,\Omega)$ onto closed subspaces of $(\Delta_{\Omega'}^{-\vect s}(\lambda_k))  \ell^q_0(K;L^p_0(\Nc))$ and $(\Delta_{\Omega'}^{-\vect s}(\lambda_k))  \ell^q(K;L^p(\Nc))$, respectively, for every $p,q\in ]0,\infty]$ and for every $\vect s\in \R^r$.
	
	Now, let us prove that the continuous linear mapping\footnote{Notice that, if $H$ is a bounded subset of $\Sc_{\Omega,L}(\Nc)$, then there is a finite subset $K'$ of $K$ such that
	\[
	\left\langle u_k*\psi'_k\vert \tau\right\rangle =0
	\]
	for every $u_k\in \Sc'(\Nc)$, for every $\tau\in H$, and for every $k\in K\setminus K'$.  By the arbitrariness of $H$, the sum $ \sum_k u_k*\psi'_k$ converges in $ \Sc'_{\Omega,L}(\Nc)$ and $\Pc$ is continuous.}
	\[
	\Pc\colon \Sc'(\Nc)^K\ni (u_k)\mapsto \sum_k u_k*\psi'_k\in \Sc'_{\Omega,L}(\Nc),
	\]
	induces  continuous linear mappings of $(\Delta_{\Omega'}^{-\vect s}(\lambda_k)) \ell^q_0(K;L^p_0(\Nc))$ (resp.\ $(\Delta_{\Omega'}^{-\vect s}(\lambda_k))  \ell^q(K;L^p(\Nc))$) into  $\mathring B^{\vect s}_{p,q}(\Nc,\Omega)$ (resp.\ $B^{\vect s}_{p,q}(\Nc,\Omega)$) for every $p,q\in [1,\infty]$  and for every $\vect s\in \R^r$. 
	Take $(u_k)\in (\Delta_{\Omega'}^{-\vect s}(\lambda_k)) \ell^q_0(K;L^p_0(\Nc))$ (resp.\ $(u_k)\in (\Delta_{\Omega'}^{-\vect s}(\lambda_k))  \ell^q(K;L^p(\Nc))$) and define $u\coloneqq \Pc[(u_k)]$. Observe that
	\[
	u*\psi_k=\sum_{k'\in K_k} u_{k'}*\psi'_{k'}*\psi_k
	\]
	for every $k\in K$, so that  $u*\psi_k\in L^p_0(\Nc)$ (resp.\ $u*\psi_k\in L^p(\Nc)$) and
	\[
	\norm{u*\psi_k}_{L^p(\Nc)}\leqslant C_1\sum_{k'\in K_k} \norm{u_{k'}}_{L^p(\Nc)}
	\]
	for every $k\in K$, where $C_1\coloneqq \sup\limits_{k\in K} \norm{\psi_{k}}_{L^1(\Nc)}\sup\limits_{k\in K}\norm{\psi'_k}_{L^1(\Nc)}$ is finite by~\cite[Proposition 4.2]{CalziPeloso}.
	Hence, by~\cite[Corollary 2.49]{CalziPeloso}, there is a constant $C_2>0$ such that
	\[
	\Delta_{\Omega'}^{\vect s}(\lambda_k)\norm{u*\psi_k}_{L^p(\Nc)}\leqslant C_2\sum_{k'\in K_k} \Delta_{\Omega'}^{\vect s}(\lambda_{k'})\norm{u_{k'}}_{L^p(\Nc)}
	\]
	
	Therefore, $u\in \mathring B^{\vect s}_{p,q}(\Nc,\Omega)$ (resp.\ $u\in  B^{\vect s}_{p,q}(\Nc,\Omega)$) and 
	\[
	\norm*{\Delta_{\Omega'}(\lambda_k)\norm{u*\psi_k}_{L^p(\Nc)}}_{\ell^q(K)}\leqslant C_2 N \norm*{\Delta_{\Omega'}(\lambda_k)\norm{u_k}_{L^p(\Nc)}}_{\ell^q(K)},
	\]
	whence our assertion.
	Furthermore, observe that $\Pc \Ic(u)=u$ for every $u\in \Sc_{\Omega,L}'(\Nc)$.
	
	Therefore, the assertion follows from~\cite[Theorems 5.1.1 and 6.4.2]{BerghLofstrom} and Proposition~\ref{prop:4}.	
\end{proof}

We now proceed to defining a notion of complex interpolation for quasi-normed spaces embedded in a semi-complete Hausdorff locally convex space, inspired by~\cite[2.4.4]{Triebel}. 

\begin{deff}
	Let $A_0$ and $A_1$ be two quasi-normed spaces continuously embedded in a semi-complete Hausdorff locally convex space  $X$. We then say that $(A_0,A_1,X)$ is an admissible triple.
	
	Define $S\coloneqq \Set{z\in \C\colon 0<\Re z <1}$, and define $\Fc_X(A_0,A_1)$ as the space of bounded continuous functions $f\colon \overline S\to X$ which are holomorphic in $S$ 
	and map $i\R$ and $1+i \R$ boundedly into $A_0$ and $A_1$, respectively. We endow $\Fc_X(A_0,A_1)$ with the quasi-norm
	\[
	f \mapsto \max_{j=0,1} \sup\limits_{t\in\R} \norm{f(j+i t)}_{A_j}.
	\]
	
	For every $\theta\in ]0,1[$, we define $(A_0,A_1)_{X,\theta}$ as the image of the functional $\Fc_X(A_0,A_1)\ni f \mapsto f(\theta)\in X$, endowed with the corresponding topology.
\end{deff}

Note that we need not impose that $A_0$ and $A_1$ be complete, since holomorphy is defined with reference to the semi-complete Hausdorff locally convex space $X$.
In addition, observe that if $(A_0,A_1)$ is a Banach pair, then $(A_0,A_1,\Sigma(A_0,A_1))$ is an admissible triple and there are continuous  inclusions $(A_0,A_1)_{[\theta]}\subseteq (A_0,A_1)_{\Sigma(A_0,A_1),\theta}\subseteq (A_0,A_1)^{[\theta]}$ for every $\theta \in ]0,1[$, with the notation of~\cite{BerghLofstrom}. In particular, if $A_0$ or $A_1$ is reflexive, then the preceding inclusions are equalities by~\cite[Theorem 4.3.1]{BerghLofstrom}.  We do not know if any of the preceding inclusions is an equality in full generality.

\begin{prop}
	Let $(A_0,A_1,X)$ be an admissible triple. Then, $(A_0,A_1)_{X,\theta}$ is a quasi-normed space and embed continuously into $X$ for every $\theta\in ]0,1[$. If, in addition,  $A_0$ and $A_1$ are complete, then  $(A_0,A_1)_{X,\theta}$ is complete for every $\theta\in ]0,1[$.
\end{prop}

\begin{proof}
	By assumption, for every equicontinuous subset $H$ of $F'$ there is a constant $C_H>0$ such that $\abs{\left\langle v', x_j\right\rangle}\leqslant C_H \norm{x_j}_{A_j}$ for every $v'\in H$, for every $x_j\in A_j$, and for $j=0,1$.	
	Therefore, by the three-lines lemma,
	\[
	\abs{\left\langle v', f(z)\right\rangle}\leqslant C_H\norm{f}_{\Fc_X(A_0,A_1)}
	\]
	for every equicontinuous subset $H$ of $F'$, for every $v'\in H$, for every $z\in \overline S$, and for every $f\in \Fc_X(A_0,A_1)$. This implies that the mapping $\Fc_X(A_0,A_1) \ni f \mapsto f(\theta)\in X$ is continuous (cf.~\cite[Corollary 1 to Proposition 7 of Chapter III, \S 3, No.\ 5]{BourbakiTVS}), so that its kernel is closed, and that $f=0$ if $\norm{f}_{\Fc_X(A_0,A_1) }=0$.  It then follows easily that $\Fc_X(A_0,A_1)$ and $(A_0,A_1)_{X,\theta}$ are quasi-normed spaces continuously embedded in $X$ for every $\theta\in ]0,1[$.
	
	Then, assume that  $A_0$ and $A_1$ are complete, and let us prove that $\Fc_X(A_0,A_1)$ is complete. This will imply that $(A_0,A_1)_{X,\theta}$ is complete by~\cite[Chapter I, \S 3, No.\ 2]{BourbakiTVS}. 
	Then, let $(f_j)$ be a Cauchy sequence in $\Fc_X(A_0,A_1)$, and observe that the preceding estimates show that $(\left\langle v',f_j\right\rangle)$ is a Cauchy sequence in the topology of uniform convergence, uniformly as $v'$ runs through an equicontinuous subset of $F'$. By~\cite[Corollary 1 to Proposition 7 of Chapter III, \S 3, No.\ 5]{BourbakiTVS} again, this implies that $(f_j)$ converges uniformly on $\overline S$ to some bounded continuous function $f\colon \overline S\to X$, which is then holomorphic on $S$. Since  $(f_j(j+i t))$ is a Cauchy sequence in $A_j$ for every $t\in \R$ and for $j=0,1$, this implies that $f$ maps $j+i \R$ boundedly into $A_j$ for $j=0,1$, so that $f\in  \Fc_X(A_0,A_1)$. It is then clear that $(f_j)$ converges to $f$ in $\Fc_X(A_0,A_1)$. The proof is complete.
\end{proof}

The following result shows that $(\,\cdot\,,\,\cdot\,)_{\,\cdot\,,\theta}$ is an exact interpolation functor for the category of admissible triples, endowed with suitable morphisms.

\begin{prop}
	Let $(A_0,A_1,X)$ and $(B_0,B_1,Y)$ be two admissible triples, and let $T\colon X\to Y$ be a continuous linear mapping which maps $A_j$ continuously into $B_j$ for $j=0,1$. Then, $T$ maps $(A_0,A_1)_{X,\theta}$ continuously into $(B_0,B_1)_{Y,\theta}$ for every $\theta\in ]0,1[$. In addition, 
	\[
	\norm{T}_{\Lc((A_0,A_1)_{X,\theta},(B_0,B_1)_{Y,\theta})}\leqslant \norm{T}_{\Lc(A_0,B_0)}^{1-\theta}\norm{T}_{\Lc(A_1,B_1)}^{\theta}
	\]
	for every $\theta\in ]0,1[$.
\end{prop}

\begin{proof}
	Define $C_j\coloneqq \norm{T}_{\Lc(A_j,B_j)}$ for $j=0,1$.
	It suffices to observe that, if $f\in \Fc_X(A_0,A_1)$, then $g\coloneqq (C_0/C_1)^{\,\cdot\,-\theta} T f\in \Fc_Y(B_0,B_1)$, $g(\theta)=T[f(\theta)]$, and $\norm{g}_{\Fc_Y(B_0,B_1)}\leqslant C_0^{1-\theta} C_1^{\theta}\norm{f}_{\Fc_X(A_0,A_1)}$.
\end{proof}

\begin{teo}\label{prop:9}
	Take $p_0,p_1,q_0,q_1\in ]0,\infty]$, $\vect{s_0},\vect{s_1}\in \R^r$ and $\theta\in ]0,1[$. Then,
	\[
	(\mathring B^{\vect{s_0}}_{p_0,q_0}(\Nc,\Omega),\mathring B^{\vect{s_1}}_{p_1,q_1}(\Nc,\Omega))_{\Sc'_{\Omega,L}(\Nc),\theta}=\mathring B^{\vect{s}_\theta}_{p_\theta,q_\theta}(\Nc,\Omega)
	\]
	and
	\[
	( B^{\vect{s_0}}_{p_0,q_0}(\Nc,\Omega), B^{\vect{s_1}}_{p_1,q_1}(\Nc,\Omega))_{\Sc'_{\Omega,L}(\Nc),\theta}= B^{\vect{s}_\theta}_{p_\theta,q_\theta}(\Nc,\Omega),
	\] 
	where 
	\[
	\frac{1}{p_\theta}= \frac{1-\theta}{p_0}+\frac{\theta}{p_1} ,\qquad \frac{1}{q_\theta}= \frac{1-\theta}{q_0}+\frac{\theta}{q_1},\qquad \text{and} \qquad \vect{s}_\theta= (1-\theta)\vect{s_0} +\theta \vect{s_1}.
	\]
\end{teo}

\begin{proof}
	Fix $(\lambda_k)$, $(\varphi_k)$, $(\psi_k)$, and $(\psi'_k)$ as in the proof of Proposition~\ref{prop:5}. Define $K_k\coloneqq \Set{k'\in K\colon \psi'_{k'}*\psi'_k\neq 0}$, and observe that there is $N\in\N$ such that $\card(K_k)\leqslant N$ for every $k\in K$ (cf.~Section~\ref{sec:2}).
	Take $u\in B^{\vect{s}_\theta}_{p_\theta,q_\theta}(\Nc,\Omega)$	such that
	\[
	\norm*{\Delta_\Omega^{\vect{s}_\theta}(\lambda_{k}) \norm{u*\psi_{k}}_{L^{p_\theta}(\Nc)}}_{\ell^{q_\theta}(K)}=1,
	\]
	and define
	\[
	U\colon \C\ni z\mapsto \sum_{k\in K} \Delta_{\Omega'}^{z \vect{s_0'}+\vect{s_1'}}(\lambda_k)\norm{u*\psi_{k}}_{L^{p_\theta}(\Nc)} ^{c z+d}  \big(\abs{u*\psi_k}^{a z+b}u*\psi_k\big)*\psi'_k\in \Sc'_{\Omega,L}(\Nc)
	\]
	for some $a,b,c,d\in \R$ and $\vect{s_0'},\vect{s_1'}\in \R^r$ to be chosen. Let us first prove that $U$ is well defined, holomorphic, and bounded on $\overline S$. Arguing as in the proof of Proposition~\ref{prop:5}, we see that $U$ is well defined. Then, take $\tau\in \Sc_{\Omega,L}(\Nc)$, and observe that there is a finite subset $K'$ of $K$ such that
	\[
	\left\langle U(z) \vert \tau\right\rangle= \sum_{k\in K'} \Delta_{\Omega'}^{z \vect{s_0'}+\vect{s_1'}}(\lambda_k)\norm{u*\psi_{k}}_{L^{p_\theta}(\Nc)} ^{c z+d} \left\langle  \big(\abs{u*\psi_k}^{a z+b}u*\psi_k\big)*\psi'_k\Big\vert \tau \right\rangle
	\]
	for every $z\in \C$. In order to prove that $U$ is holomorphic and bounded on $\overline S$, it then suffices to prove that the mapping $z \mapsto \left\langle  \abs{u*\psi_k}^{a z+b}u*\psi_k\Big\vert \tau *\psi'^*_k\right\rangle$ is holomorphic on $\C$ and bounded on $\overline S$, for every $k\in K'$. Since $u*\psi_k\in L^\infty(\Nc)$ by Corollary~\ref{cor:10}, this is easily proved.
	
	Observe that $U(\theta)=u$ if $a\theta+b=c\theta+d=0$ and $\theta\vect{s_0'}+\vect{s_1'}=\vect 0$.  In addition, by Corollary~\ref{cor:12} and a homogeneity argument, there is a constant $C_1>0$ such that 
	\[
	\norm*{\big(\abs{u*\psi_k}^{a j+b}u*\psi_k\big)*\psi'_k*\psi_{k'}}_{L^{p_j}(\Nc)}\leqslant\norm*{\abs{u*\psi_k}^{a j+b+1}*\abs{\psi'_k*\psi_{k'}}}_{L^{p_j}(\Nc)}\leqslant C_1\norm{u*\psi_k}_{L^{p_\theta}(\Nc)}^{a j+b+1}
	\]
	for every $j=0,1$ and for every $k,k'\in K$, provided that  $\frac{a j+b+1}{p_\theta}=\frac{1}{p_j}$ if $p_\theta<\infty $ and $aj+b=0$ if $p_\theta=\infty$. We then choose $a\coloneqq \frac{p_\theta}{p_0}-\frac{p_\theta}{p_1}$ and $b\coloneqq \frac{p_\theta}{p_0}-1$ if $p_\theta<\infty$, and   simply $a=b=0$ if $p_\theta=\infty$, so that the preceding conditions are satisfied.
	Therefore, by means of~\cite[Corollary 2.49]{CalziPeloso}, we see that there is a constant $C_2>0$ such that
	\[
	\begin{split}
		\Delta_\Omega^{\vect{s_j}}(\lambda_k)\norm{U(j+i t)*\psi_k }_{L^{p_j}(\Nc)}&\leqslant \Delta_\Omega^{\vect{s_j}}(\lambda_k)C_1 N^{(1/p_j-1)_+} \sum_{k'\in K_k}\Delta_\Omega^{j \vect{s_0'}+\vect{s_1'}}(\lambda_{k'}) \norm{u*\psi_{k'}}_{L^{p_\theta}(\Nc)}^{(a+c)j+b+d+1}\\
			&\leqslant C_2 \sum_{k'\in K_k}\Delta_\Omega^{\vect{s}_\theta}(\lambda_{k'}) \norm{u*\psi_{k'}}_{L^{p_\theta}(\Nc)}^{(a+c)j+b+d+1}\\
	\end{split}
	\]
	for every $t\in \R$, for every $k\in K$, and for $j=0,1$, provided that $\vect{s_j}+j \vect{s_0'}+\vect{s_1'}=\vect{s}_\theta$. Then, set $c\coloneqq \frac{q_\theta}{q_0}-\frac{q_\theta}{q_1}-a$ and $d\coloneqq \frac{q_\theta}{q_0}-1-b$ if $q_\theta<\infty$ and $c\coloneqq-a$ and $d\coloneqq-b$ if $q_\theta=\infty$, so that $c\theta+d=0$, and $\frac{(a+c) j+b+d+1}{q_j}=\frac{1}{q_\theta}$ if $q_\theta<\infty$ and $(a+c)j+b+d=0$ if $q_\theta=\infty$. In addition, set $\vect{s_0'}\coloneqq \vect{s_0}-\vect{s_1}$ and $\vect{s_1'}\coloneqq \vect{s}_\theta-\vect{s_0}$, so that the preceding conditions are satisfied. Then,
	\[
	\norm*{\Delta_\Omega^{\vect{s}_j}(\lambda_k)\norm{U(j+i t)*\psi_k }_{L^{p_j}(\Nc)}}_{\ell^{q_j}(K)}\leqslant C_2 N^{\max(1,1/q_j)}
	\] 
	for $j=0,1$. This proves that  $B^{\vect{s}_\theta}_{p_\theta,q_\theta}(\Nc,\Omega)\subseteq( B^{\vect{s_0}}_{p_0,q_0}(\Nc,\Omega), B^{\vect{s_1}}_{p_1,q_1}(\Nc,\Omega))_{\Sc'_{\Omega,L}(\Nc),\theta}$ continuously. 
	
	If, in addition, $u\in \mathring B^{\vect{s}_\theta}_{p_\theta,q_\theta}(\Nc,\Omega)$, then it is readily verified that $U(j+i t)\in \mathring B^{\vect s_j}_{p_j,q_j}(\Nc,\Omega)$ for $j=0,1$ and for every $t\in \R$, so that $\mathring B^{\vect{s}_\theta}_{p_\theta,q_\theta}(\Nc,\Omega)\subseteq( \mathring B^{\vect{s_0}}_{p_0,q_0}(\Nc,\Omega), \mathring B^{\vect{s_1}}_{p_1,q_1}(\Nc,\Omega))_{\Sc'_{\Omega,L}(\Nc),\theta}$ continuously. 
		
	Conversely,  take $u\in ( B^{\vect{s_0}}_{p_0,q_0}(\Nc,\Omega), B^{\vect{s_1}}_{p_1,q_1}(\Nc,\Omega))_{\Sc'_{\Omega,L}(\Nc),\theta}$, and take 
	\[
	U\in \Fc_{\Sc'_{\Omega,L}(\Nc)}(B^{\vect{s_0}}_{p_0,q_0}(\Nc,\Omega), B^{\vect{s_1}}_{p_1,q_1}(\Nc,\Omega))
	\]
	such that $U(\theta)=u$. 
	Set $\ell\coloneqq\min(p_0,p_1,q_0,q_1)$. Then,~\cite[2.4.6/2]{Triebel2} shows that there are two probability measures $\mu_0,\mu_1$ on $\R$ such that
	\[
	\abs{f(z)}^\ell\leqslant \left( \int_\R \abs{f(it)}^\ell\,\dd \mu_0(t) \right)^{1-\theta} \left( \int_\R \abs{f(1+it)}^\ell\,\dd \mu_1(t) \right)^\theta
	\]
	for every $z\in \theta+i \R$ and for every bounded uniformly continuous function $f\colon\overline S\to \C$ which is holomorphic on $S$.
	For every $\varepsilon>0$, define $U_\varepsilon\colon \overline S\ni z \mapsto \ee^{\varepsilon(z^2-z)-\varepsilon(\theta^2-\theta)} U(z)\in \Sc'_{\Omega,L}(\Nc)$, so that $U_\varepsilon\in  \Fc_{\Sc'_{\Omega,L}(\Nc)}(B^{\vect{s_0}}_{p_0,q_0}(\Nc,\Omega), B^{\vect{s_1}}_{p_1,q_1}(\Nc,\Omega))$ and $U_\varepsilon(\theta)=u$. Observe that, for every $(\zeta,x)\in \Nc $, the function 
	\[
	z \mapsto (U_\varepsilon(z)*\psi_k)(\zeta,x)= \ee^{\varepsilon(z^2-z)-\varepsilon(\theta^2-\theta)}\left\langle U(z) \vert L_{(\zeta,x)}\psi_k^*\right\rangle
	\]
	is bounded and uniformly continuous (actually, vanishes at the point at infinity of $\overline S$). 
	Then, by H\"older's and Minkowski's integral inequalities,
	\[
	\begin{split}
	\norm{u*\psi_k}_{L^{p_\theta}(\Nc)}&\leqslant \norm*{\left( \int_\R \abs*{U_\varepsilon(i t)*\psi_k}^\ell\,\dd \mu_0(t) \right)^{1-\theta} \left( \int_\R \abs*{U_\varepsilon(1+i t)*\psi_k}^\ell\,\dd \mu_1(t) \right)^\theta   }_{L^{p_\theta/\ell}(\Nc)}^{1/\ell}\\
		&\leqslant \norm*{ \int_\R \abs*{U_\varepsilon(i t)*\psi_k}^\ell\,\dd \mu_0(t) }_{L^{p_0/\ell}(\Nc)}^{(1-\theta)/\ell} \norm*{\int_\R \abs*{U_\varepsilon(1+i t)*\psi_k}^\ell\,\dd \mu_1(t)    }_{L^{p_1/\ell}(\Nc)}^{\theta/\ell}\\
		&\leqslant \left(\int_\R \norm{U_\varepsilon(i t)*\psi_k}_{L^{p_0}(\Nc)}^\ell\,\dd \mu_0(t) \right)^{(1-\theta)/\ell} \left(\int_\R \norm{U_\varepsilon(1+i t)*\psi_k}_{L^{p_1}(\Nc)}^\ell\,\dd \mu_1(t)    \right)^{\theta/\ell }
	\end{split}
	\]
	for every $k\in K$. Therefore, by H\"older's and Minkowski's integral inequalities again,
	\[
	\begin{split}
		\norm*{\Delta_{\Omega'}^{\vect s_\theta}(\lambda_k) \norm{u*\psi_k}_{L^{p_\theta}(\Nc)}}_{\ell^{q_\theta}(K)}&\leqslant \left(\int_\R \norm*{\Delta_{\Omega'}^{\vect s_0}(\lambda_k)\norm{U_\varepsilon(i t)*\psi_k}_{L^{p_0}(\Nc)}}_{L^{q_0}(K)}^\ell\,\dd \mu_0(t) \right)^{(1-\theta)/\ell} \\
			&\qquad \times\left(\int_\R \norm*{\Delta_{\Omega'}^{\vect s_1}(\lambda_k)\norm{U_\varepsilon(1+i t)*\psi_k}_{L^{p_1}(\Nc)}}_{L^{q_1}(K)}^\ell\,\dd \mu_1(t)    \right)^{\theta/\ell }\\
			& \leqslant \norm{U_\varepsilon}_{\Fc_{\Sc'_{\Omega,L}(\Nc)}(B^{\vect{s_0}}_{p_0,q_0}(\Nc,\Omega), B^{\vect{s_1}}_{p_1,q_1}(\Nc,\Omega))}.
	\end{split}
	\]
	By the arbitrariness of $\varepsilon>0$ and $U$, this implies that $( B^{\vect{s_0}}_{p_0,q_0}(\Nc,\Omega), B^{\vect{s_1}}_{p_1,q_1}(\Nc,\Omega))_{\Sc'_{\Omega,L}(\Nc),\theta}$ embeds continuously into $ B^{\vect{s}_\theta}_{p_\theta,q_\theta}(\Nc,\Omega)$. 
	
	Finally, assume that  $u\in (\mathring  B^{\vect{s_0}}_{p_0,q_0}(\Nc,\Omega), \mathring B^{\vect{s_1}}_{p_1,q_1}(\Nc,\Omega))_{\Sc'_{\Omega,L}(\Nc),\theta}$, and that
	\[
	U\in \Fc_{\Sc'_{\Omega,L}(\Nc)}(\mathring B^{\vect{s_0}}_{p_0,q_0}(\Nc,\Omega), \mathring B^{\vect{s_1}}_{p_1,q_1}(\Nc,\Omega)).
	\]
	Then, the preceding remarks imply that
	\[
	\abs{(u*\psi_k)(\zeta,x)}\leqslant \left( \int_\R \abs{(U_\varepsilon(it)*\psi_k)(\zeta,x)}^\ell\,\dd \mu_0(t) \right)^{(1-\theta)/\ell} \left( \int_\R \abs{(U_\varepsilon(1+it)*\psi_k)(\zeta,x)}^\ell\,\dd \mu_1(t) \right)^{\theta/\ell},
	\]
	for every $(\zeta,x)\in \Nc$ and for every $k\in K$,
	so that $u*\psi_k\in C_0(\Nc)$ by Corollary~\ref{cor:10} and the dominated convergence theorem. Hence, $u*\psi_k\in L^{p_\theta}_0(\Nc)$ for every $k\in K$. Analogously, from the inequality (proved above)
	\[
	\begin{split}
	\Delta_{\Omega'}^{\vect s_\theta}(\lambda_k)\norm{u*\psi_k}_{L^{p_\theta}(\Nc)}&\leqslant \left(\int_\R \Delta_{\Omega'}^{\vect s_0}(\lambda_k)\norm{U_\varepsilon(i t)*\psi_k}_{L^{p_0}(\Nc)}^\ell\,\dd \mu_0(t) \right)^{(1-\theta)/\ell} \\
		&\qquad \times\left(\int_\R \Delta_{\Omega'}^{\vect s_1}(\lambda_k)\norm{U_\varepsilon(1+i t)*\psi_k}_{L^{p_1}(\Nc)}^\ell\,\dd \mu_1(t)    \right)^{\theta/\ell },
	\end{split}
	\]
	for every $k\in K$, one deduces that $(\Delta_{\Omega'}^{\vect s_\theta}(\lambda_k)\norm{u*\psi_k}_{L^{p_\theta}(\Nc)})\in \ell^\infty_0(K)$, so that $(\Delta_{\Omega'}^{\vect s_\theta}(\lambda_k)\norm{u*\psi_k}_{L^{p_\theta}(\Nc)})\in \ell^{q_\theta}_0(K)$. Hence, $u\in \mathring B^{\vect s_\theta}_{p_\theta,q_\theta}(\Nc,\Omega)$.
\end{proof}

\section{`Classical' Besov spaces on $\Nc$}\label{sec:7}

In this section we introduce some Besov spaces on $\Nc$ associated to a suitable positive Rockland operator.\footnote{A Rockland operator is a homogeneous, hypoelliptic, left-invariant differential operator. We say that a Rockland operator $\Rc$ is positive if $\int_\Nc (\Rc f) \overline f\,\dd \Hc^{2n +m}=\int_\Nc f \overline {\Rc f}\,\dd \Hc^{2n +m}\geqslant 0$ for every $f\in C^\infty_c(\Nc)$. } In the classical case, that is, when $\Nc=F$, these Besov spaces are exactly the classical \emph{homogeneous} Besov spaces on $F$. In the general case, the resulting spaces are non-commutative analogues of the classical homogeneous Besov spaces. 
Since the Sobolev spaces associated to positive Rockland operators on graded groups investigated in~\cite{FischerRuzhansky} do not depend on the choice of the Rockland operator, it is likely that the same holds for the Besov spaces defined in this setting, even though we shall not prove that.
We mention here that Besov spaces on graded Lie groups were studied in~\cite{CardonaRuzhansky}, which nonetheless contains no proofs. For this reason, but also to discuss the space $\Sc_\Lc(\Nc)$ and its relationships with the Besov spaces $B^s_{p,q}(\Nc,\Lc)$, we shall provide complete proofs. 
Notice, in addition, that, in virtue of the (real) interpolation results claimed in~\cite[Theorem 3.2]{CardonaRuzhansky}, the Besov spaces $B^s_{p,q}(\Nc,\Lc)$ do not actually depend on the positive Rockland operator $\Lc$, \emph{at least when $p\in]1,\infty[$ and $q\in [1,\infty[$}.

\begin{deff}
	Define
	\[
	\Lc\coloneqq \frac 1 4 \left(\sum_{j} (Z_j\overline{Z_j}+\overline{Z_j} Z_j)\right)^2-\sum_{k} U_k^2,
	\]
	where the $Z_j$  are the left-invariant vector fields (with complex coefficients) on $\Nc$ such that $(Z_j)_{(0,0)}=\partial_{E,v_j}$  for some orthonormal basis $(v_j)$ of $E$ over $\C$, while $(U_k)$ is an orthonormal basis of invariant vector fields on the centre $F$ on $\Nc$. 
	
	Given a bounded measurable function $\varphi\colon \R\to \C$, we denote by $\Kc(\varphi)$ the (right) convolution kernel of the operator $\varphi(\Lc)$, defined by spectral calculus.
\end{deff}

Observe that $\Lc$ does not depend on the choice of $(Z_j)$, $(\overline{Z_j})$, and $(U_k)$. In addition, $\Lc$ is a positive Rockland operator of degree $2$ on $\Nc$, so that it induces an essentially self-adjoint operator on $L^2(\Nc)$ with initial domain $C^\infty_c(\Nc)$. We shall therefore make use of the corresponding spectral calculus. Cf.\ e.g.~\cite{Martini,Calzi,Calzi2} for more information on the spectral calculus associated to positive Rockland (or more general) operators on graded (or general) Lie groups. We mention, in particular, that $\Kc(\Sc(\R))\subseteq \Sc(\Nc)$ and that $\Kc(m_1 m_2)=\Kc(m_1)*\Kc(m_2)$ for every two bounded measurable functions $m_1,m_2$ on $\R$.
Observe that $\Lc$ has real coefficients, so that $\Kc(\varphi)$ is real whenever $\varphi$ is real.

\medskip

For every $\lambda\in F'$, define $J_\lambda\in \Lc_\C(E)$ so that
\[
\left\langle \lambda_\C, \Phi(\zeta,\zeta')\right\rangle=\left\langle \zeta \vert -i J_\lambda \zeta'\right\rangle_E
\]
for every $\zeta,\zeta'\in E$.

\begin{prop}\label{prop:1}
	Take $\lambda\in \Omega'$. Then, $\dd \pi_\lambda(\Lc)$ has purely discrete spectrum and $P_{\lambda,0}$ is an eigenprojector of $\dd \pi_\lambda(\Lc)$. In addition, 
		\[
		\dd \pi_\lambda(\Lc) P_{\lambda,0}=\left((\tr\abs{J_\lambda})^2+\abs{\lambda}^2\right) P_{\lambda,0}=\left(\Big(\sum_{j}\left\langle \lambda, \Phi(v_j)\right\rangle\Big)^2+\abs{\lambda}^2\right) P_{\lambda,0}
		\]
		for every orthonormal basis $(v_j)$ of $E$ over $\C$.
\end{prop}

\begin{proof}
	Observe that
	\[
	\dd \pi_\lambda(\Lc)=\left(  \sum_{j} (2 \left\langle \lambda, \Phi(\,\cdot\,,v_j)\right\rangle\partial_{v_j}+ \left\langle \lambda, \Phi(v_j)\right\rangle I)\right) ^2+\abs{\lambda}^2 I
	\]
	for every orthonormal basis $(v_j)$ of $E$ over $\C$, thanks to~\cite[Proposition 1.15]{CalziPeloso}.
	Now, choose the orthonormal basis $(v_j)$ of $E$ over $\C$ so that it is orthogonal for the scalar product $\left\langle\lambda,\Phi(\,\cdot\,,v_j)\right\rangle$.	
	In order to see that $\dd \pi_\lambda(\Lc)$ has purely discrete spectrum, it will suffice to observe that the monomials $w_\alpha\coloneqq \prod_{j}\left\langle \lambda,\Phi(\,\cdot\,,v_j)\right\rangle^{\alpha_j}$, $\alpha\in\N^n$, form a total orthogonal family in $H_\lambda$ (argue as in~\cite[Theorem 1.63]{Folland}), and that
	\[
	\dd \pi_\lambda(\Lc) w_\alpha= \left( \Big(\sum_j (2\alpha_j+1) \left\langle \lambda, \Phi(v_j)\right\rangle \Big)^2+\abs{\lambda}^2\right)w_\alpha
	\]
	thanks to the above formula for $\dd \pi_\lambda(\Lc)$, for every $\alpha\in\N^n$.
	In particular, $P_{\lambda,0}$ is an eigenprojector of $\dd \pi_\lambda(\Lc)$, with corresponding eigenvalue
	\[
	  \Big( \sum_{j}\left\langle \lambda, \Phi(v_j)\right\rangle \Big) ^2+\abs{\lambda}^2.
	\]
	To conclude, observe that  $\left\langle \lambda, \Phi(v,w)\right\rangle= \left\langle\abs{J_\lambda}v\vert w\right\rangle$ for every $v,w\in E$ by the definition of $J_\lambda$, so that
	\[
	\Big( \sum_{j}\left\langle \lambda, \Phi(v_j)\right\rangle \Big) ^2+\abs{\lambda}^2=(\tr\abs{J_\lambda})^2+\abs{\lambda}^2.\qedhere
	\]
\end{proof}

\begin{deff}
	Define $N\colon F'\ni \lambda \mapsto \left( (\tr\abs{J_\lambda})^2+\abs{\lambda}^2 \right)\in \R_+$.
	
	For every compact subset $K$ of $\R^*_+$, define 
	\[
	\Sc_{\Lc,K}(\Nc)\coloneqq \Set{\varphi\in \Sc(\Nc)\colon \varphi=\varphi*\Kc(\chi_K)},
	\]
	endowed with the topology induced by $\Sc(\Nc)$. We define $\Sc_\Lc(\Nc)$ as the inductive limit of the $\Sc_{\Lc,K}(\Nc)$, as $K$ runs through the set of compact subsets of $\R_+^*$. We denote by $\Sc'_{\Lc}(\Nc)$ the strong dual of $\Sc_{\Lc}(\Nc)$.
	
	We denote by $\Pc$ the space of polynomials on $\Nc$.
\end{deff}

Notice that, since $\Kc(\chi_K)$ is real-valued for every compact subset $K$ of $\R_+^*$, the space $\Sc_\Lc(\Nc)$ is invariant under conjugation. Hence, we may define  $\left\langle T\vert \varphi\right\rangle \coloneqq \left\langle T, \overline \varphi\right\rangle$  for $T\in \Sc'_\Lc(\Nc)$ and $\varphi\in \Sc_\Lc(\Nc)$.

\begin{prop}\label{prop:2}
	The following hold:
	\begin{enumerate}
		\item[\textnormal{(1)}] $\Sc_{\Lc,K}(\Nc)$ is a Fr\'echet Montel space for every compact subset $K$ of $\R_+^*$;
		
		\item[\textnormal{(2)}] $\Sc_\Lc(\Nc)$ is a complete bornological Montel space;
		
		\item[\textnormal{(3)}] the closure of $\Sc_\Lc(\Nc)$ in $\Sc(\Nc)$ is 
		\[
		\Set{\varphi\in \Sc(\Nc)\colon \forall P\in \Pc\:\: \int_\Nc \varphi(\zeta,x) P(\zeta,x)\,\dd (\zeta,x)=0 }.
		\]
	\end{enumerate}
\end{prop}

\begin{proof}
	(1) This follows from the fact that $\Sc_{\Lc,K}(\Nc)$ is a closed subspace of the Fr\'echet Montel space $\Sc(\Nc)$.
	
	(2) This follows  from~\cite[Proposition 9 of Chapter II, \S 4, No.\ 6, Example 3 of Chapter III, \S 2, and Example 3 of Chapter IV, \S 2, No.\ 5]{BourbakiTVS}.
	
	(3) Let us first prove that $ \int_\Nc \varphi(\zeta,x) P(\zeta,x)\,\dd (\zeta,x)=0$ for every $\varphi\in \Sc_\Lc(\Nc)$ and for every $P\in \Pc$.
	Notice that, since $\Lc$ is formally self-adjoint and for every  $P\in \Pc$ there is $k\in\N$ such that $\Lc^k P=0$ by homogeneity, it will suffice to prove that, for every $\varphi\in \Sc_\Lc(\Nc)$ and for every $k\in\N$, there is $\varphi_k\in \Sc(\Nc)$ such that $\varphi= \Lc^k \varphi_k$.
	Then, fix $\varphi\in \Sc_\Lc(\Nc)$ and $k\in \N$. Observe that there is s positive $\tau\in C^\infty_c(\R_+^*)$ such that $\varphi=\varphi*\Kc(\tau)$. Then,  define $I'_k\coloneqq \Kc[(\,\cdot\,)^{-k}\tau]$, so that $I'_k\in \Sc(\Nc)$ and
	\[
	\Lc^k I'_{k}=\Kc(\tau).
	\]
	If we set
	\[
	\varphi_k\coloneqq\varphi*I'_{k}\in \Sc(\Nc),
	\]
	then
	\[
	\Lc^k \varphi_k= \varphi*\Kc(\tau)=\varphi,
	\]
	whence our claim.

	Then, fix $\eta\in C^\infty_c(\R)$ so that $\chi_{[0,1]}\leqslant \eta\leqslant \chi_{[-1,2]}$. If we define $\psi_j\coloneqq \Kc(\eta(2^{-2j}\,\cdot\,))=(2^{ -j}\,\cdot\,)_* \Kc(\eta)$ for every $j\in \Z$, then clearly $f*(\psi_h-\psi_{-k})=(\eta(2^{-2h}\Lc)-\eta(2^{2k}\Lc))f$ converges to $f$ in $L^2(\Nc)$ for $h,k\to +\infty$, for every $f\in L^2(\Nc)$, by spectral theory.  In addition, if $\varphi\in \Sc(\Nc)$ and $ \int_\Nc \varphi(\zeta,x) P(\zeta,x)\,\dd (\zeta,x)=0$   for every $P\in \Pc$, then~\cite[Proposition 5.8]{Calzi}  implies that the set of $\varphi*(\psi_h-\psi_{-k})$, as $h,k\in \N$, is bounded (hence relatively compact) in $\Sc(\Nc)$. Therefore, the preceding arguments imply that $\varphi*(\psi_h-\psi_{-k})$ converges to $\varphi$ in $\Sc(\Nc)$ for $h,k\to +\infty$. Since clearly $\varphi*(\psi_h-\psi_{-k})\in \Sc_\Lc(\Nc)$ for every $h,k\in \N$, the assertion follows.
\end{proof}

\begin{deff}
	Given $u\in \Sc'_\Lc(\Nc)$ and $\varphi\in \Sc_\Lc(\Nc)$, we define $u*\varphi\in \Sc'(\Nc)$ so that
	\[
	\left\langle u*\varphi\vert\tau\right\rangle=\left\langle u\vert\tau *\varphi^*\right\rangle
	\]
	for every $\tau\in \Sc(\Nc)$. We shall identify $u*\varphi$ with an element of $C^\infty(\Nc)$.
\end{deff}

\begin{lem}\label{lem:1}
	Take $s\in \R$ and $p,q\in ]0,\infty]$. Take two bounded families $(\varphi_j)_{j\in \Z}$ and $(\varphi'_j)_{j\in \Z}$ of positive elements of $C^\infty_c(\R^*)$ such that
	\[
	\sum_{j} \varphi_j(2^{-2j}\,\cdot\,), \sum_{j}\varphi'_j(2^{-2j}\,\cdot\,)\geqslant 1.
	\]
	Define $\psi_j\coloneqq \Kc(\varphi_j(2^{-2j}\,\cdot\,))$ and $\psi'_j\coloneqq \Kc(\varphi'_j(2^{-2j}\,\cdot\,))$.
	Then, there is a constant $C>0$ such that
	\[
	\frac 1 C \norm*{2^{s j} \norm{T*\psi'_j}_{L^p(\Nc)} }_{\ell^q(\Z)}\leqslant \norm*{2^{s j} \norm{T*\psi_j}_{L^p(\Nc)} }_{\ell^q(\Z)}\leqslant C\norm*{2^{s j} \norm{T*\psi'_j}_{L^p(\Nc)} }_{\ell^q(\Z)}
	\]
	for every $T\in \Sc'(\Nc)$. In addition, 
	\[
	(2^{s j}T*\psi_j) \in \ell^q_0(\Z; L^p_0(\Nc)) \quad \text{if and only if} \quad (2^{s j}T*\psi'_j) \in \ell^q_0(\Z; L^p_0(\Nc)).
	\]
\end{lem}

\begin{proof}
	Choose $\ell\in \R$ and $M\in\N$   so that $\Supp{\varphi_j}, \Supp{\varphi'_j}\subseteq [2^{\ell},2^{\ell+2 M}]$ for every $j\in \Z$.
	Define
	\[
	\widetilde \varphi\coloneqq \sum_{j\in \Z} \varphi_k(2^{-2 j}\,\cdot\,),
	\]
	and observe that the sum defining $\widetilde \varphi$ is locally finite on $\R_+^*$, so that $\widetilde \varphi$ is of class $C^\infty$ on $\R_+^*$. 
	In addition, if $\varphi'_{j'}(2^{-2 j'}\,\cdot\,)\varphi_j(2^{-2 j}\,\cdot\,)\neq 0 $ then $\abs{j-j'}\leqslant M$. 
	Then, for every $j'\in \Z$,
	\[
	\varphi'_{j'}(2^{-2 j'}\,\cdot\,)=\sum_{j=j'-M}^{j'+M} \frac{\varphi_{j'}'(2^{-2 j'}\,\cdot\,) \varphi_j (2^{-2 j}\,\cdot\,)}{\widetilde \varphi}= \sum_{j=j'-M}^{j'+M} \widetilde \varphi'_{j'}(2^{-2 j'}\,\cdot\,) \varphi_j (2^{-2 j}\,\cdot\,),
	\]
	where
	\[
	\widetilde \varphi'_{j'}\coloneqq \frac{\varphi_{j'}'}{\widetilde \varphi(2^{2 j'}\,\cdot\,)}= \frac{\varphi'_{j'}}{\sum_{j=j'-M}^{j'+M} \varphi_j(2^{2(j'-j)}\,\cdot\, ) }= \frac{\varphi'_{j'}}{\sum_{j=-M}^{M} \varphi_j(2^{-2j}\,\cdot\, ) }
	\]
	for every $j'\in \Z$.
	Notice that the family $(\widetilde \varphi'_{j'})_{j'\in \Z}$ is bounded in $C^\infty_c(\R_+^*)$.
	Then, define $\widetilde \psi'_{j}\coloneqq \Kc(\widetilde \varphi'_{j}(2^{-2 j}\,\cdot\,))$ for every $j\in \Z$, and observe that Corollary~\ref{cor:11}  and a dilation argument imply that there is a constant $C_1>0$ such that, for every $T\in \Sc_{\Lc}'(\Nc)$ and for every $j,j'\in \Z$, 
	\[
	\norm{(T*\psi_{j})*\widetilde \psi'_{j'}}_{L^p(\Nc)}\leqslant C_1\norm{T*\psi_j}_{L^p(\Nc)}.
	\]
	Now, for every $j'\in \Z$,
	\[
	T*\psi'_{j'}=\sum_{j=j'-M}^{j'+M}(T*\psi_j)*\widetilde \psi'_{j'}
	\]
	by the associativity of convolution on $\Sc_{\Lc}'(\Nc)\times \Sc_\Lc(\Nc)\times \Sc_\Lc(\Nc)$, so that
	\[
	\norm{T*\psi'_{j'}}_{L^p(\Nc)}\leqslant C_1 (2M+1)^{(1/p-1)_+}\sum_{j=j'-M}^{j'+M}\norm{T*\psi_j}_{L^p(\Nc)} 
	\]
	for every $j'\in \Z$. Hence, there is a constant $C_2>0$ such that
	\[
	2^{j' s}\norm{T*\psi'_{j'}}_{L^p(\Nc)}\leqslant C_2\sum_{j=j'-M}^{j'+M} 2^{j s}\norm{T*\psi_j}_{L^p(\Nc)}
	\]
	for every $j'\in \Z$.
	
	Therefore, 
	\[
	\norm*{2^{s j} \norm{T*\psi'_j}_{L^p(\Nc)} }_{\ell^q(\Z)}\leqslant C_2 (2 M+1)^{\max(1,1/q)}\norm*{2^{s j} \norm{T*\psi_j}_{L^p(\Nc)} }_{\ell^q(\Z)}.
	\]
	By symmetry, the first assertion is proved. For what concerns the second assertion, assume that $(2^{s j}T*\psi_j) \in \ell^q_0(\Z; L^p_0(\Nc))$. Then, it is clear that
	\[
	T*\psi'_{j'}=\sum_{j=j'-M}^{j'+M}(T*\psi_j)*\widetilde \psi'_{j'}\in L^p_0(\Nc)
	\]
	for every $j'\in \Z$. Analogously, it is readily seen that
	\[
	\big(2^{j' s}\norm{T*\psi'_{j'}}_{L^p(\Nc)}\big)_{j'}\leqslant \bigg(C_2\sum_{j=j'-M}^{j'+M} 2^{j s}\norm{T*\psi_j}_{L^p(\Nc)} \bigg)_{j'}\in \ell^q_0(\Z),
	\]
	whence $(2^{s j}T*\psi'_j) \in \ell^q_0(\Z; L^p_0(\Nc))$. The second assertion follows therefore by symmetry.
\end{proof}

\begin{deff}
	Take $s\in \R$ and $p,q\in ]0,\infty]$. Take $(\psi_j)$ as in Lemma~\ref{lem:1}. Then, we define $B_{p,q}^s(\Nc,\Lc)$ (resp.\ $\mathring B_{p,q}^s(\Nc,\Lc)$) as the space of $T\in \Sc'_\Lc(\Nc)$ such that
	\[
	(2^{s j}T*\psi_j) \in \ell^q(\Z; L^p(\Nc)) \qquad \text{(resp.\ $(2^{s j}T*\psi_j) \in \ell^q_0(\Z; L^p_0(\Nc))$)},
	\]
	endowed with the corresponding topology. 
\end{deff}

Notice that we chose to define $B_{p,q}^s(\Nc,\Lc)$ as a subspace of $\Sc'_\Lc(\Nc)$ since this choice simplifies some results. Nonetheless, one may prove the $B_{p,q}^s(\Nc,\Lc)$ embeds canonically into $\Sc'(\Nc)/\Pc$.

\begin{prop}\label{prop:11}
	Take $s\in \R$ and $p,q\in ]0,\infty]$. Then, $B_{p,q}^s(\Nc,\Lc)$ and $\mathring B_{p,q}^s(\Nc,\Lc)$ are complete metrizable topological vector spaces and embed canonically into $\Sc'(\Nc)/\Pc$. Furthermore, $\mathring B_{p,q}^{s}(\Nc,\Lc)$ is the closure of $\Sc_\Lc(\Nc)$   in $ B_{p,q}^{s}(\Nc,\Lc)$.
\end{prop}

Before we prove this result, we need some preparations.

\begin{prop}\label{prop:12}
	Take $s_1,s_2\in \R$ and $p_1,p_2,q_1,q_2\in ]0,\infty]$ such that
	\[
	p_1\leqslant p_2, \qquad q_1\leqslant q_2, \qquad \text{and} \qquad s_2=s_1+\left(\frac{1}{p_2}-\frac{1}{p_1}\right) Q. 
	\]
	Then, there are continuous inclusions
	\[
	B_{p_1,q_1}^{s_1}(\Nc,\Lc)\subseteq B_{p_2,q_2}^{s_2}(\Nc,\Lc) \qquad \text{and} \qquad \mathring B_{p_1,q_1}^{s_1}(\Nc,\Lc)\subseteq \mathring B_{p_2,q_2}^{s_2}(\Nc,\Lc) .
	\]
\end{prop}

\begin{proof}
	This follows from Corollary~\ref{cor:10}, a dilation argument, and the canonical inclusions $\ell^{q_1}(\Z)\subseteq \ell^{q_2}(\Z)$ and $\ell^{q_1}_0(\Z)\subseteq \ell^{q_2}_0(\Z)$.
\end{proof}

\begin{prop}\label{prop:13}
	Take $s\in \R$, $p,q\in ]0,\infty]$ and two bounded families $(\varphi_j)_{j\in \Z}$, $(\varphi'_j)_{j\in \Z}$ of positive elements of $C^\infty_c(\R_+^*)$ such that 
	\[
	\sum_{j\in \Z} \varphi_j(2^{-2 j}\,\cdot\,)\varphi'_j(2^{-2 j}\,\cdot\,)=1
	\]
	on $\R_+^*$. Then, the sesquilinear form
	\[
	B^{s}_{p,q}(\Nc,\Lc)\times B^{-s+(1/p-1)_+ Q}_{p',q'}(\Nc,\Lc)\ni (u,u')\mapsto \sum_j \left\langle u*\psi_j\big\vert  u'*\psi'_j\right\rangle\in \C,
	\]
	where $\psi_j\coloneqq \Kc(\varphi_j(2^{-2 j}\,\cdot\,))$ and $\psi'_j\coloneqq \Kc(\varphi'_j(2^{-2 j}\,\cdot\,))$ for every $j\in\Z$,  is well defined and continuous, and does not depend on the choice of $(\varphi_j), (\varphi'_j)$.
\end{prop}

Arguing as in the proof of~\cite[Theorem 4.23]{CalziPeloso}, one may actually prove that the above sesquilinear form induces an antilinear isomorphism of $B^{-s+(1/p-1)_+ Q}_{p',q'}(\Nc,\Lc)$ onto $\mathring B^{s}_{p,q}(\Nc,\Lc)'$. 

An analogous assertion holds for bilinear forms. 

\begin{proof}
	The fact that $\left\langle\,\cdot\,\vert\,\cdot\,\right\rangle$ is well defined and continuous follows from the inclusion 
	\[
	B^{s}_{p,q}(\Nc,\Lc)\subseteq B^{s-(1/p-1)_+ Q}_{\max(1,p),\max(1,q)}(\Nc,\Lc)
	\]
	(cf.~Proposition~\ref{prop:12}). Then, take two bounded families $(\eta_j),(\eta'_j)$ of positive elements of $C^\infty_c(\R_+^*)$ such that
	\[
	\sum_{j\in \Z} \eta_j(2^{-2 j}\,\cdot\,)\eta'_j(2^{-2 j}\,\cdot\,)=1
	\]
	on $\R_+^*$, and define $\tau_j\coloneqq \Kc(\eta_j(2^{-2 j}\,\cdot\,))$ and $\tau'_j\coloneqq \Kc(\eta'_j(2^{-2 j}\,\cdot\,))$ for every $j\in \Z$. Then, for every $u\in B^{s}_{p,q}(\Nc,\Lc)$ and for every $u'\in B^{-s+(1/p-1)_+ Q}_{p',q'}(\Nc,\Lc)$,
	\[
	\begin{split}
	\sum_j \left\langle u*\psi_j\big\vert  u'*\psi'_j\right\rangle&=\sum_j \left\langle \sum_{j'} u*\tau_{j'}*\tau'_{j'}*\psi_j\Bigg\vert u'*\psi'_j\right\rangle\\
		&=\sum_j \sum_{j'} \left\langle u*\tau_{j'}\big\vert u'*\tau'_{j'}*\psi_j*\psi'_j\right\rangle
	\end{split}
	\]
	since only a finite number of terms of the inner sum are non-zero. For a similar reason, the sum in $j$ and $j'$ converges absolutely, so that one may reverse the above computations and infer that $	\sum_j \left\langle u*\psi_j\big\vert  u'*\psi'_j\right\rangle=	\sum_{j'}\left\langle u*\tau_{j'}\big\vert  u'*\tau'_{j'}\right\rangle$, whence the result.
\end{proof}

\begin{lem}\label{lem:2}
	Take $s\in \R$, $p,q\in ]0,\infty]$, and $\varphi\in C^\infty_c(\R_+^*)$. Then, $\Oc^p(\Kc(\varphi))$ embeds as a closed subspace of both $B_{p,q}^s(\Nc,\Lc)$ and $L^p(\Nc)$.
\end{lem}

\begin{proof}
	It suffices to observe that there is a bounded family $(\varphi_j)_{j\in \Z}$ of positive elements of $C^\infty_c(\R_+^*)$ such that
	\[
	\sum_{j\in \Z} \varphi_j(2^{-2j}\,\cdot\,)\geqslant 1,
	\]
	$\varphi_0\varphi=\varphi$, and $\varphi_j(2^{-2j}\,\cdot\,)\varphi=0$ for every $j\in \Z$, $j\neq 0$. 
\end{proof}

\begin{proof}[Proof of Proposition~\ref{prop:11}.]
	The fact that $B^s_{p,q}(\Nc,\Lc)$ is complete follows from the fact that the equivalent quasi-norms which define its topology extend to  lower semi-continuous convex functions on $\Sc'_\Lc(\Nc)$ which are finite only on $B^s_{p,q}(\Nc,\Lc)$, and from the completeness of $\Sc'_\Lc(\Nc)$. For what concerns the second assertion, it will suffice to prove that  $\Sc_\infty(\Nc)\coloneqq \Set{\varphi\in \Sc(\Nc)\colon \forall P\in \Pc\:\:\left\langle \varphi \vert P\right\rangle=0 }$  embeds continuously into $B^{-s+(1/p-1)_+}_{p',q'}(\Nc,\Lc)$, and to apply  Proposition~\ref{prop:13}.
	Then, take a positive $\eta\in C^\infty_c(\R_+^*)$ such that 
	\[
	\sum_{j\in \Z} \eta(2^{-2 j}\,\cdot\,)=1
	\]
	on $\R_+^*$, and define $\psi_j\coloneqq \Kc(\eta(2^{-2 j}\,\cdot\,))$ for every $j\in \Z$. In addition, define $\psi_{k,j}\coloneqq \Kc(((\,\cdot\,)^k \eta)(2^{-2 j}\,\cdot\,)) $ for every $j,k\in \Z$, and denote by $I_k$ a log-homogeneous fundamental solution of $\Lc^k$ of degree $2k-Q$,\footnote{In other words, there are a homogeneous polynomial $P$ of degree $2 k-Q$ and a homogeneous norm  $\rho$ of class $C^\infty$ on $\Nc\setminus \Set{(0,0)}$ such that $I_k- P\log\rho$ is homogeneous of degree $2 k-Q$, and $\Lc^k I_k=\delta_{(0,0)}$.} if $k>0$ (cf.~\cite{Geller}), and $\Lc^{-k}\delta_{(0,0)}$ otherwise.
	Then, a simple modification of~\cite[Proposition 5.8]{Calzi} implies that $\Sc_\infty(\Nc)*I_k\subseteq\Sc(\Nc)$ for every $k\in \Z$. In addition,~\cite[Proposition 5.8]{Calzi} implies that
	\[
	2^{-2 k j} \varphi*\psi_{j}=(\varphi*I_k)*\psi_{k,j}
	\]
	is uniformly bounded in $\Sc(\Nc)$ as $j$ runs through $\Z$, for every $\varphi\in \Sc_\infty(\Nc)$ and for every $k\in \Z$. Thus, $(2^{s' j}\varphi*\psi_j)\in\ell^q_0(\Z;L^p_0(\Nc))$ for every $\varphi\in\Sc_\infty(\Nc)$, for every $s'\in \R$, and for every $p,q\in ]0,\infty]$. The assertion follows by means of the closed graph theorem.
	
	For what concerns the last assertion, observe first that (the canonical image of) $\Sc_{\Lc,L}(\Nc)$ is contained in $\mathring B^{s}_{p,q}(\Nc,\Lc)$. Conversely, take $u\in \mathring B^{s}_{p,q}(\Nc,\Lc)$ and observe that $u=\sum_{j\in \Z} u*\psi_j$, with convergence in $ B^{s}_{p,q}(\Nc,\Lc)$ (argue as in the proof of~\cite[Lemma 4.22]{CalziPeloso}). Furthermore, fix $\varphi\in \Sc(\Nc)$ so that $\varphi(0,0)=\norm{\varphi}_{L^\infty(\Nc)}=1$, and define $\varphi_k\coloneqq \varphi(2^{-k}\,\cdot\,)$ for every $k\in \N$. Take $\eta'\in C^\infty_c(\R^*_+)$ such that $\eta=\eta' \eta$, and define $\psi'_j\coloneqq \Kc(\eta'(2^{-2 j}\,\cdot\,))$ for every $j\in \Z$.
	Let us prove that $[\varphi_k(u*\psi_j)]*\psi'_j$ converges to $u*\psi_j$ in $L^p(\Nc)$ for every $j\in \Z$. Observe first that $\abs{u*\psi_k}*\abs{\psi'_j}\in L^p_0(\Nc)$ for every $j\in \Z$, by Corollary~\ref{cor:12}. Since $\abs{[\varphi_k(u*\psi_j)]*\psi'_j}\leqslant \abs{u*\psi_k}*\abs{\psi'_j}$ for every $k\in \N$ and for every $j\in \Z$, it will suffice to prove that $[\varphi_k(u*\psi_j)]*\psi'_j$ converges locally uniformly to $u*\psi_j$ for every $j\in \Z$,  and this is clear. Since clearly $[\varphi_k(u*\psi_j)]*\psi'_j\in \Sc_\Lc(\Nc)$  for every $j\in \Z$, the assertion follows by means of Lemma~\ref{lem:2}.
\end{proof}

In a completely analogous fashion, one may prove that $B^{\vect s}_{p,q}(\Nc,\Omega)$ embeds continuously into the dual $\widetilde\Sc'_\Omega(\Nc)$ of the conjugate of the closure $\widetilde\Sc_\Omega(\Nc)$ of $\Sc_{\Omega,L}(\Nc)$ in $\Sc(\Nc)$.\footnote{We shall prove in~\cite{CalziPeloso2} that this definition of $ \widetilde \Sc_\Omega(\Nc)$ agrees with the one given in~\cite{CalziPeloso}.} This solves a problem left open in~\cite[4.3.7]{CalziPeloso}.

\begin{prop}\label{prop:6}
	Take $\vect s\in \R^r$ and $p,q\in ]0,\infty]$. Then, the canonical linear mapping $\widetilde \Sc_{\Omega}(\Nc)\to \Sc'_{\Omega,L}(\Nc)$ is continuous and injective, and induces a continuous injective linear mapping 
	\[
	\widetilde \Sc_{\Omega}(\Nc)\to \mathring B^{-\vect s-(1/p-1)_+(\vect b+\vect d)}_{p',q'}(\Nc,\Omega).
	\]
	This latter mapping, in turn, induces a continuous injective linear mapping 
	\[
	B^{\vect s}_{p,q}(\Nc,\Omega)\to \widetilde \Sc_{\Omega}'(\Nc).
	\]
\end{prop}

\begin{proof}
	Observe first that~\cite[Proposition 4.11]{CalziPeloso} implies that the mapping $\varphi\mapsto \varphi*I_\Omega^{\vect{s'}}$ induces an automorphism of $\widetilde \Sc_{\Omega}(\Nc)$ for every $\vect{s'}\in \R^r$.
	Let us first prove that the canonical (continuous) linear mapping $\widetilde \Sc_{\Omega}(\Nc)\to \Sc'_{\Omega,L}(\Nc)$ is  injective. Indeed, if $\varphi\in\widetilde \Sc_{\Omega}(\Nc)$ and $\left\langle \varphi\vert \psi\right\rangle=0$ for every $\psi\in \Sc_{\Omega,L}(\Nc)$, then
	\[
	(\varphi*\psi)(\zeta,x)=\left\langle \varphi\big\vert \delta_{(\zeta,x)}*\psi^* \right\rangle=0
	\]
	for every $\psi\in \Sc_{\Omega}(\Nc)$ and for every $(\zeta,x)\in \Nc$, thanks to~\cite[Propositions 4.2 and 4.5]{CalziPeloso}. Therefore,
	\[
	0=\pi_\lambda(\varphi*\psi)=\Fc_\Nc(\psi)(\lambda)\pi_\lambda(\varphi)
	\]
	for every $\psi\in \Sc_\Omega(\Nc)$ and for every $\lambda \in \Omega'$, whence $\varphi=0$ since $\Fc_\Nc(\Sc_\Omega(\Nc))=C^\infty_c(\Omega')$ and $\pi_\lambda(\varphi)=0$ for every $\lambda\in F'\setminus (W\cup \Omega')$.
		
	Now, take $\varphi\in \widetilde \Sc_\Omega(\Nc)$, and let us prove that its canonical image in $\Sc'_{\Omega,L}(\Nc)$ is contained in the space $B^{-\vect s-(1/p-1)_+(\vect b+\vect d)}_{p',q'}(\Nc,\Omega)$.	
	Take $(\lambda_k)$, $(\varphi_k)$, and $(\psi_k)$ as in Definition~\ref{def:2}. Take $\vect{s'}\in \R^r$, $\ell\in\N$, and $\varphi\in \widetilde \Sc_{\Omega,L}(\Nc)$. Then,~\cite[Proposition 4.11]{CalziPeloso} and the preceding remarks show that $\varphi*I_\Omega^{\vect{s'}} \in \widetilde\Sc_{\Omega}(\Nc)$, $\psi_k*I_\Omega^{-\vect{s'}}\in \Sc_\Omega(\Nc)$, and that
	\[
	\pi_\lambda(\varphi*\psi_k)=\pi_\lambda((\varphi*I_\Omega^{\vect{s'}})*(\psi_k*I_\Omega^{-\vect{s'}}))
	\]
	for every $\lambda\in \Omega'$ and for every $k\in K$. Hence,
	\[
	\varphi*\psi_k=(\varphi*I_\Omega^{\vect{s'}})*(\psi_k*I_\Omega^{-\vect{s'}})
	\]
	for every $k\in K$.
	Next, define
	\[
	\psi_{k,\vect{s'},\ell}\coloneqq \Fc_\Nc^{-1}( N^{-\ell}  \Fc_\Nc(\psi_k*I^{-\vect{s'}}_\Omega) )
	\]
	and observe that
	\[
	\pi_\lambda(\varphi*\psi_k)=\pi_\lambda(\Lc^{\ell}(\varphi*I_\Omega^{\vect{s'}})*\psi_{k,\vect{s'},\ell})
	\]
	for every $\lambda\in \Omega'$ and for every $k\in K$, whence
	\[
	\varphi*\psi_k=\Lc^{\ell}(\varphi*I_\Omega^{\vect{s'}})*\psi_{k,\vect{s'},\ell}
	\]
	for every $k\in K$.
	
	Therefore,~\cite[Corollaries 2.49 and 2.51, and Proposition 4.11]{CalziPeloso} imply that
	\[
	\norm{\varphi*\psi_k}_{L^{p'}(\Nc)}\leqslant \Delta_{\Omega'}^{\vect{s'}}(\lambda_k)\max(1,N(\lambda_k))^{-\ell} C_{\vect{s'},\ell}\max\left(\norm{\Lc^\ell( \varphi*I^{\vect{s'}}_\Omega)}_{L^{p'}(\Nc)} ,\norm{ \varphi*I^{\vect{s'}}_\Omega}_{L^{p'}(\Nc)}\right)
	\]
	for every $k\in K$,
	where 
	\[
	C_{\vect{s'},\ell}\coloneqq \max\left(\sup\limits_{N(\lambda_k)\leqslant 1} \Delta_{\Omega'}^{-\vect{s'}}(\lambda_k)\norm{ \psi_k*I_\Omega^{-\vect{s'}}}_{L^1(\Nc)}, \sup\limits_{N(\lambda_k)\geqslant 1} \Delta_{\Omega'}^{-\vect{s'}}(\lambda_k)N(\lambda_k)^\ell\norm{\psi_{k,\vect{s'},\ell}}_{L^1(\Nc)}	\right)
	\]
	is finite.\footnote{To see this, observe that $\Fc_\Nc(\psi_{k,\vect{s'},\ell})=i^{\vect{s'}}N^{-\ell} \Delta_{\Omega'}^{\vect{s'}} \varphi_k(\,\cdot\, t_k^{-1})$, and that the family $(N(\lambda_k)^\ell N^{-\ell}(\,\cdot\,t_k) \Delta_{\Omega'}^{\vect{s'}} \varphi_k)_{N(\lambda_k)\geqslant 1}$ is bounded in $C^\infty_c(\Omega')$.} Since clearly $\varphi*\psi_k\in L^{p'}_0(\Nc)$ for every $k\in K$, this implies that $(\max(1,N(\lambda_k))^{\ell}\Delta_{\Omega'}^{-\vect{s'}}(\lambda_k)\varphi*\psi_k)\in \ell^\infty(K;L^{p'}_0(\Nc)) $ for every $\vect{s'}\in \R^r$ and for every $\ell\in \N$. By the arbitrariness of $\vect{s'}$ and $\ell$, this implies that $(\Delta_{\Omega'}^{\vect{s}}(\lambda_k)\varphi*\psi_k)\in \ell^{q'}_0(K;L^{p'}_0(\Nc)) $, that is, $\varphi\in \mathring B^{-\vect s-(1/p-1)_+(\vect b+\vect d)}_{p',q'}(\Nc,\Omega)$. 
	Thus, there is a canonical linear mapping $\widetilde \Sc_{\Omega}(\Nc)\to \mathring B^{-\vect s-(1/p-1)_+(\vect b+\vect d)}_{p',q'}(\Nc,\Omega)$, which is necessarily continuous by the closed graph theorem. Alternatively, continuity may be proved directly using the above estimates.
	Since $B^{\vect s}_{p,q}(\Nc,\Omega)$ embeds continuously and canonically into $B^{\vect s+(1/p-1)_+(\vect b+\vect d)}_{\max(1,p),\max(1,q)}(\Nc)$, which is canonically identified with the dual of $\mathring B^{-\vect s-(1/p-1)_+(\vect b+\vect d)}_{p',q'}(\Nc,\Omega)$ thanks to~\cite[Proposition 4.19 and Theorem 4.23]{CalziPeloso}, the last assertion follows by transposition.
\end{proof}

\section{Comparison between Classical and Analytic Besov Spaces}\label{sec:8}

We now proceed to compare the Besov spaces of analytic type $B^{\vect s}_{p,q}(\Nc,\Omega)$ with the `classical' Besov spaces $B^{s}_{p,q}(\Nc,\Lc)$. We shall first show that $B^{\vect s}_{p,q}(\Nc,\Omega)$ is the closure in the topology $\sigma^{\vect s}_{p,q}$ of the union of an increasing sequence $(V_j)$ of closed vector subspaces on which the topologies induced by $B^{\vect s}_{p,q}(\Nc,\Omega)$ and $B^{\sum_j s_j}_{p,q}(\Nc,\Lc)$ coincide (cf.~Proposition~\ref{prop:7}). Thus, it is natural to restrict attention to embeddings of the form $B^{\vect s}_{p,q}(\Nc,\Omega)\to B^{\sum_j s_j}_{p,q}(\Nc,\Lc)$. 

When $r=1$, Proposition~\ref{prop:7} shows that $B^{\vect s}_{p,q}(\Nc,\Omega)$ is actually canonically isomorphic to a closed subspace of $B^{\sum_j s_j}_{p,q}(\Nc,\Lc)$. When $r>1$, though, the situation is far more delicate and it turns out that the existence of a canonical embedding $B^{\vect s}_{p,q}(\Nc,\Omega)\to B^{\sum_j s_j}_{p,q}(\Nc,\Lc)$ is equivalent to a suitable decoupling inequality $(D')^{\vect s}_{p,q}$, which cannot hold unless $\vect s\in \R_-^r$ (cf.~Theorem~\ref{teo:1} and Proposition~\ref{prop:8}). If, in addition, $\vect s\neq \vect 0$, then the decoupling condition $(D')^{\vect s}_{p,q}$ is equivalent to the condition $(D)^{\vect s}_{p,q}$ defined in~\cite[Definition 5.5]{CalziPeloso} which, in turn, is closely related to the determination of the boundary value spaces of suitable weighted Bergman spaces on $D$ (cf.~\cite[Theorem 5.10]{CalziPeloso}).
When $\Omega$ is a Lorentz cone, then the   $\ell^2$-decoupling inequality proved by Bourgain and Demeter~\cite{BourgainDemeter} allows to determine much more precise sufficient conditions for property $(D)^{\vect s}_{p,q}$ to hold, as noticed in~\cite{BekolleGonessaNana}. To the best of our knowledge, a complete characterization of the validity of property $(D)^{\vect s}_{p,q}$ (or $(D')^{\vect s}_{p,q}$) is not available, in general.

\begin{prop}\label{prop:7}
	Take $\vect s\in \R^r$ and $p,q\in ]0,\infty]$. In addition, fix a compact subset $H$ of $\Omega'$, and define 
	\[
	V\coloneqq \Set{T\in B_{p,q}^{\vect s}(\Nc,\Omega)\colon \forall \tau\in C^\infty_c(\Omega'\setminus (\R_+^* H))  \quad T*\Fc_\Nc^{-1}(\tau)=0 }
	\]
	and
	\[
	V_0\coloneqq V\cap \mathring B_{p,q}^{\vect s}(\Nc,\Omega),
	\]
	endowed with the topology induced by $B_{p,q}^{\vect s}(\Nc,\Omega)$.
	Then, $V$ and $V_0$ embed as closed subspaces of $B_{p,q}^{\sum_j s_j}(\Nc,\Lc)$ and $\mathring B_{p,q}^{\sum_j s_j}(\Nc,\Lc)$, respectively.
\end{prop}

Notice that, if $(H_k)$ is an increasing sequence of compact subsets of $\Omega'$ whose union is $\Omega'$, then the corresponding subspaces $V_{0,k}$ of $\mathring B_{p,q}^{\vect s}(\Nc,\Omega)$ form an increasing sequence of closed subspaces of $\mathring B_{p,q}^{\vect s}(\Nc,\Omega)$ whose union is dense (and also dense in $B_{p,q}^{\vect s}(\Nc,\Omega)$ for the weak topology $\sigma^{\vect s}_{p,q}$).

\begin{proof}
	We may assume that $H$ contains $e_{\Omega'}$.
	Take $\varphi\in C^\infty_c(\R_+^*)$, and  define $\psi\coloneqq \Kc(\varphi)$. Take $\eta\in \Sc_\Omega(\Nc)$, and observe that Proposition~\ref{prop:1} and~\cite[Proposition 3.7]{Martini} imply that
	\[
	\dd \pi_\lambda(\eta*\psi)= (\Fc_\Nc \eta)(\lambda) P_{\lambda,0} \varphi(\dd \pi_\lambda(\Lc))=(\Fc_\Nc \eta)(\lambda) \varphi(N(\lambda)) P_{\lambda,0} 
	\]
	for every $\lambda\in \Omega'$. Analogously, one proves that $\dd \pi_\lambda(\eta*\psi)=0$ for every $\lambda\in F'\setminus (W\cup \Omega')$, so that $\eta*\psi\in \Sc_\Omega(\Nc)$.  
	
	Observe that $\delta \coloneqq \inf\limits_{0\neq j\in \Z} d_{\Omega'}(e_{\Omega'}, 2^j e_{\Omega'})/2>0$. Then, we may find a family $(t_k)_{k\in K}$ which is maximal for the property that $d_{\Omega'}(e_{\Omega'}\cdot t_k, e_{\Omega'}\cdot t_{k'})\geqslant 2\delta$ for every $k,k'\in K$, $k\neq k'$, such that $\Z\subseteq K$, and such that $\lambda\cdot t_j=2^j \lambda$ for every $j\in \Z$ and for every $\lambda\in \Omega'$. In particular, if we define $\lambda_k\coloneqq e_{\Omega'}\cdot t_k$, then $(\lambda_k)$ is a $(\delta,2)$-lattice on $\Omega'$. 
	Let $H'$ be a compact neighbourhood of $H$ in $\Omega'$, and let us prove that there is $R>1$ such that $H'\subseteq \bigcup_{j\in \Z} B(\lambda_j, R\delta)$. Indeed, the sets $B_R\coloneqq \bigcup_{j\in \Z} B(\lambda_j, R\delta)$,  $R>0$, form an increasing open covering of $\Omega'$ such that $2^j B_R=B_R$ for every $j\in \Z$. Hence, there is $R>0$ such that $[1,2] H'\subseteq B_R$, so that $\R_+^* H'\subseteq\bigcup_{j\in \Z} 2^j[1,2] H'\subseteq B_R$.
	
	Now, take $\eta\in C^\infty_c(\R_+^*)$ such that
	\[
	\sum_{j\in \Z} \eta(2^{-2 j}\,\cdot\,)=1
	\]
	on $\R_+^*$, and fix $\tau\in C^\infty_c(\Omega')$ so that $\chi_{B(e_{\Omega'},R\delta)}\leqslant \tau\leqslant \chi_{B(e_{\Omega'},2 R\delta)}$. Define
	\[
	\varphi_j\coloneqq \tau(\eta\circ N)
	\]
	for every $j\in\Z$, so that
	\[
	 \chi_{\R_+^* H'}\leqslant \sum_{j\in\Z} \varphi_j(\,\cdot\, t_j^{-1})\leqslant\chi_{\Omega'}.
	\]
	Then, we may find a bounded family $(\varphi_k)_{k\in K\setminus \Z}$ of positive elements of $C^\infty_c(\Omega')$ such that
	\[
	\sum_{k\in K} \varphi_k(\,\cdot\, t_k^{-1})=1
	\]
	on $\Omega'$. In particular, $\varphi_k(\,\cdot\, t_k^{-1})=0$ on  $\R_+^* H'$ for every $k\in K\setminus \Z$, and
	\[
	u*\Fc_\Nc^{-1}(\varphi_k(\,\cdot\, t_k^{-1}))=\begin{cases}
		u*\Kc(\eta_k(2^{-2 k}\,\cdot\,)) &\text{if $k\in \Z$}\\
		0 & \text{if $k\not \in \Z$}
	\end{cases}
	\]
	for every $u\in V$ and for every $k\in K$.
	Since $\Delta_{\Omega'}^{\vect s}(\lambda_k)=2^{k \sum_j s_j}$ for every $k\in \Z$, the assertion follows.
\end{proof}

Choosing  $H=\Set{1}$ when $r=1$, we get the following corollary.

\begin{cor}
	Assume that $r=1$ and take $s\in \R$ and $p,q\in ]0,\infty]$. Then, $B_{p,q}^{s}(\Nc,\Omega)$ embeds canonically as a closed subspace of $B_{p,q}^s(\Nc,\Lc)$.
\end{cor}

 The situation when $r>1$ is more complicated.

\begin{deff}
	Take $\vect s\in \R^r$ and $p,q\in ]0,\infty]$. We say that property $(D')_{p,q}^{\vect s,0}$  holds if there are a $(\delta, R)$-lattice $(\lambda_k)_{k\in K}$ on $\Omega'$, for some $\delta>0$ and some $R>1$, a bounded family $(\varphi_k)_{k\in K}$ of positive elements of $C^\infty_c(\Omega')$ such that
	\[
	\sum_{k\in K} \varphi_k(\,\cdot\, t_k^{-1})\geqslant 1
	\]
	on $\Omega'$, where $t_k\in T_+$ and $\lambda_k=e_{\Omega'}\cdot t_k$ for every $k\in K$, and two constants $C>0$ and $c>1$ such that
	\[
	\norm*{\sum_{\substack{k\in K_c}} u_k*\psi_k}_{L^p(\Nc)}\leqslant C  \norm*{ \Delta_{\Omega'}^{\vect s}(\lambda_k) \norm{u_k*\psi_k}_{L^p(\Nc)} }_{\ell^q(K_c)}
	\]
	for every $(u_k)\in \Sc_{\Omega,L}(\Nc)^{(K_c)}$, where $\psi_k\coloneqq \Fc_\Nc^{-1}(\varphi_k(\,\cdot\, t_k^{-1}))$ for every $k\in K$, while 
	\[
	K_c\coloneqq \Set{k\in K\colon \varphi_k(\,\cdot\, t_k^{-1})( \chi_{[1/c , c]}\circ N)\neq 0}.
	\]
\end{deff}

Notice that the same argument used to prove~\cite[Lemma 5.6]{CalziPeloso} shows that property $(D')^{\vect 0}_{p,q}$ holds for every $p\in ]0,\infty]$ and for every $q\in ]0, \min(p,p')]$. 

In addition, if property $(D')^{\vect s}_{p,q}$ holds, then property $(D')^{\vect {s_2}}_{p,q_2}$ holds for every $\vect{s_2}\leqslant \vect s$ and for every $q_2\in ]0,q]$ (cf.~\cite[Lemma 2.34]{CalziPeloso}).

\begin{lem}\label{lem:4}
	Take $\vect s\in \R^r$ and $p,q\in ]0,\infty]$, and assume that property $(D')_{p,q}^{\vect s}$ holds. Take a $(\delta, R)$-lattice $(\lambda_k)_{k\in K}$ on $\Omega'$, for some $\delta>0$ and some $R>1$, a bounded family $(\varphi_k)_{k\in K}$ of positive elements of $C^\infty_c(\Omega')$ such that
	\[
	\sum_{k\in K} \varphi_k(\,\cdot\, t_k^{-1})\geqslant 1
	\]
	on $\Omega'$, where $t_k\in T_+$ and $\lambda_k=e_{\Omega'}\cdot t_k$ for every $k\in K$, and a constant $c>1$. Then, there is a constant $C>0$ such that
	\[
	\norm*{\sum_{\substack{k\in K_c}} u_k*\psi_k}_{L^p(\Nc)}\leqslant C  \norm*{ \Delta_{\Omega'}^{\vect s}(\lambda_k) \norm{u_k*\psi_k}_{L^p(\Nc)} }_{\ell^q(K_c)}
	\]
	for every $(u_k)\in \Sc'_{\Omega,L}(\Nc)^{(K_c)}$, where $\psi_k\coloneqq\Fc_\Nc^{-1}(\varphi_k(\,\cdot\, t_k^{-1}))$ for every $k\in K$, while 
	\[
	K_c\coloneqq \Set{k\in K\colon \varphi_k(\,\cdot\, t_k^{-1})( \chi_{[1/c , c]}\circ N)\neq 0}.
	\]
\end{lem}

\begin{proof}
	\textsc{Step I.} We first prove the assertion when $u_k\in \Sc_{\Omega,L}(\Nc)$ for every $k$.
	By assumption, there are a $(\delta',R')$-lattice $(\lambda'_{k'})_{k'\in K'}$ on $\Omega'$ for some $\delta'>0$ and some $R'>1$, a bounded family $(\varphi'_{k'})_{k'\in K'}$ of elements of $C^\infty_c(\Omega')$ such that 
	\[
	\sum_{k'\in K'}\varphi'_{k'}(\,\cdot\, t_{k'}'^{-1})\geqslant 1
	\]
	on $\Omega'$, where $t'_{k'}\in T_+$ and 
	\[
	\lambda'_{k'}=e_{\Omega'}\cdot t'_{k'}
	\]
	for every $k'\in K'$, and two constants $C'>0$ and $c'>1$ such that
	\[
	\norm*{\sum_{k'\in K'_{c'}} u_{k'}*\psi'_{k'}}_{L^p(\Nc)}\leqslant C' \norm*{\Delta_{\Omega'}^{\vect s}(\lambda'_{k'})  \norm{u_{k'}*\psi'_{k'}}_{L^p(\Nc)}  }_{\ell^q(K'_{c'})}
	\]
	for every  $(u_{k'})\in \Sc_{\Omega,L}(\Nc)^{(K'_{c'})}$, where $\psi'_{k'}\coloneqq\Fc_\Nc^{-1}(\varphi'_k(\,\cdot\,t_{k'}'^{-1}))$ for every $k'\in K'$, while 
	\[
	K'_{c'}\coloneqq\Set{k'\in K'\colon \varphi'_{k'}(\,\cdot\, t'^{-1}_{k'})( \chi_{[1/c' , c']}\circ N)\neq 0}.
	\]
	By a simple dilation argument, this implies that
	\[
	\norm*{\sum_{k'\in K'_{c'}} u_{k'}*\psi'_{k',j}}_{L^p(\Nc)}\leqslant C' \norm*{\Delta_{\Omega'}^{\vect s}(\lambda'_{k'})  \norm{u_{k'}*\psi'_{k',j}}_{L^p(\Nc)}  }_{\ell^q(K'_{c'})}
	\]
	for every  $(u_{k'})\in \Sc_{\Omega,L}(\Nc)^{(K'_{c'})}$  and for every $j\in \Z$, where $\psi'_{k',j}= c'^{ j Q} \psi'_{k'}(c'^{j}\,\cdot\,)$.
	Now, observe that~\cite[Corollary 2.51]{CalziPeloso} implies that there is $c''\geqslant c$ such that $\Supp{\psi_k}\subseteq N^{-1}([1/c'',c''])$ for every $k\in K_c$. 
	Then, there is a finite subset $J$ of $\Z$ such that 
	\[
	[1/c'',c'']\subseteq \bigcup_{j\in J} c'^{2j}[1/c',c'].
	\]

	For every $k'\in K'$ and for every $j\in J$, define
	\[
	K_{k',j}\coloneqq \Set{k\in K\colon \psi_k*\psi'_{k',j}\neq 0} \quad \text{ and} \quad K''_{k'}\coloneqq\Set{k''\in K'\colon \psi'_{k''}*\psi'_{k'}\neq 0},
	\]
	and observe that there is $M\in \N$ such that 
	\[
	\card(J),\card(K_{k',j}), \card(K''_{k'})\leqslant M
	\]
	for every $k'\in K'$ and for every $j\in J$, and such that each $k\in K$ belongs to at most $M$ of the sets $(K_{k',j})_{k'\in K'}$, for every $j\in J$ (cf.~Section~\ref{sec:2}).
	Define 
	\[
	\widetilde \varphi'\coloneqq \sum_{j\in J}\sum_{k'\in K'_{c'}} \varphi'_{k'}(\,\cdot\, c'^{-j}t_{k'}'^{-1}),
	\]
	so that $\widetilde \varphi'$ is well-defined, of class $C^\infty$, and $\geqslant 1$ on $N^{-1}([1/c'',c''])$. Define, in addition,
	\[
	\widetilde \varphi_{k}\coloneqq \frac{\varphi_k}{\widetilde \varphi'(\,\cdot\,t_k)} \qquad \text{and} \qquad \widetilde \varphi'_{k',j}\coloneqq \frac{\varphi'_{k'}(c'^{-j}\,\cdot\,)}{\widetilde \varphi'(\,\cdot\,t'_{k'})} 
	\]
	for every $k\in K$, for every $k'\in K'$, and for every $j\in J$, so that
	\begin{align*}
		\varphi_k(\,\cdot\,t_k^{-1})&=\sum_{j\in J}\sum_{\substack{k'\in K'_{c'}\\ k\in K_{k',j}}} \widetilde \varphi_{k}(\,\cdot\, t_k^{-1}) \varphi'_{k'}(\,\cdot\,  c'^{-j}t'^{-1}_{k'}),\\
		\widetilde \varphi_{k}&=\frac{\varphi_k}{\sum_{j\in J}\sum_{k'\in K'_{c'}\colon k\in K_{k',j}} \varphi'_{k'}(\,\cdot\, c'^{-j}(t'_{k'} t_k^{-1})^{-1} )}
	\end{align*}
	and
	\[
	\widetilde \varphi'_{k',j}=\frac{\varphi'_{k'}(c'^{-j}\,\cdot\,)}{\sum_{j\in J}\sum_{k''\in K'_{c'}\cap K''_{k'}} \varphi'_{k''}(\,\cdot\, c'^{-j}(t'_{k'} t_{k''}'^{-1})^{-1} )}
	\]
	for every $k\in K$, for every $k'\in K'$, and for every $j\in J$. 
	By means of~\cite[Lemma 2.52]{CalziPeloso} and the preceding arguments, we see that the families $(\widetilde \varphi_{k})$ and $(\widetilde \varphi'_{k',j})$ are bounded in $C^\infty_c(\Omega')$ for every $j\in J$. Then, define $\widetilde \psi_{k}, \widetilde \psi'_{k',j}\in\Sc_\Omega(\Nc)$ so that 
	\[
	\Fc_\Nc \widetilde \psi_{k}=\widetilde \varphi_{k}(\,\cdot\, t_k^{-1})\qquad \text{and} \qquad\Fc_\Nc \widetilde \psi'_{k',j}=\widetilde \varphi'_{k',j}(\,\cdot\, t'^{-1}_{k'})
	\]
	for every $k\in K$ and for every $k'\in K'$.
	
	Fix $(u_k)\in \Sc_{\Omega,L}(\Nc)^{(K_c)}$, and define 
	\[
	u_{k',j}\coloneqq \sum_{k\in K_{c}\cap K_{k',j} } u_k*\widetilde \psi_{k}
	\]
	for every $k'\in K'_{c'}$ and for every $j\in J$, so that $(u_{k',j})\in \Sc_{\Omega,L}(\Nc)^{(K'_{c'})}$ and
	\[
	\sum_{k\in K_{c}} u_k*\psi_k=\sum_{j\in J}\sum_{k'\in K'_{c'}} u_{k',j}*\psi'_{k',j}.
	\]
	Now, Corollary~\ref{cor:11} and a homogeneity argument imply that there is a constant $C_1>0$ such that
	\[
	\norm{u'*\widetilde \psi_{k}*\psi'_{k',j}}_{L^p(\Nc)}=\norm{u'* \psi_{k}*\widetilde\psi'_{k',j}}_{L^p(\Nc)}\leqslant C_1 \norm{u'* \psi_k}_{L^p(\Nc)}
	\]
	for every $u'\in \Sc'_{\Omega,L}(\Nc)$, for every $k'\in K'$,   for every $k\in K$, and for every $j\in J$, so that
	\[
	\norm{u_{k',j}*\psi'_{k',j}}_{L^p(\Nc)}\leqslant C_1 M^{(1/p-1)_+} \sum_{k\in K_c\cap K_{k',j}} \norm{u_k*\psi_k}_{L^p(\Nc)}.
	\]
	In addition,~\cite[Corollary 2.49]{CalziPeloso}  shows that there is a constant $C_2>0$ such that $\Delta_{\Omega'}^{\vect s}(\lambda'_{k'})\leqslant C_2 \Delta_{\Omega'}^{\vect s}(\lambda_k)$ for every $k'\in K'$,  for every $k\in K_{k',j}$, and for every $j\in J$.
	Therefore,
	\[
	\norm*{\sum_{k\in K_c} u_k*\psi_k}_{L^p(\Nc)}\leqslant C_3  \norm*{\Delta_{\Omega'}^{\vect s}(\lambda_k)\norm{u_k*\psi_k}_{L^p(\Nc)}  }_{\ell^q(K_c)},
	\]
	where $C_3\coloneqq C' C_1 M^{(1/p-1)_+ +\max(1,1/p)+\max(1,1/q)}C_2$.
	
	\textsc{Step II.} Now, consider the general case. By~\textsc{step I}, there is $C>0$ such that
	\[
	\norm*{\sum_{\substack{k\in K_c}} u_k*\psi'_k}_{L^p(\Nc)}\leqslant C  \norm*{ \Delta_{\Omega'}^{\vect s}(\lambda_k) \norm{u_k*\psi'_k}_{L^p(\Nc)} }_{\ell^q(K_c)}
	\]
	for every $(u_k)\in \Sc_{\Omega,L}(\Nc)^{(K_c)}$, where $\psi'_k=\Fc_\Nc^{-1}(\varphi'_k(\,\cdot\,t_k)^{-1})$ and $(\varphi'_k)$ is a bounded family of positive elements of $C^\infty_c(\Omega')$ such that $\varphi'_k=1$ on $\supp(\varphi_k)+\overline{\Omega'}\cap\overline B_{F'}(0,1)$ for every $k\in K$. Then, fix a positive $\tau\in C^\infty_c(\Omega'\cap B_{F'}(0,1))$ so that $\eta\coloneqq \Fc_\Nc^{-1}(\tau)$ satisfies $\eta(0,0)=\norm{\eta}_{L^\infty(\Nc)}=1$, and define $\eta_\rho\coloneqq \eta(\rho\,\cdot\,)=\Fc_\Nc^{-1}(\rho^{-m}\tau(\rho^{-1}\,\cdot\,))$ for every $\rho>0$. Take $(u_k)\in \Sc_{\Omega,L}'(\Nc)^{(K_c)}$, and observe that
	\[
	(u_k*\psi_k)\eta_\rho=[(u_k*\psi_k)\eta_\rho]*\psi'_k\in \Sc_{\Omega,L}(\Nc)
	\] 
	for every $k\in K$, thanks to~\cite[Proposition 4.5 and Corollary 4.6]{CalziPeloso}. Then,
	\[
	\norm*{\sum_{\substack{k\in K_c}} (u_k*\psi_k)\eta_\rho}_{L^p(\Nc)}\leqslant C  \norm*{ \Delta_{\Omega'}^{\vect s}(\lambda_k) \norm{ (u_k*\psi_k)\eta_\rho}_{L^p(\Nc)} }_{\ell^q(K_c)}\leqslant C  \norm*{ \Delta_{\Omega'}^{\vect s}(\lambda_k) \norm{ u_k*\psi_k}_{L^p(\Nc)} }_{\ell^q(K_c)}.
	\]
	The assertion follows passing to the limit for $\rho\to 0^+$.
\end{proof}

\begin{prop}\label{prop:8}
	Assume that $r>1$ and take $\vect s\in \R^r$ and $p,q\in ]0,\infty]$. If property $(D')^{\vect s}_{p,q}$ holds, then $\vect s\leqslant \vect 0$.
\end{prop}

Notice that, if $r=1$, then property $(D')^{s}_{p,q}$ trivially holds for every $s\in \R$ and for every $p,q\in ]0,\infty]$.

\begin{proof}
	Fix a non-zero $\varphi\in \Sc_\Omega(\Nc)$ and $j\in \Set{1,\dots,r}$. With the notation of~\cite[\S 2.1]{CalziPeloso}, define $t_k\coloneqq (e-e_j)+e_j/(k+1)\in T_+$, so that $(e_{\Omega'}\cdot t_k)$ converges to some non-zero element of $\partial \Omega'$,  $\Delta^{\vect e_j}(t_k)\to 0$, and $\Delta^{\vect e_\ell}(t_k)\to 1$ for every $\ell=1,\dots, r$, $\ell\neq j$. For every $k\in\N$, choose $g_k\in GL(E)$ such that $t_k\cdot \Phi=\Phi\circ (g_k\times g_k)$. Define $\varphi_k\coloneqq \Delta^{-\vect s-(\vect b+\vect d)/p}(t_k) \varphi((g_k\times t_k)\,\cdot\,)$ for every $k\in\N$, so that
	\[
	\norm{\varphi_k}_{L^p(\Nc)} =\Delta^{-\vect s}(t_k) \norm{\varphi}_{L^p(\Nc)}
	\]
	for every $k\in \N$.  In addition,~\cite[Proposition 4.17]{CalziPeloso} shows that the sequence $(\varphi_k)$ is uniformly bounded in $\mathring B_{p,q}^{\vect s}(\Nc,\Omega)$. 
	Observe that $\Supp{\Fc_\Nc \varphi_k}=\Supp{\Fc_\Nc\varphi}\cdot t_k$ for every $k\in\N$ (cf.~\cite[Proposition 4.2]{CalziPeloso}), so that we may find  $c>1$ such that $\supp{\Fc_\Nc \varphi_k} \subseteq N^{-1}([1/c,c])$ for every $k\in \N$. Then,  the preceding remarks imply that the sequence $(\Delta^{-\vect s}(t_k))$ is bounded, so that $s_j\leqslant 0$.
	By the arbitrariness of $j$, this proves that $\vect s\in \R_-^r$.
\end{proof}

\begin{deff}\label{def:1}
	Take $\vect s\in \R^r$ and $p,q\in ]0,\infty]$. We say that $B_{p,q}^{\vect s}(\Nc,\Omega)$ embeds canonically into $B_{p,q}^{\sum_j s_j}(\Nc,\Lc)$ if the following hold: the canonical mapping $\Sc'(\Nc)\to \Sc'_{\Omega,L}(\Nc)$ induces a  linear  mapping 
	\[
	S\colon\Sc_{\Lc}(\Nc)\to B_{p',q'}^{-\vect s-(1/p-1)_+(\vect b+\vect d)}(\Nc,\Omega),
	\]
	and the image of the mapping
	\[
	\iota \colon B^{\vect s}_{p,q}(\Nc,\Omega)\ni u\mapsto ( \Sc_\Lc(\Nc)\ni \varphi\mapsto \left\langle u \vert S(\varphi)\right\rangle  )\in \Sc_\Lc'(\Nc)
	\]
	is contained in $B^{\sum_j s_j}_{p,q}(\Nc,\Lc)$.
\end{deff}

Notice that, with the notation of Definition~\ref{def:1}, the linear mapping $S$ is necessarily continuous  by the closed graph theorem,\footnote{Notice that $B_{p',q'}^{-\vect s-(1/p-1)_+(\vect b+\vect d)}(\Nc,\Omega) $ is a Fr\'echet space, since $p',q'\geqslant 1$.} so that $\iota\colon B^{\vect s}_{p,q}(\Nc,\Omega)\to \Sc_\Lc'(\Nc)$ is well defined and continuous. In addition, $\iota$ is one-to-one since $\Sc_\Lc(\Nc)\supseteq \Sc_{\Omega,L}(\Nc)$ (cf.~the proof of Proposition~\ref{prop:7}). Finally, $\iota\colon B^{\vect s}_{p,q}(\Nc,\Omega) \to B^{\sum_j s_j}_{p,q}(\Nc,\Lc)$  is then continuous by the closed graph theorem.

\begin{teo}\label{teo:1}
	Take $\vect s\in \R^r$ and $p,q\in ]0,\infty]$. Then, the following conditions are equivalent:
	\begin{enumerate}
		\item[\textnormal{(1)}] the canonical mapping $\Sc_{\Omega,L}(\Nc)\to \mathring B_{p,q}^{\sum_j s_j}(\Nc,\Lc)$ induces a continuous linear mapping from $\mathring B_{p,q}^{\vect s}(\Nc,\Omega)$   into $\mathring B_{p,q}^{\sum_j s_j}(\Nc,\Lc)$;
		
		\item[\textnormal{(2)}]  $B_{p,q}^{\vect s}(\Nc,\Omega)$ embeds canonically into $B_{p,q}^{\sum_j s_j}(\Nc,\Lc)$;
		
		\item[\textnormal{(3)}] property $(D')_{p,q}^{\vect s}$   holds.
	\end{enumerate}
\end{teo}

\begin{proof}
	(1) $\implies$ (3). Fix a $(\delta,R)$-lattice $(\lambda_k)_{k\in K}$ in $\Omega'$, a bounded family $(\varphi_k)_{k\in K}$  of positive elements of $C^\infty_c(\Omega')$ supported in $B(e_{\Omega'},R\delta)$ such that 
	\[
	\sum_k \varphi_k(\,\cdot\,t_k^{-1})\geqslant 1
	\]
	on $\Omega'$, where $t_k\in T_+$ is such that $\lambda_k=e_{\Omega'}\cdot t_k$ for every $k\in K$. In addition, take a constant $c>1$, and define $K_c\coloneqq \Set{k\in K\colon  \varphi_k(\,\cdot\, t_k^{-1}) (\chi_{[1/c,c]}\circ N)\neq 0}$. Thanks to~\cite[Corollary 2.51]{CalziPeloso}, it is clear that $\bigcup_{k\in K_c} B_{\Omega'}(\lambda_k, R\delta)$ is contained in a compact subset of $\overline{\Omega'}\setminus\Set{0}$, so that Proposition~\ref{prop:1} implies that there is $ \tau\in C^\infty_c(\R_+^*)$ such that $ \psi_k *\Kc(\tau)=\psi_k$ for every $k\in K_c$, where $\psi_k\coloneqq \Fc_\Nc^{-1}(\varphi_k(\,\cdot\, t_k^{-1}))$.
	Therefore, by means of Lemma~\ref{lem:2} we see that there is a constant $C_1>0$ such that
	\[
	\norm*{\sum_{k\in K_c} u_k*\psi_k}_{L^p(\Nc)}\leqslant C_1\norm*{\Delta_{\Omega'}^{\vect s}(\lambda_k) \norm*{\sum_{k'\in K_c} u_{k'}*\psi_{k'}*\psi_k}_{L^p(\Nc)}}_{\ell^q(K)}
	\] 
	for every $(u_k)\in \Sc_{\Omega,L}(\Nc)^{(K_c)}$.
	To conclude, it will suffice to prove that there is a constant $C_2>0$ such that 
	\[
	\norm*{\Delta_{\Omega'}^{\vect s}(\lambda_k) \norm*{\sum_{k'\in K_c} u_{k'}*\psi_{k'}*\psi_k}_{L^p(\Nc)}}_{\ell^q(K)}\leqslant C_2 \norm*{\Delta_{\Omega'}^{\vect s}(\lambda_k) \norm*{ u_{k}*\psi_k}_{L^p(\Nc)}}_{\ell^q(K_c)}
	\]
	for every  $(u_k)\in \Sc_{\Omega,L}(\Nc)^{(K_c)}$.
	Now, by Corollary~\ref{cor:11} and a homogeneity argument, there is a constant $C'_2>0$ such that
	\[
	\norm*{ u*\psi_{k'}*\psi_k}_{L^p(\Nc)}\leqslant C'_2 \norm{u*\psi_{k'}}
	\]
	for every $u\in \Sc'_{\Omega,L}(\Nc)$ and for every $k,k'\in K$. In addition,  there is $M\in \N$ such that the set $K'_k\coloneqq \Set{k'\in K\colon \psi_{k'}*\psi_k\neq 0}$ has at most $M$ elements for every $k\in K$ (cf.~Section~\ref{sec:2}), so that
	\[
	\norm*{\sum_{k'\in K_c} u_{k'}*\psi_{k'}*\psi_k}_{L^p(\Nc)}\leqslant M^{(1/p-1)_+} C_2' \sum_{k'\in K_c\cap K'_k} \norm{u_{k'}*\psi_{k'}}_{L^p(\Nc)}.
	\]
	Using~\cite[Corollary 2.49]{CalziPeloso} we also find a constant $C_2''>0$ such that
	\[
	\Delta_{\Omega'}^{\vect s}(\lambda_k)\leqslant C_2'' \Delta_{\Omega'}^{\vect s}(\lambda_{k'})
	\]
	for every $k\in K$ and for every $k'\in K_k$. Hence,
	\[
	\norm*{\Delta_{\Omega'}^{\vect s}(\lambda_k) \norm*{\sum_{k'\in K_c} u_{k'}*\psi_{k'}*\psi_k}_{L^p(\Nc)}}_{\ell^q(K)}\leqslant C_2 \norm*{\Delta_{\Omega'}^{\vect s}(\lambda_k) \norm*{ u_{k}*\psi_k}_{L^p(\Nc)}}_{\ell^q(K_c)}
	\]
	with $C_2\coloneqq C_2' C_2'' M^{(1/p-1)_+ +\max(1,1/q)}$.

	(3) $\implies$ (1). Define $T'_+$ as the quotient of $T_+$ by its closed central subgroup which acts by homotheties on $\Omega'$, 
	and endow $T'_+$ with the corresponding quotient Lie group structure. Fix a relatively compact  open neighbourhood $V'$ of the identity in $T'_+$, and let $(t'_k)_{k\in K}$ be a maximal family in $T'_+$ such that $(V' t'_k)\cap (V' t'_{k'})=\emptyset$ for every $k,k'\in K$, $k\neq k'$. Therefore, $H'\coloneqq \overline{V'^{-1} V'}$  is a compact subset of $T_+'$ and  $T'_+=\bigcup_{k\in K} H' t'_k$. Then, there is a bounded family $(\varphi'_k)_{k\in K}$ of positive elements of $C^\infty_c(T'_+)$ such that $\chi_{V' }\leqslant\varphi'_k\leqslant \chi_{H'}$ for every $k\in K$, and such that
	\[
	\sum_{k\in K} \varphi'_k(t' t_k'^{-1})=1
	\]
	for every $t'\in T_+'$. In addition,  fix a positive $\eta\in C^\infty_c(\R_+^*)$ such that  $\chi_{N(e_{\Omega'})[3/4,2]}\leqslant \eta\leqslant \chi_{N(e_{\Omega'})[1/2,4]}$ and such that
	\[
	\sum_{j\in \Z} \eta(2^{-2j}\,\cdot\,)=1
	\]
	on $\R_+^*$, and observe that there is a unique bounded family $(\varphi_{k,j})_{k\in K,j\in \Z}$ of positive elements of $C^\infty_c(\Omega')$ such that
	\[
	\varphi_{k,j}(\lambda)=\eta(N(\lambda)) \varphi'_k(\pi(\lambda))
	\]
	for every $\lambda\in \Omega'$, where $\pi$ is the composition of the inverse of the mapping $T_+\ni t \mapsto e_{\Omega'}\cdot t\in \Omega'$ with the canonical projection $T_+\to T_+'$.
	Then, 
	\[
	\sum_{j\in \Z} \sum_{k\in K} \varphi_{k,j}(\lambda \cdot t_{k,j}^{-1})=1
	\]
	for every $\lambda\in \Omega'$, where $t_{k,j}$ is the unique element of $T_+$ such that $N(e_{\Omega'}\cdot t_{k,j})=2^{2j}$ and whose class in $T_+'$ is $t_k$. Let us prove that $(e_{\Omega'} \cdot t_{k,j})_{k\in K,j\in \Z}$ is a $(\delta,R)$-lattice on $\Omega'$ for some $\delta>0$ and some $R>1$. Indeed, let $V$ and $H$ be the subsets of $\Omega'$ such that
	\[
	\chi_V(\lambda)=\chi_{N(e_{\Omega'})]3/4,2[}(N(\lambda)) \chi_{V'}(\pi(\lambda)) \qquad \text{and} \qquad \chi_H(\lambda)=\chi_{N(e_{\Omega'})[1/2,4]}(N(\lambda)) \chi_{H'}(\pi(\lambda))
	\]
	for every $\lambda\in \Omega'$. Then, the $V\cdot t_{k,j}$ are open and pairwise disjoint, while the $H\cdot t_{k,j}$ are compact and cover $\Omega'$. Since $  e_{\Omega'}\in V$, there are $\delta>0$ and $R>1$ such that $B(e_{\Omega'}, \delta)\subseteq V$ and $H\subseteq B(e_{\Omega'},R\delta)$, so that $(e_{\Omega'}\cdot t_{k,j})$ is a $(\delta,R)$-lattice.
	
	Now, define $\psi_{k,j}\coloneqq \Fc_\Nc^{-1}(\varphi_{k,j}(\,\cdot\,t_{k,j}^{-1}))$, and observe that
	\[
	f*\Kc(\eta(2^{-2j}\,\cdot\,))=\sum_{k\in K} f*\psi_{k,j}
	\]
	for every $f\in \Sc_{\Omega,L}(\Nc)$, since $\pi_\lambda$, applied to both sides of the asserted equality, gives rise to the same operators, for almost every $\lambda\in F'\setminus W$.
	Hence, Lemma~\ref{lem:4} and a dilation argument imply that there is a constant $C_3>0$ such that
	\[
	\norm{f*\Kc(\eta(2^{-2j}\,\cdot\,))}_{L^p(\Nc)}=\norm*{\sum_{k\in K} f*\psi_{k,j}}_{L^p(\Nc)}\leqslant C_3 \norm*{\Delta^{\vect s}(t_{k,0})\norm*{ f*\psi_{k,j}}_{L^p(\Nc)}}_{\ell^q(K)}
	\]
	for every $f\in \Sc_{\Omega,L}(\Nc)$ and for every $j\in \Z$. Hence,
	\[
	\norm*{2^{j \sum_\ell s_\ell}\norm{f*\Kc(\eta(2^{-2j}\,\cdot\,))}_{L^p(\Nc)}}_{\ell^q(\Z)}\leqslant C_3 \norm*{\Delta^{\vect s}(t_{k,j})\norm*{ f*\psi_{k,j}}_{L^p(\Nc)}}_{\ell^q(K\times \Z)}
	\]
	for every $f\in \Sc_{\Omega,L}(\Nc)$, so that  $\mathring B_{p,q}^{\vect s}(\Nc,\Omega)$ embeds continuously into $\mathring B_{p,q}^{\sum_j s_j}(\Nc,\Lc)$.	 
	
	(3) $\implies$ (2).	The assertion follows by the preceding case unless $\max(p,q)=\infty$. By transposition (cf.~Section~\ref{sec:2}), we infer that the canonical mapping $\Sc'(\Nc)\to \Sc'_{\Omega,L}(\Nc)$ induces a continuous linear mapping $S\colon \Sc_\Lc(\Nc)\to B^{-\vect s-(1/p-1)_+(\vect b+\vect d)}_{p',q'}(\Nc,\Omega)$. Hence, there is a canonical one-to-one continuous linear mapping $\iota\colon B^{\vect s}_{p,q}(\Nc,\Omega)\to \Sc_\Lc'(\Nc)$ (cf.~Definition~\ref{def:1}). The assertion will therefore follow if we prove that the image of $\iota$ is contained in $B^{\sum_j s_j}_{p,q}(\Nc,\Lc)$. Take $(t_{k,j})$, $(\varphi_{k,j})$, and $(\psi_{k,j})$ as in the proof of the implication (3) $\implies$ (1).
	Assume first that $q=\infty$, so that $S(\Sc_\Lc(\Nc))\subseteq \mathring B^{-\vect s-(1/p-1)_+(\vect b+\vect d)}_{p',q'}(\Nc,\Omega)$. Then,
	\[
	\iota(u)= \sum_{k,j} u*\psi_{k,j}
	\]
	with convergence in $\Sc_\Lc'(\Nc)$, for every $u\in B^{\vect s}_{p,q}(\Nc,\Omega)$. If, otherwise, $q<\infty$, then by inspection of the proof of~\cite[Lemma 4.22]{CalziPeloso} we see that
	\[
	u=\sum_{k,j} u*\psi_{k,j}
	\]
	in $B^{\vect s}_{p,q}(\Nc,\Omega)$, so that, again,
	\[
	\iota(u)= \sum_{k,j} u*\psi_{k,j},
	\]
	with convergence in $\Sc_\Lc'(\Nc)$, for every $u\in B^{\vect s}_{p,q}(\Nc,\Omega)$.	
	
	In particular,
	\[
	\iota(u)*\Kc(\eta(2^{-2j}\,\cdot\,))=\sum_{k',j'} u*\psi_{k',j'}*\Kc(\eta(2^{-2j}\,\cdot\,))=\sum_{k',j'} \sum_{k} u*\psi_{k',j'}*\psi_{k,j}= \sum_k u*\psi_{k,j},
	\]
	where the sums converge locally uniformly, for every $u\in B^{\vect s}_{p,q}(\Nc,\Omega)$.	Therefore,  
	\[
	\norm{\iota(u)*\Kc(\eta(2^{-2j}\,\cdot\,))}_{L^p(\Nc)}\leqslant \sup\limits_{K' } \norm*{\sum_{k\in K'} u*\psi_{k,j}}_{L^p(\Nc)}
	\]
	for  every $u\in B^{\vect s}_{p,q}(\Nc,\Omega)$, where $K'$ runs through the set of finite subsets of $K$. We may then proceed as in the proof of the implication (3) $\implies$ (1).
	
	(2) $\implies$ (1). Obvious.
\end{proof}

We now recall  (an equivalent formulation of) the definition of property $(D)^{\vect s}_{p,q}$, introduced in~\cite[Definition 5.5]{CalziPeloso}.\footnote{This property has been considered in a redundant way in~\cite{CalziPeloso}.} 
We say that property $(D)_{p,q}^{\vect s}$ holds if there are a $(\delta, R)$-lattice $(\lambda_k)_{k\in K}$ on $\Omega'$, for some $\delta>0$ and some $R>1$, a bounded family $(\varphi_k)_{k\in K}$ of elements of $C^\infty_c(\Omega')$ such that
\[
\sum_{k\in K} \varphi_k(\,\cdot\, t_k^{-1})\geqslant 1
\]
on $\Omega'$, where $t_k\in T_+$ and $\lambda_k=e_{\Omega'}\cdot t_k$ for every $k\in K$, and two constants $C>0$ and $c>1$ such that
\[
\norm*{\sum_{\substack{k\in K}} u_k*\psi_k}_{L^p(\Nc)}\leqslant C  \norm*{ \Delta_{\Omega'}^{\vect s}(\lambda_k) \ee^{c\left\langle \lambda_k, e_\Omega\right\rangle}\norm{u_k*\psi_k}_{L^p(\Nc)} }_{\ell^q(K)}
\]
for every $(u_k)\in \Sc_{\Omega,L}(\Nc)^{(K)}$, where $\psi_k\coloneqq\Fc_\Nc^{-1}(\varphi_k(\,\cdot\, t_k^{-1}))$ for every $k\in K$. As in Lemma~\ref{lem:4}, one then sees that the same holds for $(u_k)\in \Sc_{\Omega,L}'(\Nc)^{(K)}$.

\begin{prop}\label{prop:15}
	Take $\vect s\in \R^r$ and $p,q\in ]0,\infty]$. Assume that either $\sum_j s_j<0$, or  $\sum_j s_j=0$ and $q<\min(1,p)$.
	Then, the following conditions are equivalent:
	\begin{enumerate}
		\item[\textnormal{(1)}] property  $(D)^{\vect s}_{p,q}$ holds;
		
		\item[\textnormal{(2)}] property $(D')^{\vect s}_{p,q}$.
	\end{enumerate}
\end{prop}

This extends~\cite[Proposition 4.16]{BekolleBonamiGarrigosRicci}, which deals with the case in which $n=0$, $p,q\in [1,\infty[$, $\vect s\in \R \vect d$, and $\Omega$ is irreducible and symmetric.

\begin{proof}
	(1) $\implies$ (2). Obvious.
	
	(2) $\implies$ (1). Take $(t_{k,j})_{k\in K,j\in \Z}$, $(\varphi_{k,j})_{k\in K,j\in \Z}$, and $(\psi_{k,j})_{k\in K,j\in \Z}$ as in the proof of Theorem~\ref{teo:1}, and observe that, by Lemma~\ref{lem:4} and a dilation argument, there is a constant $C_1>0$ such that, for every $(u_{k,j})\in \Sc_{\Omega,L}(\Nc)^{(K\times \Z)}$,
	\[
	\norm*{\sum_{k\in K}\sum_{j\in \Z} u_{k,j}*\psi_{k,j} }_{L^p(\Nc)}^{\min(1,p)}\leqslant \sum_{j\in \Z} \norm*{\sum_{k\in K}u_{k,j}*\psi_{k,j} }_{L^p(\Nc)}^{\min(1,p)}\leqslant C_1\sum_{j\in \Z} \norm*{ \Delta^{\vect s}(t_{k,0}) \norm{u_{k,j}*\psi_{k,j}}_{L^p(\Nc)} }_{\ell^q(K)}^{\min(1,p)}.
	\]
	Therefore, by H\"older's inequality,
	\[
	\norm*{\sum_{k\in K}\sum_{j\in \Z} u_{k,j}*\psi_{k,j} }_{L^p(\Nc)}^{\min(1,p)}\leqslant C_1 C_2 \norm*{ \Delta^{\vect s}(t_{k,j}) \ee^{\left\langle e_{\Omega'}\cdot t_{k,j},e_\Omega\right\rangle} \norm{u_{k,j}*\psi_{k,j}}_{L^p(\Nc)} }_{\ell^q(K\times \Z)}^{\min(1,p)},
	\]
	where\footnote{Notice that $\inf\limits_{k\in K}\left\langle e_{\Omega'}\cdot t_{k,0},e_\Omega\right\rangle>0 $ since the linear functional $\left\langle\,\cdot\,,e_\Omega\right\rangle$  vanishes nowhere on the compact set $\Set{\lambda\in \Omega'\colon N(\lambda)=1}$ because $\Omega=(\Omega')'$.}
	\[
	C_2\coloneqq \norm*{\sup\limits_{k\in K}\ee^{-2^{ j}\left\langle e_{\Omega'}\cdot t_{k,0},e_\Omega\right\rangle} 2^{- j \sum_\ell s_\ell}   }_{\ell^{\min(1,p)/(1-\min(1,p)/q)_+}(\Z)}<\infty.
	\]
	The proof is complete.
\end{proof}


\begin{thebibliography}{9}
	\bibitem{BekolleBonamiGarrigosRicci}
	B\'ekoll\'e, D., Bonami, A., Garrig\'os, G., Ricci, F., \emph{Littlewood--Paley Decompositions Related to Symmetric Cones and Bergman Projections in Tube Domains}, {P.\ Lond.\ Math.\ Soc.} {89} (2004), p.~317--360.
	
	\bibitem{BekolleGonessaNana}
	B\'ekoll\'e, D., Gonessa, J., Nana, C., Lebesgue Mixed Norm Estimates for Bergman Projectors: from Tube Domains over Homogeneous Cones to
	Homogeneous Siegel Domains of Type II, \emph{Math.\ Ann.} \textbf{374} (2019), p.~395--427.
		
	\bibitem{BerghLofstrom}
	Bergh, J., L\"ofstr\"om, J., \emph{Interpolation Spaces}, Springer-Verlag, 1976.
	
	\bibitem{BourbakiTVS}
	Bourbaki, N., \emph{Topological Vector Spaces}, Springer, 2003.
	
	\bibitem{BourgainDemeter}
	Bourgain, J., Demeter, C., The Proof of the $\ell^2$-Decoupling Conjecture, \emph{Ann.\ Math.} \textbf{182} (2015), p.~351--389.
	
	\bibitem{Calzi}
	Calzi, M., \emph{Functional Calculus on Homogeneous Groups}, Ph.D.\ thesis.
	
	\bibitem{Calzi2}
	Calzi, M., Spectral Multipliers on 2-Step Stratified Groups, I, \emph{J.\ Fourier Anal.\ Appl.} \textbf{26}, 35 (2020), https://doi.org/10.1007/s00041-020-09740-y.
	
	\bibitem{CalziPeloso}
	Calzi, M., Peloso, M.\ M., Holomorphic Function Spaces on Homogeneous Siegel Domains, to appear in \emph{Diss.\ Math.}
	
	\bibitem{CalziPeloso2}
	Calzi, M., Peloso, M.\ M., Paley--Wiener Spaces on the Boundaries of Siegel Domains, in preparation.
	
	\bibitem{CardonaRuzhansky}
	Cardona, D., Rhuzansky, M., Multipliers  for  Besov  Spaces  on  Graded  Lie  Groups, \emph{C.\ R.\ Acad.\ Sci.\ Paris, Ser.\ I} \textbf{355} (2017), p.~400--405.
	
	\bibitem{FarautKoranyi}
	Faraut, J., Kor\'anyi, A., \emph{Analysis on Symmetric Cones}, Clarendon Press, 1994.
	
	\bibitem{FischerRuzhansky}
	Fischer, V., Ruzhansky, M., Sobolev Spaces on Graded Lie Groups, \emph{Ann.\ Inst.\ Fourier, Grenoble} \textbf{67} (2017), p.~1671--1723.
	
	\bibitem{Folland}
	Folland, G.\ B., \emph{Harmonic Analysis in Phase Space}, Princeton Univ.\ Press, 1989.
	
	\bibitem{GKKP}
	Georgiadis, A.\ G., Kerkyacharian, G., Kyriazis, G., Petrushev, P., Homogeneous Besov and Triebel--Lizorkin Spaces Associated 	to Non-Negative Self-Adjoint Operators, \emph{J.\ Math.\ Anal.\ Appl.} \textbf{449} (2017), p.~1382--1412.
	
	
	\bibitem{KP}
	Kerkyacharin, G., Petrushev, P., Heat Kernel Based Decomposition of Spaces of Distributions in the Framework of Dirichlet Spaces, \emph{Trans.\ Amer.\ Math.\ Soc.} \textbf{367} (2015), p.~121--189.
	
	\bibitem{LabaWolff}
	\L{}aba, I., Wolff, T., A Local Smoothing Estimate in Higher Dimensions, \emph{J.\ Anal.\ Math.} \textbf{88}	(2002), p.~ 149--171.
	
	\bibitem{GSS}
	Garrig\'os, G., Schlag, W., Seeger, A., Improvements in Wolff's Inequality for Decompositions of Cone Multipliers, \emph{preprint available online}, 2008
	
	\bibitem{Geller}
	Geller, D., Liouville's Theorem for Homogeneous Groups, \emph{Comm.\ in Partial Differential Equations} \textbf{8} (1983), p.~1665--1677.
	
	\bibitem{Martini}
	Martini, A., Spectral theory for commutative algebras of differential 	operators on Lie groups, \emph{J.\ Funct.\ Anal.} \textbf{260} (2011), p.~2767--2814.
	
	\bibitem{Murakami}
	Murakami, S., \emph{On Automorphisms of Siegel Domains}, Springer-Verlag, 1972. 
	
	\bibitem{RicciTaibleson}
	Ricci, F., Taibleson, M., Boundary Values of Harmonic Functions in Mixed Norm Spaces 	and their Atomic Structure, \emph{Ann.\ Scuola Norm.\ Sup.\ Pisa Cl.\ Sci.} \textbf{10} (1983), p.~1--54.
	
	\bibitem{Triebel2}
	Triebel, H., \emph{Interpolation Theory, Function Spaces, Differential Operators}, North-Holland Publishing Company, 1978.
	
	\bibitem{Triebel}
	Triebel, H., \emph{Theory of Funtion Spaces}, Birkh\"auser, 1983.
	
	\bibitem{Vinberg2}
	Vinberg, E.\ B., The Morozov-Borel Theorem for Real Lie Groups, \emph{Dokl.\ Akad.\ Nauk SSSR} \textbf{141} (1961), p.~270--273.
	
	\bibitem{Vinberg}
	Vinberg, E.\ B., The theory of convex homogeneous cones, \emph{Trans.\ Moscow Math.\ Soc.} \textbf{12}	(1965), p.~340--403.
	
	\bibitem{Wolff}
	Wolff, T., Local Smoothing Type Estimates on $L^p$ for Large $p$, \emph{Geom.\ Funct.\ Anal.} \textbf{10}	(2000), p.~1237--1288.
	
\end{thebibliography}
\end{document}